\documentclass[a4paper]{amsart}
\usepackage{graphicx}
\usepackage{etex}
\usepackage{amsthm,stmaryrd}
\usepackage{amssymb}
\usepackage{array}
\usepackage{tabu}
\usepackage{epsfig}
\usepackage[usenames,dvipsnames]{color}
\usepackage{verbatim}
\usepackage{hyperref}
\usepackage{mathrsfs}
\usepackage[all]{xy}
\usepackage{xcolor}
\usepackage{enumerate}
\usepackage{tikz}   
\usepackage{changepage} 
\usetikzlibrary{arrows,calc,positioning,decorations.pathreplacing}

\usepackage{pst-node}
\usepackage{tikz-cd}

\newcommand{\Aut}{\operatorname{Aut}}

\newcommand{\Hom}{\operatorname{Hom}}

\newcommand{\Z}{\mathbb{Z}}
\newcommand{\Q}{\mathbb{Q}}

\newcommand{\C}{\mathbb{C}}
\newcommand{\N}{\mathbb{N}}

\newcommand{\Li}{\mathbb{L}}

\newcommand{\F}{\mathbb{F}}
\newcommand{\DD}{\mathbb{D}}
\newcommand{\U}{\mathscr{U}}

\newcommand{\HH}{{\ensuremath{\mathcal{H}}}}

\newcommand{\unit}{\ensuremath{\mathbb{I}}}

\newcommand{\tcoev}{\stackrel{\longleftarrow}{\operatorname{coev}}}
\newcommand{\tCoev}{\stackrel{\longleftarrow}{\operatorname{Coev}}}

\newcommand{\tev}{\stackrel{\longleftarrow}{\operatorname{ev}}}
\newcommand{\ev}{\stackrel{\longrightarrow}{\operatorname{ev}}}
\newcommand{\Ev}{\stackrel{\longrightarrow}{\operatorname{Ev}}}

\newcommand{\qs}{q}
\newcommand{\q}{\textbf{q}}
\newcommand{\coev}{\stackrel{\longrightarrow}{\operatorname{coev}}}

\newcommand{\D}{{\mathcal{D}}}

\newcommand{\phs}{{{\varphi}^{\hat{W}}_{n,m}}}
\newcommand{\phh}{{\varphi}^{\hat{W}^{ \textbf{q},\lambda}}_{n,m}}
\newcommand{\pv}{{\varphi}^{\hat{W}^{N-1}}_{n,m}}
\newcommand{\ph}{{\varphi}^{W^N}_{n,m}}

\newtheorem{definition}{Definition}[subsection]
\newtheorem{theorem}[definition]{Theorem}
\newtheorem*{theorem*}{Theorem}
\newtheorem{Comment}[definition]{Comment}

\newtheorem{proposition}[definition]{Proposition}
\newtheorem{lemma}[definition]{Lemma}
\newtheorem{remark}[definition]{Remark}
\newtheorem{notation}[definition]{Notation}
\newtheorem{problem}[definition]{Problem}

\newtheorem{corollary}[definition]{Corollary}
\newcounter{exo} \newcounter{numexercice}
\renewcommand{\theexo}{\arabic{exo}}

\newcounter{IntroCounter}
\stepcounter{IntroCounter}

\begin{document}
\title[A topological model for the coloured Jones polynomials]{A topological model for the coloured Jones polynomials} 
  \author{Cristina Ana-Maria Anghel}
\address{Mathematical Institute, University of Oxford, Oxford, United Kingdom} \email{palmeranghel@maths.ox.ac.uk;cristina.anghel@imj-prg.fr} 
  \thanks{ }
  \date{\today}

\begin{abstract}
In this paper we will present a homological model for Coloured Jones Polynomials. For each colour $N \in \N$, we will describe the invariant $J_N(L,q)$ as a graded intersection pairing of certain homology classes in a covering of the configuration space on the punctured disk. This construction is based on the Lawrence representation and a result due to Kohno that relates quantum representations and homological representations of the braid groups.  
\end{abstract}

\maketitle
\setcounter{tocdepth}{1}
 \tableofcontents
 {
\section{Introduction} 

The theory of quantum invariants of knots started in 1984 with the discovery of the Jones polynomial. Later on, in 1989, Reshetikhin and Turaev constructed a general method that having as input any ribbon category leads to coloured link invariants. This method uses techniques which are purely algebraic and combinatorial. The coloured Jones polynomials $J_N(L,q)$ are a family of quantum link invariants, constructed in this manner from the representation theory $\{ V_N |N \in \mathbb N\}$ of the quantum group $U_q(sl(2))$. More precisely, for an oriented knot $L$, the $N^{th}$ coloured Jones polynomial $J_N(L,q)$ is defined algebraically through a Reshetikhin-Turaev type construction:
 $$\left( U_q(sl(2)), V_N \right) \ \ \rightarrow \ \ J_N(L,q)  \ \in \  \Z[q^{\pm 1}].$$ 

In this paper, we aim to create a bridge between the representation theory that is behind this description and topology. We give a topological interpretation for the coloured Jones polynomials, showing that they are graded intersection pairings between homology classes in coverings of configuration spaces in the punctured disc. 
\subsection*{Jones polynomial} 

The first invariant from this sequence, corresponding to $N=2$, is the original Jones polynomial $J(L,q)$. This invariant has many different flavours, it is a quantum invariant, but it also has a different characterisation, namely it can be described by skein relations. However, its relation with the topology of the knot complement is a mysterious question. On the topological side,  in 1993, R. Lawrence constructed a sequence of representations of the braid group $\{ \HH_{n,m} \}_{n,m \in \mathbb N}$ using the homology of a certain covering of the configuration space on the punctured disc.  Later on, Bigelow \cite{Big} and Lawrence \cite{Law} constructed a homological model for the original Jones polynomial $J(L,q)$, describing it as a graded intersection pairing between homology classes in a covering of a configuration space on the punctured disk. They used the Lawrence representation and the skein nature of the invariant for the proof.  

\subsection*{Coloured Jones polynomials}
Related to the question concerning the knot complement, it was shown that the the coloured Jones polynomials, evaluated at roots of unity, recover the sequence of Kashaev invariants. An important conjecture in low dimensional topology, the Volume Conjecture formulated by Kashaev \cite{K} and its generalisation due to Murakami and Murakami \cite{MM}, predicts that the limit of these quantum invariants when the colour $N$ goes to infinity recovers topological information, namely the hyperbolic/ simplicial volume of the complement of the knot.   

We give a topological model for all coloured Jones polynomials. Unlike the original case, they cannot be described directly by skein relations. Our strategy is to use their definition as quantum invariants, to study more deeply the Reshetikhin-Turaev functor and to construct step by step homological counterparts.

\subsection*{Description of the topological tools} In this part we present a brief description of the homology groups that occur in the topological model. The main ingredients in our construction are the following:
\begin{enumerate}
\item the sequence of Lawrence representations $\HH_{n,m}$ ( definition\ref{1} )
\item a sequence of dual Lawrence representation $\HH^{\partial}_{n,m}$ ( definition \ref{22} )\\ 
(defined using the homology relative to the boundary of the same covering) 
\item certain topological intersection pairings $< , >$ between the Lawrence representations and their dual representations ( definition \ref{333} ).
\end{enumerate}
 \
Let us fix $n,m \in \N$ and consider $C_{n,m}$ to be the unordered configuration space of $m$ points in the $n$-punctured disc. Then, we define a $\Z \oplus \Z$ covering space of this configuration space, given by a certain surjective morphism $$\varphi: \pi_1(C_{n,m})\rightarrow\Z \oplus \Z.$$
Consider the covering space $\tilde{C}_{n,m}$ of $C_{n,m}$ associated to this morphism. Then, the deck transformations of this covering space are $$\Z \oplus \Z=<x>_{\Z} \oplus <d>_{\Z}.$$ We consider the Borel-Moore homology of this covering space $\HH^{lf}_{m}(\tilde{C}_{n,m}, \Z).$ Using the deck transformation structure, this becomes a $\Z[x^{\pm},d^{\pm}]$-module. On the other hand, since the braid group $B_n$ is the mapping class group of the punctured disc, this induces an action of the braid group on this homology group, which is compatible with the one coming from deck transformations:
$$ \ \ \ \ \ \ \ \ \ \ \ \ \ \ MCG \ \ \ \ \ \ \ \ \ \ \ \ \ \ \ \ \ \ \ \ \ \ \ \ \ \  \ \ \ \ \ Deck \ \ \ \ \ \ \ \ \ \ \ \ \ \ \ \ \ \ \ \ $$
$$B_n\curvearrowright H^{lf}_m(\tilde{C}_{n,m},\Z)=\Z[x^{\pm},d^{\pm}]-module.$$
Related to this construction, in 2012, Kohno \cite{Koh} showed that 
the braid group action on certain subspaces of this homology group correspond to the quantum representation on the highest weight spaces in the $U_q(sl(2))$-Verma modules. In the sequel we introduce briefly these subspaces and we discuss their precise definition in Section \ref{3}.
\subsection*{I. First homology group-Lawrence representation}
\

 We describe certain subspaces inside these homology groups, generated by classes of Lagrangian submanifolds. In the sequel we use the following indexing sets, parametrised by partitions:
$$ E_{n,m}=\{e=(e_1,...,e_{n-1})\in \N^n \mid e_1+...+e_{n-1}=m \}.$$
For each partition $e$, one associates an $m$-dimensional disc $\F_{e}$ in the base space $C_{n,m}$, which can be lifted to a submanifold $\tilde{\F}_e$ in the covering space $\tilde{C}_{n,m}$. Then, we consider the subspace generated by the classes given by these submanifolds, as presented in definition \ref{multif}:
\begin{equation}\label{1}
\HH_{n,m}:= <[\tilde{\F}_e] \ | \  e\in E_{n,m}>_{\Z[x^{\pm},d^{\pm}]} \subseteq H^{\text{lf}}_m(\tilde{C}_{n,m}, \Z).
\end{equation}
\subsection*{ II. The dual Lawrence representation}
We consider a dual space, defined using the homology of the covering space relative to its boundary, which is generated by barcodes and also prescribed by partitions, as in definition \ref{barc}:
\begin{equation}\label{22}
\HH^{\partial}_{n,m}:=<[\tilde{\DD}_f] \ | \  f\in E_{n,m}>_{\Z[x^{\pm}, d^{\pm}]}\subseteq H_m(\tilde{C}_{n,m}, \Z; \partial).
\end{equation}
\subsection*{ III. Topological pairing}
These homology groups are related by a topological intersection pairing, which is a certain type of Poincar\'e-Lefschetz duality.  More specifically, there is a non-degenerate sesquilinear intersection form:
\begin{equation}\label{333}
< , >: \HH_{n,m} \otimes \HH^{\partial}_{n,m}\rightarrow \Z[x^{\pm},d^{\pm}].
\end{equation}
The advantage of this pairing consists in the fact that even if a priori it is defined using the homology of the covering space, it can be computed actually using diagrams of curves in the configuration space $C_{n,m}$ and the local system $\varphi$.

\subsection*{The specific model} Let $N \in \N$ be the colour of the coloured Jones invariant that we want to study. For the topological model, we use the Lawrence representation and its dual, with certain parameters which depend on the number of the strands of the braid representative and the colour of the invariant. 

In the following picture we have drawn the covering space $\tilde{C}_{n,m}$ that we work with as well as the submanifolds whose homology classes lead to the generators for the two homology groups, which correspond to the case: 
$$ \left(n\rightarrow {\bf 2n};  \ \ \ \ \ m \rightarrow {\bf n(N-1)} \right).
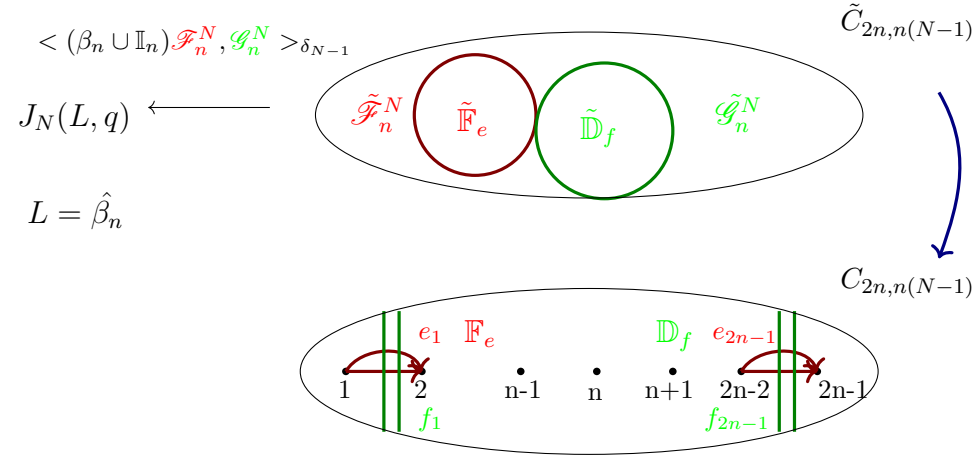
$$
\begin{center}
\begin{tikzpicture}

\foreach \x/\y in {-1.2/2, -0.2/2 , 1.1/2 , 4/2 , 2.1/2 , 3.1/2 , 5/2 } {\node at (\x,\y) [circle,fill,inner sep=1pt] {};}
[x=0.9mm,y=0.9mm]
\node at (-1,2) [anchor=north east] {$1$};
\node at (0,2) [anchor=north east] {$2$};
\node at (1.5,2) [anchor=north east] {n-1};
\node at (2.3,1.9) [anchor=north east] {n};
\node at (3.5,2) [anchor=north east] {n+1};
\node at (4.5,2) [anchor=north east] {2n-2};
\node at (5.8,2) [anchor=north east] {2n-1};
\node at (0.2,2.7) [anchor=north east] {$\color{red}e_1$};
\node at (4.6,2.7) [anchor=north east] {$\color{red}e_{2n-1}$};
\node at (0.2,1.65) [anchor=north east] {\color{green}${f_1}$};
\node at (4.5,1.65) [anchor=north east] {$\color{green}{f_{2n-1}}$};
\node at (0.9,2.8) [anchor=north east] {\Large{\color{red}$\F_e$}};
\node at (3.5,2.8) [anchor=north east] {\Large{\color{green}$\DD_f$}};
\node at (0.8,5.7) [anchor=north east] {\Large{\color{red}$\tilde{\F}_e$}};
\node at (2.5,5.6) [anchor=north east] {\Large{\color{green}$\tilde{\DD}_f$}};
\node at (-0.3,5.8) [anchor=north east] {\Large{\color{red}$\tilde{\mathscr{F}}^N_n$}};
\node at (4.4,5.8) [anchor=north east] {\Large{\color{green}$\tilde{\mathscr{G}}^N_n$}};
\node at (-3.9,5.7) [anchor=north east] {\Large{$J_N(L,q)$}};
\node at (-4,4.5) [anchor=north east] {\Large{$L=\hat{\beta_n}$}};

\node at (7.2,3.5) [anchor=north east] {\large{$C_{2n,n(N-1)}$}};
\node at (7.2,7) [anchor=north east] {\large{$\tilde{C}_{2n,n(N-1)}$}};

\node at (-1,6.7) [anchor=north east] {$ <(\beta_n \cup \mathbb I_{n}) {\color{red}\mathscr{F}^N_n }, {\color{green}\mathscr{G}^N_n} >_{\delta_{N-1}}$};
\draw[->,color=black]             (-2.2,5.5) to node[yshift=-3mm,font=\small]{} (-3.8,5.5);
\draw [very thick,black!50!red,->](-1.2,2) to[out=60,in=110] (-0.2,2);
\draw [very thick,black!50!red,->](-1.2,2) to (-0.2,2);
\draw [very thick,black!50!red,->](4,2) to[out=60,in=110] (5,2);
\draw [very thick,black!50!red,->](4,2) to (5,2);
\draw [very thick,black!50!green,-](-0.7,1.2) to (-0.7,2.8);
\draw [very thick,black!50!green,-](-0.5,1.2) to (-0.5,2.8);

\draw [very thick,black!50!blue,->](6.6,5.7) to[out=-60,in=70] (6.6,3.5);
 \draw[very thick,black!50!red] (0.5, 5.4) circle (0.8);
\draw [very thick,black!50!green,-](4.7,1.2) to (4.7,2.8);
\draw [very thick,black!50!green,-](4.5,1.2) to (4.5,2.8);
 \draw[very thick,black!50!green] (2.2, 5.19) circle (0.9);

\
\draw (2,2) ellipse (3.8cm and 1.1cm);
\draw (2,5.4) ellipse (3.6cm and 1.1cm);

\end{tikzpicture}
\end{center} 
The intersection pairing takes values in the ring of Laurent polynomials over two variables. For our purpose, we will change these coefficients using certain specialisations, given in \ref{coeff}. We have sketched below a diagram with some of the functions involved:
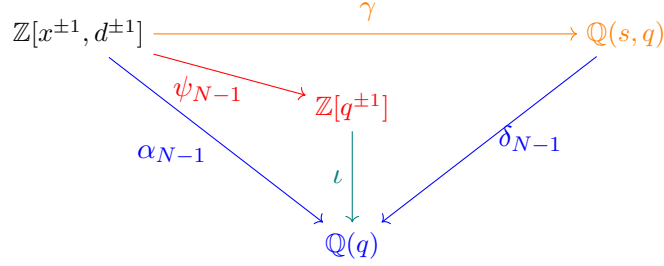
\begin{center}
\begin{tikzpicture}
[x=1.2mm,y=1.4mm]

\node (b1)               at (0,15)    {$\mathbb Z[x^{\pm1},d^{\pm1}]$};
\node (b3) [color=red]   at (30,8)   {$\mathbb Z[q^{\pm1}]$};
\node (b4) [color=blue]  at (30,-5)   {$\mathbb Q(q)$};
\node (b5) [color=orange]  at (60,15)   {$\mathbb Q(s,q)$};

\draw[->,color=teal]             (b3)      to node[left,font=\large]{$\iota$}   (b4);
\draw[->,color=red]             (b1)      to node[left,xshift=3mm,yshift=-2mm,font=\large]{$\psi_{N-1}$}   (b3);

\draw[->,color=blue]             (b1)      to node[left,yshift=-2mm,font=\large]{$\alpha_{N-1}$}   (b4);
\draw[->,color=orange]   (b1)      to node[right,xshift=-2mm,yshift=3mm,font=\large]{$\gamma$}                        (b5);
\draw[->,color=blue]   (b5)      to node[right,font=\large]{$\delta_{N-1}$}                        (b4);

\end{tikzpicture}
\end{center}
\noindent
The main result is the following:
 \begin{theorem}\label{T:geom} ({\bf Topological model with specialised homology classes}) \label{T:model}\\
Let $n \in \mathbb{N}$. Then, for any colour $N \in \mathbb N$, there exist two homology classes
 $$ \tilde{\mathscr F}_n^N \in H_{2n, n(N-1)}|_{\alpha_{N-1}} \ \ \text{and} \ \ \  \tilde{\mathscr G}_n^N \in H^{\partial}_{2n, n(N-1)}|_{\alpha_{N-1}}$$ 
 such that for any oriented knot $L$ for which there exists a braid $\beta_{n} \in B_{n}$ with $L=\hat{\beta}_{n} $ (braid closure), the $N^{th}$coloured Jones polynomial of $L$ has the formula:
$$J_N(L,q)= \frac{1}{[N]_q}q^{-(N-1)w(\beta_n)}< (\beta_n \cup \mathbb I_{n}) \tilde{\mathscr{F}}_n^N , \tilde{\mathscr{G}}_n^N>|_{\alpha_{N-1}}.$$ 
Here, $\mathbb I_{n}$ is the trivial braid with $n$ strands.

\end{theorem}

\subsection*{ Strategy}
Our strategy is to start with an oriented knot $L$ and consider a braid $\beta_{n} \in B_{n}$ such that $L=\hat{\beta}_{n}$. We study the Reshetikhin-Turaev functor applied to the corresponding diagram at three main levels: the cups, the braid level $\beta_{n}\cup \mathbb I_n$ and the caps. 
\subsection*{ Step \ref{StepI}}The first step deals with the part of the diagram concerning the cups. Following the properties coming from representation theory, we notice that they lead to an analog of a so-called highest weight space, inside the tensor power of the finite dimensional representation and its dual $V_N^{\otimes n}\otimes (V^{*}_N)^{\otimes n}$. 
\subsection*{ Step \ref{StepII}}Then, the idea of the next part consists in adding an additional normalising function on the part corresponding to the last $n$ strands, such that by composing with the identity on the first $n$ strands and with this function on the last $n$, we arrive in a certain highest weight space $W^{N}_{2n,n(N-1)}$ inside $(V_N)^{\otimes 2n}$.
\subsection*{ Step \ref{StepIII}}Then, in the next part, we show that we can see the invariant through the Reshetikhin-Turaev procedure, with this additional normalising function added at the top of the cups and at the bottom of the caps. In other words, we conclude that the coloured Jones polynomial can be seen through this particular highest weight space inside the finite dimensional module. 
\subsection*{ Step \ref{StepIV}} 

This part has to deal with the subtlety concerning highest weight spaces. More specifically, in 2012, Kohno (\cite{Koh},\cite{Ito2}) proved that 
the highest weight spaces of the Verma module for $U_q(sl(2))$ are isomorphic as braid group representations with certain specialisations of the homological Lawrence representations. However, we are not yet in these bigger highest weight spaces. In \cite{Ito3}, it was mentioned that the highest weight spaces corresponding to finite dimensional $U_q(sl(2))$-representations do not yet have a homological correspondent and this is one of the reasons why there were not known topological models for coloured Jones polynomials. 

Our approach is to include highest weight spaces from the finite dimensional module inside the ones from the Verma module, and use those instead. 
\subsection*{ Step \ref{StepV}} In this step, we conclude the we can see the coloured Jones polynomials through these bigger highest weight spaces.
\subsection*{ Step \ref{StepVI}}Then, following the normalised coevaluation and the Kohno's isomorphism, we construct a homology class $\tilde{\mathscr {F}}^N_n \in \HH_{2n,n(N-1)}|_{\alpha_{N-1}}$ which encodes homologically the cups from the diagram. 
\subsection*{ Step \ref{StepVII}} This part concerns the caps from the diagram. Here the subtlety is that the intersection pairing between Lawrence representation and its dual is defined initially over a ring. Following the discussion from Section \ref{4}, we change the coefficients to their field of fractions and use the non-degeneracy of this pairing over this field. Then, in a dual manner, we describe an element in the dual Lawrence representation $\tilde{\mathscr {G}}^N_n \in \HH_{2n,n(N-1)}|_{\alpha_{N-1}}$ which encodes homologically the caps from the diagram. The braid part from the picture corresponds to the action on the Lawrence representation.
\subsection*{ Step \ref{StepVIII}}

Finally we show that the construction of the invariant $J_N(L,q)$ through the Reshetikhin-Turaev functor on the quantum side will correspond to the topological intersection form between the two classes from the Lawrence representation and its dual, with specialised coefficients.

We would like to mention that the homology classes $\tilde{\mathscr {F}}_n^N$ and $\tilde{\mathscr {G}}_n^N$ are intrinsic in the sense that they do not depend on the link. The only part where the link plays a role in this model is encoded in the action of the corresponding braid onto the homological representation. We would like to stress the fact that the specialisation of the coefficients depends on the choice of the colour $N$.    
 
\subsection*{Loop expansion of the coloured Jones polynomials} In \cite{Ito3}, Ito gave a homological formula for the loop expansion of the coloured Jones polynomials, as an infinite sum of traces of specialised Lawrence representations. Using this, he showed a homological proof for the Melvin-Merton-Rozansky conjecture, which predicts that the $0-$loop part in the loop expansion of $J_N(L,q)$ recovers the inverse of the Alexander invariant of the knot.

The topological model from Theorem \ref{T:model}, would also lead to a topological model for the loop expansion of coloured Jones polynomials, by specialising the variables $x$ and $d$ in a specific manner. It would be interesting to investigate the precise formula and whether this provides certain topological information concerning the coefficients of this loop expansion.

\subsection*{ Colour that goes to infinity}

Pursuing this line, we are interested in the parts from this model which depend on the colour. In the formula for $J_N(L,q)$ from Theorem \ref{T:model}, the parameter $N$ appears a priori in two places:
\begin{itemize}
\item The number of points in the configuration spaces: $Conf_{n(N-1)}(\mathbb D_{2n})$.
\item The specialisation of the coefficients $\alpha_{N-1}$.
\end{itemize}  
 
The advantage of the study described in Section \ref{s} consists in the property that it shows that the homology classes which lead to the coloured invariants, can be lifted to the homology groups before the specialisations with respect to the colour as follows:
\begin{equation}
\begin{cases}
\tilde{\mathscr F_n^N} \in H_{2n, n(N-1)}|_{\alpha_{N-1}}\dashrightarrow\mathscr F_n^N \in H_{2n, n(N-1)}|_{\gamma}\\
\tilde{\mathscr G_n^N} \in H_{2n, n(N-1)}|_{\alpha_{N-1}}\dashrightarrow\mathscr G_n^N \in H_{2n, n(N-1)}|_{\gamma}
\end{cases}
\end{equation}
(here $\gamma$ is just a morphism which enlarges the ring $\Z[x^{\pm},d^{\pm}]$ to the field $\Q(q,s)$). 
 
 \
 
\begin{theorem}({ \bf Topological model for coloured Jones polynomials with globalised homology classes})\label{T:geom}

For $n,N \in \N$ there exist homology classes
 $$ \mathscr F_n^N \in H_{2n, n(N-1)}|_{\gamma} \ \ \text{and} \ \ \  \mathscr G_n^N \in H^{\partial}_{2n, n(N-1)}|_{\gamma}$$ 
 such that if a knot $L=\hat{\beta}_{n} $ with $\beta_n \in B_n$, the $N^{th}$coloured Jones polynomial has the formula:
\begin{equation}
J_N(L,q)= \frac{1}{[N]_q}q^{-(N-1)w(\beta_n)}< (\beta_{n}\cup \unit_n) \mathscr{F}_n^N , \mathscr{G}_n^N>|_{\delta_{N-1}}. 
\end{equation}
\end{theorem}
 
In other words, in order to study the behaviour of these invariants when the colour goes to infinity, one should understand the phenomena that occur when we add $n$ points at a time in the configuration space in the punctured disc.
 \

\subsection*{Comparison to the Lawrence-Bigelow model for the Jones polynomial}

This paragraph concerns the relation to the model given by Bigelow-Lawrence for the Jones polynomial. For an oriented knot which can be seen as a closure of an $n$-braid, they described the Jones polynomial as a pairing between two classes: $$\mathscr F \in \HH_{2n,n}|_{\alpha_{1}} \ \ \ \ \ \ \ \  \text{                                                    and                                            } \ \ \ \ \ \ \   \mathscr G \in \HH^{\partial}_{2n,n}|_{\alpha_{1}}.$$ 
For the $N^{th}$ coloured Jones polynomial of this link, we have used:

$$\mathscr {F}_n^N \in \HH_{2n,n(N-1)}|_{\alpha_{N-1}} \ \ \text{                            and                                    } \ \ \mathscr {G}_n^N \in \HH^{\partial}_{2n,n(N-1)}|_{\alpha_{N-1}}.$$
 
 One could take the coloured Jones polynomial, consider it as a sum of Jones polynomials of cablings of the knot, and apply the Bigelow-Lawrence model for each term. We would like to emphasise that there is a difference that occurs between these two approaches. In the cabling picture, one would get the coloured Jones invariant as a sum of pairings where the number of punctures varies between $$n  \ \ \ \ \ \ \ \text{and} \ \ \ \ (N-1)n \ \ \ \ \ $$ as well as the the number of points in the configuration space. In other words, both parameters of the Lawrence representation that is used would depend on the colour $N$.
 
 In our model, we use always the $2n$-punctured disc, which means that the first component of the Lawrence representation is fixed. Then, just the second component, namely the number of points in the configuration space depends on the colour $N$.

\subsection*{Questions-Categorification} 

A feature of this model consists in the fact that the homology classes that lead to the coloured Jones polynomial are linear combinations of Lagrangian submanifolds in the configuration space. One of the main directions that we are interested to pursue is to study the Floer type homology coming from this model and whether this theory is invariant with respect to the choice of the braid. If it is, this will lead to a geometrical categorification for the coloured Jones polynomials. After that, the aim would be to understand the relation between this categorificaton and the symplectic Khovanov homology studied by Seidel-Smith and Manolescu (\cite{SM}, \cite{M1}) for the case of the Jones polynomial.

More specifically, the Lawrence representation $\HH_{2n,n(N-1)}$ and its dual  $\HH^{\partial}_{2n,n(N-1)}$ are generated by homology classes of $m$-dimensional Lagrangian submanifolds in a $\Z \oplus \Z$-covering a configuration space $\tilde{C}_{2n,n(N-1)}$, called ``multiforks" and ``barcodes". Both are lifts of Lagrangian submanifolds from $C_{2n,n(N-1)}$. This means that $(\beta_{n}\cup \mathbb I_n )\mathcal F^N_n$ and $\mathcal G^N_n$ are given by linear combinations of homology classes of Lagrangian submanifolds in $\tilde{C}_{2n,n(N-1)}$. The question would be to apply graded Floer homology to each graded intersection from before, and to study the invariance of the corresponding Floer homology groups $HF^N_m(\beta_n)$ with respect to the choice of the braid representative.

 {\bf{Structure of the paper}}:
 The paper has six main sections. In Section \ref{2}, we present the quantum group $U_q(sl(2))$ that we work with, properties about its representation theory and the definition of the coloured Jones polynomials. Then, Section \ref{3} contains the details about the homological Lawrence representation. Further on, in Section \ref{4} we define the dual Lawrence representation and we present a graded geometric intersection form that relates the two representations, with emphasis on the way of computing this form and on the non-degeneracy of this pairing. After that, in Section \ref{ident}, we present the identifications between quantum and homological representations of the braid group and discuss in detail the specialisation at natural parameters. Section \ref{s}, is devoted to the construction and the proof of the homological model for the coloured Jones polynomials, where the homology classes are constructed in the specialised homology groups. The last part, Section \ref{ss} is devoted to the lift of the homology classes from this topological model into the non-specialised homology groups.  

\
 
{\bf{Acknowledgements}}: I would like to thank very much my advisor, Professor Christian Blanchet for asking this beautiful problem of finding homological interpretations for quantum invariants, at the beginning of my PhD. I am very grateful for many useful and nice discussions and for his continuous support. Also, I am thankful to Dr Martin Palmer for discussions about homology with twisted coefficients and Lawrence representations.

This research was supported by grants from R\'egion Ile-de-France.\\  
I would also like to thank the Isaac Newton Institute for Mathematical Sciences, Cambridge, for support and hospitality during the programme 
``Homology theories in low dimensional topology'' where work on this paper was undertaken between 12 February-28 March 2017. During this time, I was also supported by EPSRC grant no EP/K032208/1.
The second version of this paper was prepared at the University of Oxford, and I acknowledge the support of the European Research Council (ERC) under the European Union's Horizon 2020 research and innovation programme (grant agreement No 674978).

\section{Representation theory of $U_q(sl(2))$}\label{2}
\subsection{$U_q(sl(2))$ and its representations}
\begin{definition}
Let $\qs, s$ parameters and consider the ring 
$$\Li_s:=\Z[\qs^{\pm1},s^{\pm1}].$$
Consider the quantum enveloping algebra $U_q(sl(2))$, to be the algebra over $\Li_s$ generated by the elements $\{ E,F^{(n)}, K^{\pm1} | \ n \in \mathbb N^{*} \}$ with the folowing relations:
\begin{equation*}
\begin{cases}
KK^{-1}=K^{-1}K=1; \ \ \ KE=\qs^2EK; & \ \ \ KF^{(n)}=\qs^{-2n} F^{(n)}K;\\
F^{(n)}F^{(m)}= {n+m \brack n}_{\qs} F^{(n+m)}&\\
[E,F^{(n+1)}]=F^{(n)}(q^{-n}K-q^{n}K^{-1}).
\end{cases}
\end{equation*}
The generators $F^{(n)}$ correspond to the "divided powers" of the generator $F$, from the version of the quantum group $U_q(sl(2))$ with generators $ \{ E,F,K^{\pm1} \}$. 
\end{definition}
Then, one has that $U_q(sl(2))$ is a Hopf algebra with the following comultiplication, counit and antipode:
\begin{equation*}
\begin{cases}
\Delta(E)= E\otimes K+ 1 \otimes E  & S(E)=-EK^{-1}\\
\Delta (F^{(n)}) = \sum_{j=0}^n q^{-j(n-j)}K^{j-n}F^{(j)}\otimes F^{(n-j)}& S(F^{(n)})=(-1)^{n}q^{n(n-1)}K^{n}F^{(n)}\\
\Delta(K)=K\otimes K & S(K)=K^{-1}\\
\Delta(K^{-1})=K^{-1}\otimes K^{-1} &  S(K^{-1})=K.
\end{cases}
\end{equation*}

 We will use the following notations:
$$ \{ x \} :=\qs^x-\qs^{-x} \ \ \ \ [x]_{\qs}:= \frac{\qs^x-\qs^{-x}}{\qs-\qs^{-1}}$$
$$[n]_\qs!=[1]_\qs[2]_\qs...[n]_\qs$$
$${n \brack j}_\qs=\frac{[n]_\qs!}{[n-j]_\qs! [j]_\qs!}.$$

Now we will describe the representation theory of $U_q(sl(2))$. 
In the sequel the abstract variable $s$ will be thought as being the weight of the Verma module.

\begin{definition} (The Verma module)

Consider $\hat{V}$ be the $\Li_s$-module generated by an infinite family of vectors $\{v_0, v_1,...\}$. The following relations define an $U_q(sl(2))$ action on $\hat{V}$:
\begin{equation}\label{action}  
\begin{cases}
Kv_i=s\qs^{-2i}v_i,\\
Ev_i=v_{i-1},\\
F^{(n)}v_i = {n+i \brack i}_\qs \prod _{k=0}^{n-1} (s\qs^{-k-i}-s^{-1} \qs^{k+i}) v_{i+n}.
\end{cases}
\end{equation}
\end{definition}

\subsection{Specialisations}

\
For our purpose, to arrive at the definition of the coloured Jones polynomial, it is needed to consider some specialisations of the previous quantum groups and its Verma representations.
\begin{definition}
Consider the following specialisations of the coefficients:

2) Let $h, \lambda \in \C$ and $\q=e^h$. In the following, we will have $e^{\lambda h}=\q^{\lambda}$. In this case, we specialise both variable $q$ and the highest weight $s$ to concrete complex numbers:
$$\eta_{\q,\lambda}:\Z[q^{\pm},s^{\pm}]\rightarrow\C$$ 
 $$ \eta_{\q,\lambda}(q)=e^h \ \ \ \ \  \eta_{\q,\lambda}(s)=e^{\lambda h}.$$
 
3) This is the case where the coloured Jones polynomial will be defined. Consider $q$ still as a parameter (it will be the parameter from the coloured Jones polynomial), and specialise the highest weight using $\lambda=N-1 \in \N$ a natural parameter:
$$ \eta_{\lambda}:\Z[q^{\pm},s^{\pm}]\rightarrow\Z[q^{\pm}]$$ 
$$ \eta_{\lambda}(s)= q^{\lambda }.$$
\end{definition}
Using these specialisations, we will consider the corresponding specialised quantum groups and their representation theory. We obtain the following:
\begin{center}
\label{tabel}
\begin{tabu} to 1.02\textwidth { | X[c] | X[c] | X[c] | X[c] | } 
 \hline
 Ring & Quantum Group & Representations &  Specialisations \\ 
\hline                                    
$ \Li_s=\Z[q^{\pm},s^{\pm}] $ & $ U_q(sl(2))$  & $\hat{V}$  & $1) q,s \text{  param                                      \ \ \ \ \ \ \ \ }  $ \\
\hline                                    
 $ \C $ & $\U_{\q,\lambda}=U_q(sl(2))\otimes_{\eta_{\q,\lambda}}\C$  & $\hat{V}_{\q,\lambda}=\hat{V}\otimes_{\eta_{q,\lambda}}\C$  & $2)(\q=e^h, \lambda)\in \C^2$ 
 $ \eta_{\q,\lambda} $ \\
\hline
 $ \Li=\Z[q^{\pm}] $ & $\U=\U_{\lambda}=$ $U_q(sl(2))\otimes_{\eta_{\lambda}}\Z[q^{\pm}]$  & $\hat{V}_{\lambda}=\hat{V}\otimes_{\eta_{\lambda}}\Z[q^{\pm}]$  $V_N\subseteq \hat{V}_{\lambda}$& $ 3) (q \text{ param,                                                                   \ \ \ \ \ \ \ \ \ \ \ \ \ \ \ \ } $ $ \ \ \ \ \ \ \ \ \ \ \ \ \ \ \ \ \lambda=N-1 \in \N)$  $ \ \ \ \ \ \ \ \ \ \ \ \ \ \ \ \ \ \ \ \ \ \ \ \ \eta_{\lambda} \ $ \\
\hline
\end{tabu}
\end{center}

\

\begin{remark}
If we specialise as above, $\U_{\lambda}$ and $\U_{\q,\lambda}$ become Hopf algebras and  
 $\hat{V}_{\q,\lambda}$ a $\U_{\q, \lambda}$-representation and
 $\hat{V}_{\lambda}$ a $\U_{\lambda}$-representation.

\end{remark}

\begin{lemma}\label{inclusion}
If $\lambda=N-1 \in \N$,  then $ \{ v_{0},...,v_{N-1} \} $ span an $N$-dimensional $\U_{\lambda}$-submodule inside $\hat{V}_{N-1}$.
Denote this module by 
$$V_{N}:=<v_{0},...,v_{N-1}> \subseteq \hat{V}_{N-1}.$$
\end{lemma} 
\begin{proof}
We can see that $K$ acts by scalars and $E$ decreases the indices on the basis given above. We only have to see the action of the $F^{(n)}$ generators on this space. 
$$ F^{(n)}v_i= {n+i \brack i}_q \prod _{k=0}^{n-1} (q^{(N-1)-(k+i)}-q^{-[(N-1)-(k+i)]}) v_{i+n}. $$

Let $i\in \{0,...,N-1 \}$.

If $\bold {n<N-i}$, from the definition, the action of $F^{(n)}$ will remain inside the module:
$$F^{(n)}v_i \simeq v_{i+n} \in V_{N}.$$

For $\bold {n\geq N-i}$, we obtain that $n-1 \geq N-1-i$.

This shows us that in the previous formula concerning the $F^{(n)}$-action, there is a term corresponding to the parameter $k=N-1-i$ and actually its coefficient vanishes.

We obtain that $F^{(n)}v_i=0$, for any $n\geq N-i$.

This concludes the existence of the $N$-dimensional submodule $V_N$. 

\end{proof}

\subsection{The Reshetikhin-Turaev functor} 

In this section, we will present the general construction due to Reshetikhin and Turaev, that having as input any ribbon category  $ \mathscr C$, gives a functor from the category of tangles towards $\mathscr C$. In particular, this machinery leads to link invariants. We will present this method, using the category of representations of the quantum group $\U$. 
\begin{notation}
In the sequel, for any two representations $V$ and $W$, we consider the twist $\tau: V\otimes W\rightarrow W \otimes V$ which swaps their components by:
$$\tau(x \otimes y)=y \otimes x.$$
\end{notation}
In the following proposition, we denote by $U_\qs(sl(2))\hat{\otimes} U_\qs(sl(2))$ a completion of the module $U_\qs(sl(2)) \otimes U_\qs(sl(2))$, where we allow infinite formal sums of tensors products.
\begin{proposition}(\cite{JK},\cite{Ito})\label{P:RT}(Braid group action on the Verma module)

There exist an element $R \in U_\qs(sl(2))\hat{\otimes} U_\qs(sl(2))$ called $R$-matrix given by the following formula
$$R=\sum_{n=0}^{\infty} q^{\frac{n(n-1)}{2}}E^{n}\otimes F^{(n)}$$
 which leads to representations of the braid group. More precisely, for the Verma module $\hat{V}$ of $U_q(sl(2))$ on has the following well defined action:
\begin{equation}
\begin{aligned}
\varphi^V_n: B_n \rightarrow Aut_{U_q(sl(2))}\left(V^{\otimes n}\right) \ \ \ \ \ \ \ \ \ \ \ \ \ \ \ \ \ \ \ \ \ \ \\
\ \ \ \ \ \ \ \ \ \ \ \ \ \sigma _i^{\pm 1}\rightarrow Id_V^{\otimes (i-1)}\otimes (\mathscr R^{\pm 1} \circ \tau) \otimes Id_V^{\otimes (n-i-1)}.
\end{aligned}
\end{equation}
In this definition, we used the following notation: $$\mathscr R=C \circ R$$
$$C(v_i \otimes v_j)=s^{-(i+j)}q^{2ij}v_j \otimes v_i.$$
\end{proposition}
For our purpose, concerning the coloured Jones polynomials, we are interested in finite dimensional representations of the quantum group $\U$. In the following part, we will see that the generic $R$-matrix induces by specialisation braid group actions onto tensor powers of finite dimensional $\U$-representations, which will be used for the definition of these invariants. 

\begin{proposition}

 1) There exist a braid group action onto the subcategory of $\U$-representations $Rep({\U})$ (finite dimensional or Verma module). This comes from the specialisation of the $R$-matrix $\mathscr R|_{\eta_{\lambda}\otimes\eta_{\lambda}} \in \U \hat{\otimes} \U$ and it has the following form:
$$R_{V,V}=\mathscr R|_{\eta_{\lambda}\otimes\eta_{\lambda}}\circ \tau \in Hom_{U_q(sl(2))}(V\otimes V,V \otimes V),  \ \ \ \forall \ V \in Rep(\U).$$

2) The category of finite dimensional representations of ${\U}$ has the following dualities:
$$\forall \ V_N \in Rep^{\text{f. dim}}_{\U}$$ 
\begin{align}\label{E:DualityForCat}
\tcoev_{V_N} :\, & \Li \rightarrow V_N\otimes V_N^{*} \text{ is given by } 1 \mapsto
  \sum v_j\otimes v_j^*,\notag
  \\
 \tev_{V_N}:\, & V_N^*\otimes V_N\rightarrow \Li \text{ is given by } f\otimes w
  \mapsto f(w),\notag
  \\
\coev_{V_N}:\, & \Li \rightarrow V_N^*\otimes V_N \text{ is given by } 1 \mapsto \sum
   v_j^*\otimes K^{-1}v_j,
  \\
\ev_{V_N}:\, & V_N\otimes V_N^{*}\rightarrow \Li \text{ is given by } v\otimes f
  \mapsto  f(Kv),\notag
\end{align}
for $\{v_j\}$ a basis of $V_N$ and $\{v_j^*\}$ the dual basis
of ${V_N}^*$.\\
 \end{proposition}
 
 \begin{remark} \label{R:R}
The action of $\mathscr{R}$ on the standard basis of the Verma module $\hat{V}\otimes \hat{V}$ is given in \cite{Ito}(Section 4.1):
$$ \mathscr{R}(v_i\otimes v_j)= s^{-(i+j)} \sum_{n=0}^{i} F_{i,j,n}(q) \prod_{k=0}^{n-1}(sq^{-k-j}-s^{-1}q^{k+j}) \ v_{j+n}\otimes v_{i-n}.$$
In the previous formula $F_{i,j,n}\in \Z[q^{\pm}]$ has the expression:
$$F_{i,j,n}(q)=q^{2(i-n)(j+n)}q^{\frac{n(n-1)}{2}} {n+j \brack j}_q.$$
\end{remark}
In the following part we will present the Reshetikhin-Turaev method of obtaining link invariants. Firstly, we see the definition for the category of tangles.

\begin{definition}
The category of oriented tangles $\mathscr T$ is defined as follows:
$$Ob(\mathscr T)= \{(\epsilon_1,...,\epsilon_m) | \ m \in  \N,\epsilon_i \in \{\pm1 \} \}.$$ 
$$Hom_{\mathscr T}\left( (\epsilon_1,...,\epsilon_m);(\delta_1,...,\delta_n)\right) = \{ \text{oriented tangles} \ \ \ \ \ \ \ \ \ \ \ \ \ \ \ \ \ \ \ \ \ \ \ \ \ \ \ \ \ \ \ \ \ \ \ \ \ $$
$$ \ \ \ \ \ \ \ \ \ \ \ \ \ \ \ \ \ \ \ \ \ \ \ \ \ \ \ \ \ \ T : (\epsilon_1, ..., \epsilon_m) \uparrow (\delta_1, ..., \delta_n)\}/ \text{isotopy}.$$
\end{definition}
\begin{remark}
The tangles  $\mathscr T$ have to respect the signs $\epsilon_i$ which are at their boundaries.
Once we have such a tangle, it has an induced orientation, coming from the signs, using the following conventions:
$$(-) \ \downarrow,  \ \ \ (+) \ \uparrow$$
\end{remark}
\begin{theorem}(Reshetikhin-Turaev)\\
For any finite dimensional $\U$-representation $V$, there exist an unique monoidal functor 
$$\F_V: \mathscr T\rightarrow Rep^{f. dim}(\U)$$
such that it respects the following local relations:
$$1) \ \F_V((+))=V; \ \ \ \F_V((V,-))=V^{*} \ \ \ \ \ \ \ \ \ \ \ \ $$ 
\begin{align*} 
2) \F_V(
\tikz[x=1mm,y=1mm,baseline=0.5ex]{\draw[->] (3,3)--(0,0); \draw[line width=3pt,white] (0,3)--(3,0); \draw[->] (0,3)--(3,0);}
) &= R_{V,V} \in Hom(V\otimes V\rightarrow V \otimes V) \\
\F_V(
\tikz[x=1mm,y=1mm,baseline=0.5ex]{\draw[<-] (0,3) .. controls (0,0) and (3,0) .. (3,3); \draw[draw=none, use as bounding box](-0.5,0) rectangle (3,3);}
) &= \tcoev_{V} :\Z[q^{\pm}]\rightarrow V \otimes V^{*} \\
\F_V(
\tikz[x=1mm,y=1mm,baseline=0.5ex]{\draw[<-] (3,0) .. controls (3,3) and (0,3) .. (0,0); \draw[draw=none, use as bounding box](0,0) rectangle (3.5,3);}
) &= \ev_{V} :V \otimes V^{*} \rightarrow \mathbb \Z[q^{\pm}].
\end{align*}
\end{theorem}

\subsection{The coloured Jones polynomial $J_N(L,q)$ }

So far, we have seen the algebraic structure coming from the representation theory of the quantum group $\U$ and the Reshetikhin-Turaev construction. Now, we will present how this machinery is actually a tool that leads to quantum invariants for knots. 
\begin{definition} (The coloured Jones polynomial-V. Jones)\\
Let $N$ be a natural number and $L$ a link. 
Then $L \in \Hom_{\mathscr T}(\emptyset,\emptyset)$.\\
The $N'^{th}$ coloured Jones polynomial is defined from the Reshetikhin-Turaev functor, using the representation $V_N \in Rep(\U)$ as colour, in the following way:
$$J_N(L,q):= \F_{V_N}(L) \in \Z[q^{\pm}] $$ \label{D:Jones}
(in this definition, the functor applied to the link gives a morphism from $\Z[q^{\pm}]$ to $\Z[q^{\pm}]$, which is identified with a Laurent polynomial).
\end{definition} 
This definition can be seen using the braid group action and the dualities from the category of $\U$-representations.
\begin{notation}
Let $V_1,...,V_n$ be vector spaces and for $i\in \N$ we define the twisted operator as:
$$\tau_{(i,n)}:V_1\otimes...\otimes V_n \rightarrow V_1 \otimes V_{i-1} \otimes V_n \otimes V_{i+1}...\otimes V_{n-1} \otimes V_i$$
which interchanges the $i^th$ and $n^{th}$ components.

We denote by 
$$\ev^i_{V}:V^{\otimes i}\otimes(V^{\star})^{\otimes i}\rightarrow V^{\otimes i-1} \otimes (V^{\star})^{\otimes i-1}. $$
$$\tcoev^i_{V}:V^{\otimes i-1}\otimes (V^{\star})^{\otimes i-1}\rightarrow V^{\otimes i} \otimes (V^{\star})^{\otimes i}. $$
 the evaluation ( and coevaluation ) corresponding to the first and last component, which are defined as follows:
\begin{align*}
\begin{cases}
\ev^i_{V}:= (\ev_{V}\otimes Id^{\otimes 2i-2})\circ \tau_{2,2i}\\
\tcoev^i_{V}:=\tau_{2,2i} \circ (\tcoev_{V}\otimes Id^{\otimes 2i-2}).
\end{cases}
\end{align*}

We notice that $\ev^1_{V}=\ev_{V}$ and $\tcoev^1_{V}=\tcoev_{V}$.
\end{notation}
Using the notation, we define the evaluation and coevaluation corresponding to the cups and caps from a braid closure as follows.
\begin{definition}\label{D:1} Consider the following morphisms of $\U$-representations:
\begin{align*}
\begin{cases}
\ev^{\otimes n}_{V_N}:V_N^{\otimes n}\otimes (V_N^{\star})^{\otimes n}\rightarrow \Z[q^{\pm 1}]\\
{\tcoev}^{\otimes n}_{V_N}: \Z[q^{\pm 1}] \rightarrow V_N^{\otimes n} \otimes (V_N^{\star})^{\otimes n}
\end{cases}
\end{align*}
given by the expressions:
\begin{align*}
\begin{cases}
\ev^{\otimes n}_{V_N}=\ev^{1}_{V_N} \circ....\circ \ev^{n}_{V_N}\\
{\tcoev}^{\otimes n}_{V_N}=\tcoev^{n}_{V_N} \circ....\circ \tcoev^{1}_{V_N}.$$
\end{cases}
\end{align*}
\end{definition}
\begin{notation}(The trivial braid)\label{N:1}\\ 
For the next formula, we denote by:\\
1) $\mathbb I_n$ the trivial braid with $n$ strands in $B_n$, all oriented upwards.\\
2) $\bar{\mathbb I}_n$ the trivial braid with $n$ strands in $B_n$, all oriented downwards.
\end{notation}
\begin{proposition}(\cite{Ito3})(Coloured Jones polynomial from a braid presentation)\label{P:1}
Let us fix $N \in \N$.  Consider $L$ be an oriented knot and $\beta \in B_n$ such that $L=\hat{\beta}$ (braid closure). We denote by $$w:B_n\rightarrow \Z$$ the map given by the abelianisation. Then, the Reshetikhin-Turaev construction leads to the following formula:
$$J_N(L,q)=\frac{1}{[N]_q}q^{-(N-1)w(\beta)}\left( \ev^{\otimes n}_{V_N} \  \circ \ \F_{V_N} (\beta_{n} \otimes \bar{\mathbb I}_n ) \  \circ \ {\tcoev}^{\otimes n}_{V_N} \right) (1).$$
\end{proposition}

As we have seen so far, the construction that leads to the definition of $J_N(L,q)$ is purely algebraic and combinatorial. We are interested in a geometrical interpretation for this invariant. For this purpose, we study the Reshetikhin-Turaev functor applied onto certain intermediary levels of the knot diagram. More precisely, we start with $L$ as a closure of a braid $\beta\in B_{n}$. Then, we split the knot diagram into three main parts as follows: 

1) the evaluation                              
$ \ \ \ \ \ \ \ \ \ \ \  \tikz[x=1mm,y=1mm,baseline=0.5ex]{\draw[<-] (3,0) .. controls (3,3) and (0,3) .. (0,0); \draw[<-] (6,0) .. controls (6,6) and (-3,6) .. (-3,0); \draw[draw=none, use as bounding box](0,0) rectangle (3.5,3);}
$

2) braid level                                    $ \ \ \ \ \ \ \ \ \ \ \ \ \ \ \ \beta_n \otimes \bar{\unit}_{n}$

3) the coevaluation                                 
$ \ \ \ \ \ \ \  \  \tikz[x=1mm,y=1mm,baseline=0.5ex]{\draw[<-] (0,3) .. controls (0,0) and (3,0) .. (3,3); \draw[<-] (-3,3) .. controls (-3,-3) and (6,-3) .. (6,3); \draw[draw=none, use as bounding box](-0.5,0) rectangle (3,3);}
$

We plan to investigate what is happening with $\F_{V_N}$ on each of these main levels. The interesting part and the starting point in our description is the fact that at the level of braid group representation, there is a homological counterpart for the quantum representation, called Lawrence representation(\cite{Law},\cite{Koh}). This relation is established using the notion of highest weight spaces. 

\subsection{Highest weight spaces}
In this part, we will introduce and discuss the properties of certain vector subspaces which live inside the tensor powers of a fixed representation (we will refer to the ones defined in the table \ref{tabel}). These subspaces are rich objects which carry interesting braid group representations, as we will see in the sequel.
\begin{definition}\label{D:index}
For two natural numbers $n,m \in \N$ and a fixed parameter $N \in \N$, consider the following indexing sets:
$$E_{n,m}:= \{e=(e_1,...,e_{n-1})\in \N ^{n-1}\mid e_1+...+e_{n-1}=m \}$$
$$E^N_{n,m}:=\{ e=(e_1,...,e_{n-1}) \in E_{n,m} \mid e_1,...,e_{n-1}\leq N-1 \}$$
$$E^{\geq N}_{n,m}:=\{ e=(e_1,...,e_{n-1}) \in E_{n,m} \mid \exists \ i, \ e_i \geq N \}.$$

Also, for an element $e=(e_1,...,e_n)\in \N ^n$, let us denote:
$$v_e:= \hat{v}_{e_1}\otimes... \otimes \hat{v}_{e_n}.$$
\end{definition}
\begin{remark} \label{R:dd}
The set $E_{n,m}$ has elements which are partitions of the natural number $m$ into $n-1$ natural numbers (allowing zero). Its cardinal is well known and we will use the notation: 
$$d_{n,m}:= \text{card} \ (E_{n,m})= {n+m-2 \choose m}$$
\end{remark}
Let us fix $n,m\in \N $ two natural numbers. Now we will define highest weight spaces which correspond to these input, as follows.
\begin{definition}(Highest weight spaces)\\\label{D:2}  
{\bf 1) The case of two parameters $(q,s)$}\\
The $n^{th}$-weight space of the generic Verma module $\hat{V}$
corresponding to the weight $m$ is defined by:
$$ \hat{V}_{n,m}:= \{v \in \hat V^{\otimes n} \mid Kv=s^nq^{-2m}v    \}.$$
The highest weight space of the generic Verma module 
$\hat{V}^{\otimes n}$ corresponding to the weight $m$:
$$\hat{W}_{n,m}:=\hat{V}_{n,m} \cap Ker E.$$
{\bf 2) Specialisation with two complex numbers}\\
Let $h, \lambda \in \C $ and $\q=e^h$.\\
The weight space of $ \hat{V}^{\otimes n}_{\q,\lambda}$ corresponding to the weight $m$:
$$ \hat{V}^{\q,\lambda}_{n,m}:=\{v\in \hat{V}^{\otimes n}_{\q,\lambda} \mid Kv=q^{n\lambda-2m}v \}.$$
The highest weight space of the Verma module $\hat{V}^{\otimes n}_{\q,\lambda}$  corresponding to the weight $m$:
$$ \hat{W}^{\q,\lambda}_{n,m}:=\hat{W}_{n,m}|_{\eta_{q,\lambda}}.$$
{\bf 3) The case with $q$ parameter and $\lambda$ natural number }

{\bf a) Inside the Verma module $\hat{V}_{N-1} ^{\otimes n}$}

Consider $\lambda=N-1 \in \N$.\\ 
The weight space of $ \hat{V}^{\otimes n}_{N-1}$ of weight $m$:
$$ \hat{V}^{N-1}_{n,m}:=\{v\in \hat{V}^{\otimes n}_{N-1} \mid Kv=q^{n\lambda-2m}v \}.$$
The highest weight space for Verma module $\hat{V}^{\otimes n}_{N-1}$  corresponding to the weight $m$:
$$ \hat{W}^{N-1}_{n,m}:=\hat{W}_{n,m}\mid_{\eta_{N-1}}.$$

{\bf b) Inside the finite dimensional module $V_{N} ^{\otimes n}$}

The weight space for the finite dimensional representation $V^{\otimes n}_{N}$ of weight $m$:
$$V^N_{n,m}:=\{v\in V^{\otimes n}_{N} \mid Kv=q^{n(N-1)-2m}v \}.$$
The highest weight space of the finite dimensional representation $V^{\otimes n}_{N}$  corresponding to the weight $m$:
$$ W^N_{n,m}:=\hat{W}_{n,m}\mid_{\eta_{N-1}}\cap \ V^{\otimes n}_{N}.$$
\end{definition}

\begin{remark}\label{R:1}
Since $V_N \subseteq \hat{V}_{N-1}$, we have $$v_e \in V^{\otimes n}_N \text{ if and only if } e \in E^N_{n,m}.$$
This will happen also at the level of (highest) weight spaces:
$$V_{n,m}^N \subseteq \hat{V}_{n,m}^{N-1}  \ \ \ \ \text {and} \ \ \ \  W_{n,m}^N \subseteq \hat{W}_{n,m}^{N-1}.$$

\end{remark}
\begin{remark}\label{R:ddd}
{\bf 1) Basis for the weight spaces from Verma module}\\
One can see easily that:
\begin{equation*}
\begin{cases}
\hat{V}_{n,m}= <v_e  \mid e \in E_{n+1,m}>_{\Li_s} \ \ \ \ \subseteq \hat{V}^{\otimes n}.\\
 \hat{V}^{N-1}_{n,m}= <v_e \mid e \in E_{n+1,m}>_{\Z[q^{\pm}]} \subseteq \hat{V}^{\otimes n}_{N-1}.
\end{cases}
\end{equation*}
{\bf 2) Basis for the weight space of the finite dimensional module $V_N$}:\\
Using the remark \ref{R:1} and the previous part $1)$, we conclude that:
$$V_{n,m}^N= <v_e \mid e \in E^{N}_{n+1,m}>_{\Z[q^{\pm}]} \subseteq V^{\otimes n}_{N}.$$
\end{remark}
\begin{remark}
Moreover, if we denote by 
\begin{equation*}
V^{\geq N}_{n,m}= <v_e \mid e \in E^{\geq N}_{n+1,m}>_{\Li_s}\subseteq \hat{V}^{\otimes n}_{N-1}
\end{equation*}
then we have the following splitting as vector spaces:
\begin{equation*}
\hat{V}^{N-1}_{n,m}= V^N_{n,m} \oplus V^{\geq N}_{n,m}
\end{equation*}
\end{remark}
From the remark \ref{R:dd}, it follows that the dimensions of these weight space are:
$$\text{dim} \left( \hat{V}_{n,m}\right)=\text{dim} \left(\hat{V}^{N-1}_{n,m}\right)=d_{n+1,m}={n+m-2 \choose m}.$$

\subsection{Bases for heighest weight spaces}
As we have seen above, there is a strightforward definition of bases in weight spaces. Unlike this situation, for the case of highest weight spaces this becomes a more subtle question. In \cite{JK}, there were presented bases in the highest weight spaces from the Verma module, as well as connections between highest weight spaces and weight spaces, which correspond to different parameters $n$ and $m$. 

Moreover, Jackson and Kerler made the connection towards homological braid group representations, proving that for the parameter $m=2$, the braid group action onto the highest weight space $\hat{W}_{n,m}$ corresponds to the homological Lawrence-Bigelow-Krammer representation (\cite{Law1}, \cite{Big0},\cite{Kr1},\cite{Kr2}). They conjectured that this identification is true for any natural number $m$. Later on, Kohno (\cite{Ito},\cite{Koh}) proved this conjecture. We will discuss in details this identifications in section \ref{ident}.

 Now, we will present from \cite{Ito} some "good" bases for the highest weight spaces, that will have a role in the identification between quantum and homological representations of the braid groups. 
\begin{definition} (Basis for $\hat{W}_{n,m}$)
For $e \in E_{n+1,m}$, we denote by:
$$v^s_e:= s^{\sum_{i=1}^ni e_i}v_{e_1}\otimes... \otimes v_{e_n}.$$
Notice that $\mathscr B_{{\hat V}_{n,m}}:=\{ v_e^s \mid e \in E_{n+1,m} \}$ form a basis for $\hat{V}_{n,m}$.
\end{definition}
In the sequel, the highest weight spaces $\hat{W}_{n,m}$ will be identified with a certain subspace of the weight spaces $\hat{V}_{n,m}$.

Let $\iota: E_{n,m} \rightarrow E_{n+1,m}$ be the inclusion:
$$\iota ((e_1,...,e_{n-1}))=(0,e_1,...,e_{n-1}).$$
Denote by $$\hat{V}'_{n,m}:=\Li_s \hat{v}_0 \oplus \hat{V}_{n-1,m}\subseteq \hat{V}_{n,m}.$$ 
Then, the set $\mathscr B_{\hat{V}'_{n,m}}:= \{ \hat{v}^s_{\iota(e)} \mid e \in E_{n,m} \}$ gives a basis for the space $\hat{V}'_{n,m}$.
\begin{proposition}\cite{Ito}\label{P:bqu}
Consider the function
$\phi: \hat{V}'_{n,m}\rightarrow \hat{W}_{n,m}$ described by the formula:
$$\phi (w):=\sum_{k=0}^{m} (-1)^k s^{-k(n-1)} q^{2mk-k(k+1)} v_k \otimes E^k(w).$$
Then $\phi$ is an isomorphism of $\Li_s$-modules. Moreover, the set $$\mathscr B_{\hat{W}_{n,m}}=\{ \phi\left(v^s_{\iota(e)}\right) \mid e \in E_{n,m}\}$$ describes a basis for the generic highest weight space  $\hat{W}_{n,m}$.
Using the remarks \ref{R:dd} and \ref{R:ddd}, it follows that:
$$\text{ dim} \ (\hat{W}_{n,m} )=d_{n,m}={n+m-2 \choose m}.$$
\end{proposition}

\subsection{Quantum representations of the braid groups}
In the following part, we will see that the braid group action on tensor powers of the (generic) Verma module and the finite dimensional module, passes at the level of highest weight spaces.
 \begin{proposition} \label{P:hw}
Since $\varphi^{\hat{V}}_n$ gives an action on $\hat{V}^{\otimes n}$ which commutes with the quantum group action ( Proposition \ref{P:RT}), it will commute with the actions of the generators $K, E$ of the quantum group. Then, it induces a well defined action on the generic highest weight spaces $$\phs:B_n \rightarrow \Aut( \hat{W}_{n,m}).$$
\nolinebreak 
This action in the basis  $\mathscr B_{\hat{W}_{n,m}}$ leads to a representation:
$$1) \ \phs:B_n \rightarrow GL(d_{n,m}, \Li_s) \ \ \ \ \ \ $$
called the generic quantum representation on highest weight spaces of the Verma module.
\end{proposition}
\begin{proposition} \label{P:hww}Similarly, using the previous specialisations we get induced braid group actions as follows.
$$2) \ \  \phh:B_n \rightarrow \Aut( \hat{W}^{\q,\lambda}_{n,m}) \ \ \ \ \ \ $$ is a well defined action induced by $\varphi^{\hat{V}_{\q,\lambda}}_n$
$$3) \ a) \ \pv:B_n \rightarrow \Aut( \hat{W}^{N-1}_{n,m})$$ is a well defined action induced by $\varphi^{\hat{V}_{N-1}}_n$ called the quantum representation on highest weight spaces of the Verma module.

$$ 3) \ b) \ \ph:B_n \rightarrow \Aut(W^{N}_{n,m}) \ \ \ \ $$ is a well defined action induced by $\varphi^{V_N}_n$ called the quantum representation on highest weight spaces of the finite dimensional module.
\end{proposition}

As a summary, we have the following highest weights spaces, which carry braid group actions and live inside the $n^{th}$ tensor power of different specialisations of the Verma module $\hat{V}$:
\begin{center}

\label{tabel}
\begin{tabu} to 1\textwidth { | X[l] | X[c] | X[c] | X[r] | } 
 \hline
 Braid group action & Highest weight space & Representation &  Specialisation \\ 
\hline                                    
$ \ \ \ \ \ \ \ \phs$  &  $ \ \ \ \ \ \hat{W}_{n,m} $ & $ \hat{V}^{\otimes n}$ & $ \ \ 1) \ q,s\text{ param} \ \ \ $ \\
\hline                                    
 $ \ \ \ \ \ \phh$ &  $\ \ \ \ \ \hat{W}^{\q,\lambda}_{n,m} $ & $ \hat{V}^{\otimes n}_{\q,\lambda}$ & $2) \q=e^h,\lambda \in \C $ $ \eta_{\q,\lambda} \ \ \ \ \ \ \ \ \ \ $ \\
\hline
 $ \ \ \ \ \ \pv$ & $\ \ \ \ \ \ \hat{W}^{N-1}_{n,m} $ & $  \ \ \ \ \ \ \hat{V}^{\otimes n}_{N-1} \ \ \ \ $& $ \ \ \ \ \ 3)a) \ \ \ q \ \text{param} \ \ \ \ \ \ \ $ $\lambda=N-1 \in \N$ $\eta_{\lambda \ \ \ \ \ \ \ \ \ \ }$ \\
\hline
 $ \ \ \ \ \ \ph$ &  $\ \ \ \ \ \ W^{N}_{n,m} $ & $  \ \ \ \ \ \ V^{\otimes n}_N \ \ \ \ $& $ \ \ \ \ \ 3)b) \ \ \ q \ \text{param} \ \ \ \ \ \ \ $ $\lambda=N-1 \in \N$ $\eta_{\lambda \ \ \ \ \ \ \ \ \ \ }$ \\
\hline

\end{tabu}
\end{center}

\section{Lawrence representation}\label{3}
\subsection{Local system}
In this section we present certain braid group representations introduced by Lawrence (\cite{Law}). They are defined using the middle homology of a certain covering of the configuration space in the punctured disk. They are called homological Lawrence representations and they have a topological description. 

Let $n \in \N$. Consider $\D^2\subseteq \C$ be the unit disk including its boundary and $\{ p_1,...,p_n\}$-$n$ points in its interior, on the real axis.

Let $D_n:=\D^2 \setminus \{p_1,...,p_n \} $ and fix $m\in \N$ a natural number.\\
Consider $C_{n,m}$ to be the unordered configuration space of $m$ points in the $n$-punctured disk:
$$C_{n,m}=Conf_m(D_n)= \Big( \D^{\times m}_n \setminus \{x=(x_1,...,x_m) \mid \ \exists \ i,j \ \text { such that } x_i=x_j\}\Big) / Sym_m$$
(here, by $Sym_m$ we denote the symmetric group of order $m$).\\
Let us fix $d_1,...,d_m \in  \partial D_n$.

\begin{tikzpicture}
\foreach \x/\y in {0/2,2/2,4/2,2/1,2.3/1,3/1.07} {\node at (\x,\y) [circle,fill,inner sep=1pt] {};}
\node at (0.2,2) [anchor=north east] {$p_1$};
\node at (2.2,2) [anchor=north east] {$p_i$};
\node at (4.2,2) [anchor=north east] {$p_n$};
\node at (3,2) [anchor=north east] {$\sigma_i$};
\node at (2.2,1) [anchor=north east] {$d_1$};
\node at (2.5,1.02) [anchor=north east] {$d_2$};
\node at (3.2,1.05) [anchor=north east] {$d_m$};
\node at (2.62,2.3) [anchor=north east] {$\wedge$};
\draw (2,1.8) ellipse (0.4cm and 0.8cm);
\draw (2,2) ellipse (3cm and 1cm);
\foreach \x/\y in {7/2,9/2,11/2,8.9/1,9.5/1,10/1.07} {\node at (\x,\y) [circle,fill,inner sep=1pt] {};}
\node at (7.2,2) [anchor=north east] {$p_1$};
\node at (9.2,2) [anchor=north east] {$p_i$};
\node at (11.2,2) [anchor=north east] {$p_n$};
\node at (9.2,1) [anchor=north east] {$d_1$};
\node at (9.7,1.02) [anchor=north east] {$d_2$};
\node at (10.2,1.05) [anchor=north east] {$d_m$};
\node at (9,1.5) [anchor=north east] {$\delta$};
\draw (9,2) ellipse (3cm and 1cm);
\draw (9.5,1)  arc[radius = 3mm, start angle= 0, end angle= 180];
\draw [->](9.5,1)  arc[radius = 3mm, start angle= 0, end angle= 90];
\draw (8.9,1) to[out=50,in=120] (9.5,1);
\draw [->](8.9,1) to[out=50,in=160] (9.25,1.12);
\end{tikzpicture}

\begin{definition}(Local system on $C_{n,m}$)\\
Let us define the abelianisation map by $$ab: \pi_1(C_{n,m}) \rightarrow H_1(C_{n,m}).$$
\nolinebreak
Then, it is known that for any $m\geq 2$ one has that:
\begin{equation*}
\begin{aligned}
H_1(C_{n,m}) \ \ \simeq \ \ \Z^{n} \ \ \ \oplus \ \ \ \Z \ \ \ \ \ \ \ \ \ \ \ \  \ \ \ \ \ \ \ \\
{ \ \ \ \ \ \ \ \ \ \ \ \ \ \ \ \ \ \ \ \ \ \ \ \  \ \ \ \ } <ab(\Sigma_i)> \ <ab(\Delta)>, \ \  {i\in \overline{1,n}}.
\end{aligned}
\end{equation*}
Here, the $i$-th element of $\Z^{n}$ is generated by the class of a loop in the configuration space of $m$ points, whose first component goes around the puncture $p_i$ and all the others are constant:
$$\Sigma_i(t):=\{ \left(\sigma_i(t), d_2,...,d_m \right) \}, t \in [0,1].$$
The last component is generated by the class of a loop $\Delta$ which swaps two points between them:
$$\Delta(t):=\{ \left(\delta(t), d_3,...,d_m \right) \}, t \in [0,1].$$
Then, consider the following function: 
\vspace{-0.1 cm}
\begin{equation*}
\begin{aligned}
aug: \Z^{n} \ \oplus \ \Z \rightarrow \ \Z \ \ \oplus \ \ \Z \ \ \ \\
 \ \ \ \ \ \ \ \ \ \ \ \ \ \ \ \ \ \ \ \ \ \ \ <x> \ <d> \\
 \ \ \ \ \ \ \ \ \ \ \ \ \ \ aug\left((x_1,...,x_n),y\right)=\left(x_1+...+x_n,y\right).
\end{aligned}
\end{equation*}
\nolinebreak
Consider the local system defined by the composition of the previous maps:
$$\phi: \pi_1(C_{n,m}) \rightarrow \Z \oplus \Z.$$
$$\phi= aug \circ ab.$$ \label{R:ll}
\end{definition}
\begin{definition}(Covering space)\\
Let $\tilde{C}_{n,m}$ be the covering of $C_{n,m}$ corresponding to $Ker(\phi)$ and its associated projection map $\pi: \tilde{C}_{n,m} \rightarrow C_{n,m}$.
\end{definition}
\begin{remark} \label{R:mm}
The deck transformations of the covering are: 
$$Deck(\tilde{C}_{n,m})=\Z\oplus \Z.$$ Each deck transformation induces a cellular chain map for $\tilde{C}_{n,m}$ and this map will pass to the homology groups. So, we have an action: 
$$\Z \oplus \Z \curvearrowright H^{\text{lf}}_m(\tilde{C}_{n,m}, \Z)$$ 
( here $H^{lf}$ means the Borel-Moore homology, the homology of locally finite chains).
Moreover, this action will be defined at the level of the group ring: $$\Z[\Z \oplus \Z]\simeq \Z[x^{\pm}, d^{\pm}].$$  
It follows that the homology groups $H^{\text{lf}}_m(\tilde{C}_{n,m}, \Z)$ and $H_m(\tilde{C}_{n,m}, \Z; \partial)$ have the structure of $\Z[x^{\pm}, d^{\pm}]$-modules. 
\end{remark}
\subsection{Basis of multiforks}
So far, we have seen the covering of the configuration space whose homology will lead to the Lawrence representation. For that, we will define certain subspaces in the Borel-Moore homology and homology relative to the boundary of $\tilde{C}_{n,m}$.

 \begin{center}
\begin{tikzpicture}\label{picture}
\foreach \x/\y in {-1.2/2, 0.4/2 , 1.3/2 , 2.5/2 , 3.6/2 , 5/2 } {\node at (\x,\y) [circle,fill,inner sep=1.3pt] {};}
\foreach \x/\y in {-1.2/2, 0.4/2 , 1.3/2 , 2.5/2 , 3.6/2 , 5/2 } {\node at (\x,\y) [circle,fill,inner sep=1pt] {};}
\foreach \x/\y in {-0.7/2,-0.4/2.4,4.2/2.4,4.6/2} {\node at (\x,\y) [circle,fill,inner sep=2pt]{} ;}
\foreach \x/\y in {1.4/0.7,2/0.7,2.6/0.7} {\node at (\x,\y) [circle,fill,inner sep=1.8pt]{} ;}

\node at (-1,2) [anchor=north east] {$1$};
\node at (0.6,2) [anchor=north east] {$2$};
\node at (1.4,2) [anchor=north east] {i};
\node at (2.9,2) [anchor=north east] {i+1};
\node at (4,2) [anchor=north east] {n-1};
\node at (5.4,2) [anchor=north east] {n};
\node at (-0.6,2.7) [anchor=north east] {$\color{red}e_1$};
\node at (5.6,2.7) [anchor=north east] {$\color{red}e_{n-1}$};
\node at (-1,2.6) [anchor=north east] {\color{green}$f_1$};
\node at (6,2.4) [anchor=north east] {$\color{green}f_{n-1}$};
\node at (0.6,3) [anchor=north east] {\huge{\color{red}$\F_e$}};
\node at (3.5,3) [anchor=north east] {\huge{\color{green}$\DD_f$}};
\node at (0.8,6) [anchor=north east] {\huge{\color{red}$\tilde{\F}_e$}};
\node at (2.6,5.8) [anchor=north east] {\huge{\color{green}$\tilde{\DD}_f$}};
\node at (-2.5,2) [anchor=north east] {\large{$C_{n,m}$}};
\node at (-2.5,6) [anchor=north east] {\large{$\tilde{C}_{n,m}$}};

\draw [very thick,black!50!red,->](-1.2,2) to[out=60,in=110] (0.4,2);
\draw [very thick,black!50!red,-](-1.2,2) to (0.4,2);
\draw [very thick,black!50!red,->](3.6,2) to[out=60,in=110] (5,2);
\draw [very thick,black!50!red,-](3.6,2) to (5,2);
\draw [very thick,black!50!green,-](-0.7,1) to (-0.7,3);
\draw [very thick,black!50!green,-](-0.4,0.9) to (-0.4,3.1);
\draw [very thick,black!50!blue,->](-2.2,4.5) to[out=-120,in=110] (-2.2,3);
 \draw[very thick,black!50!red] (0.5, 5.5) circle (0.8);
 \draw[very thick,black!50!green] (2.2, 5.16) circle (0.9);
\draw [very thick,black!50!green,-](4.6,1) to (4.6,3);
\draw [very thick,black!50!green,-](4.2,0.9) to (4.2,3.1);
\
\draw (2,2) ellipse (4.2cm and 1.3cm);
\draw (2,5.4) ellipse (4cm and 1.2cm);
\node (d1) at (1.6,0.8) [anchor=north east] {$d_1$};
\node (d2) at (2.2,0.8) [anchor=north east] {$d_2$};
\node (dn) at (2.8,0.8) [anchor=north east] {$d_m$};

\node (dn) at (3.3,1.5) [anchor=north east] {\Large \color{blue}$\delta_f$};
\draw [very thick,black!50!blue,->][in=-30,out=-190](1.4,0.7) to (-0.7,1.5);
\draw [very thick,black!50!blue,->][in=-30,out=-190](2,0.7) to (-0.4,1.6);
\draw [very thick,black!50!blue,->][in=220,out=0](2.6,0.7) to (4.6,1.6);

\node (dn) at (1.8,1.5) [anchor=north east] {\Large \color{orange}$\gamma_e$};
\draw [very thick,orange,->][in=-60,out=-190](1.4,0.7) to  (-0.2,2);
\draw [very thick,orange,->][in=-70,out=-200](2,0.7) to (-0.1,2.4);
\draw [very thick,orange,->][in=-90,out=0](2.6,0.7) to (4,2.3);

\end{tikzpicture}
\end{center}

\begin{definition}({\bf Multiforks})\cite{Big3},\cite{Ito}

{\bf{1) Submanifolds}} 

Let $e=(e_1,...,e_{n-1}) \in E_{n,m}$ as in the definition \ref{D:index}. 
Then, for each $e$ we will construct an $m$-dimensional submanifold in $\tilde{C}_{n,m}$, which give a homology class in the Borel Moore homology of the covering. We present this construction below. 

For each $i \in \{ 1,...,n-1 \} $, consider $e_i$-disjoint horizontal segments in $D_n$, between $p_i$ and $p_{i+1}$ (which meet just at their boundary), as in the picture \ref{picture}. Denote those segments by $I^{e}_1,...,I^{e}_{e_1},..., I^{e}_{m}$. Also, for each $k \in \{ 1,...,m \} $, choose a vertical path ${\gamma}^{e}_k$ between the segment $I^{e}_k$ and $d_k$.

Then, each segment is a map $I^{e}_i: (0,1)\rightarrow D_n$. Since these segments are disjoint, their product gives a map as follows:
$$I^{e}_1\times...\times I^{e}_m: (0,1)^m\rightarrow D^{\times m}_n\setminus \{x=(x_1,...,x_m)| x_i=x_j\}$$

Let the projection defined by the quotient with respect to the Sym$_m$-action:
 $$\pi_m : D^{\times m}_n \setminus \{x=(x_1,...,x_m)| x_i=x_j\}) \rightarrow C_{n,m}.$$ 
Composing the two previous maps we obtain an $m$-dimensional disc: 
$$ \ \ \ \F_{e}: \D^m(=(0,1)^m)\rightarrow C_{n,m}$$
$$\F_e=\pi_m \circ (I^{e}_1 \times ... \times I^{e}_{m}).$$

{\bf 2) Base Points} The paths ${\gamma}^{e}_1,...,{\gamma}^{e}_m$ which start on the segments $I^{e}_{1},...,I^{e}_{m}$ and end towards the base points $d_1,..,d_m$ will help us to lift the submanifold $\F_{e}$ towards the covering $\tilde{C}_{n,m}$. 

Let ${\bf d} \in C_{n,m} $ be the point defined by the $m$-uple $(d_1,...,d_m)$.
Then, let us consider a lift of this point $\tilde{ \bf d} \in \pi^{-1} ({ \bf d}) $. The product of the paths $\gamma^e_k$, will define a path in the configuration space.
Let $$\gamma^e := \pi_m \circ (\gamma^{e}_1, ..., \gamma^{e}_m ) : [0,1] \rightarrow C_{n,m}.$$
Consider $\tilde{\gamma}^e$ to be the unique lift of the path $\gamma^e$ such that $$\tilde{\gamma}^e:  [0,1]^m \rightarrow \tilde{C}_{n,m}$$
$$\tilde{\gamma}^e(0)=\tilde{\bf d}.$$
 
{\bf 3) Multiforks}
Let $\tilde{\F}_{e}$ be the unique lift of the submanifold $\F_e$ to the covering, which passes through the point $\tilde{\gamma}^e (1)$:  
$$\tilde{\F}_{e}: \D^m(=(0,1)^m)\rightarrow \tilde{C}_{n,m}$$ 
$$\tilde{\gamma}^e (1) \in \tilde{\F}_{e}.$$
Then $\tilde{\F}_{e}$ gives a well defined a class in the Borel-Moore Homology $$[\tilde{\F}_{e}] \in H^{\text{lf}}_m(\tilde{C}_{n,m}, \Z)$$
 called the multifork corresponding to the element $e \in E_{n,m}$.
\end{definition}

The Lawrence representation will be a certain subspace of this Borel-Moore homology of the covering, spanned by these multiforks. More precisely we have the following definition:
\begin{definition}\label{multif}
Consider the subspace: 
\begin{equation}
\HH_{n,m}:= <[\tilde{\F}_e] \ | \  e\in E_{n,m}>_{\Z[x^{\pm},d^{\pm}]} \subseteq H^{\text{lf}}_m(\tilde{C}_{n,m}, \Z) \ \ \ \ \ \ \ 
\end{equation}
Denote by $\mathscr B_{\HH_{n,m}}:= \{  [\tilde{\F}_e] \ | \  e\in E_{n,m} \}$ the set of all multiforks.
\end{definition}
\begin{proposition}
From (\cite{Ito},Prop 3.1) $\HH_{n,m}$ is a free module over $\Z[x^{\pm}, d^{\pm}]$ of dimension $d_{n,m}$ and $\mathscr B_{\HH_{n,m}}$ describes a basis for it, called multifork basis. 
\end{proposition}
As we have seen, the cardinal of $E_{n,m}$ is known (proposition\ref{R:dd}), so we have: $$\text{rank}(\HH_{n,m})= d_{n,m}={{n+m-2} \choose {m}}.$$

\subsection{Braid group action}
In this part we discuss the relation between this subspace, which lives in the homology of the covering of the configuration space and the braid group action on the punctured disc. 
 It is known (\cite{kasstur},\text{ chap. I.6}) that:
$$B_n= MCG(D_n)=Homeo^+(D_n, \partial)/ \text{isotopy}. \ \ \ $$
Then $B_n \curvearrowright C_{n,m}$ and it induces an action $$B_n \curvearrowright \pi_1(C_{n,m}).$$

\begin{remark}
1) Using the properties of the specific local system that we work with, it follows that the braid group action onto the configuration space induces a well defined action of on the covering $\tilde{C}_{n,m}$.

2)We are interested to study the homology of this covering(\ref{R:mm}). One can check that the action $B_n\curvearrowright \tilde{C}_{n,m}$ commutes with the action of the Deck transformations, generated by $\{x,d\}$.

Moreover, it can be shown that $\phi$ is the finest abelian local system such that the induced action of the braid group on the corresponding covering space commutes with the action given by deck transformations. 
\end{remark}
\begin{corollary} From the previous remarks, one can conclude that there is a well defined action as follows:
$$B_n \curvearrowright H^{\text{lf}}_m(\tilde{C}_{n,m}, \Z) \ (\text{ as a }\Z[x^{\pm}, d^{\pm}]-\text{module}).$$
 \end{corollary}
 
 \begin{definition} (\cite{Ito} Prop 3.1) (Lawrence representation)
 
The subspace $\HH_{n,m}\subseteq  H^{\text{lf}}_m(\tilde{C}_{n,m}, \Z)$ is invariant under the action of $B_n$.

The braid group action onto the homology $\HH_{n,m}$ written in the multifork basis $\mathscr B_{\HH_{n,m}}$, leads to a representation which is called the Lawrence representation:
$$l_{n,m}: B_n\rightarrow GL(d_{n,m}, \Z[x^{\pm}, d^{\pm}]) \  \ (=End(\HH_{n,m}, \Z[x^{\pm}, d^{\pm}])).$$
  \end{definition}

  \section{Blanchfield pairing}\label{4}
In this section, we will present a non-degenerate duality between the Lawrence representation $\HH_{n,m}$ and a "dual" space, which we will denote by $\HH^{\partial}_{n,m}$. This dual space lives in the homology of the covering relative to its boundary. Using this form, we will be able to express any element in the dual of $\HH_{n,m}$, as certain geometric pairing, using elements from the dual space. This property will play an important role in the homological model from Section \ref{s}.
\subsection{Dual space}
We start by defining a certain subset in the homology of the covering of the configuration space relative to its boundary $H_m(\tilde{C}_{n,m}, \Z; \partial)$, by specifying a generating set. We will think that as a dual set to the multifork basis.
\begin{definition}({\bf Barcodes})\cite{Big3}
 
{\bf{1) Submanifolds}} 
Let $e=(e_1,...,e_{n-1}) \in E_{n,m}$. 
For each such $e$, we will define an $m$-dimensional submanifold in $\tilde{C}_{n,m}$, which will give a homology class in $H_{m}(\tilde{C}_{n,m}, \Z; \partial)$. 

For each $i \in \{ 1,...,n-1 \} $, consider $e_i$-disjoint vertical segments in $D_n$, between $p_i$ and $p_{i+1}$ as in the picture \ref{picture}. Denote those segments by $J^{e}_1,...,J^{e}_{e_1},..., J^{e}_{m}$. Also, for each $k \in \{ 1,...,m \} $, we choose a vertical path ${\delta}^{e}_k$ between the segment $J^{e}_k$ and the base point $d_k$.

Each of these segments gives a map $J^{e}_i: [0,1]\rightarrow D_n$. Then their product leads to a map towards the configuration space as follows. Consider the map:
$$J^{e}_1\times...\times J^{e}_m: [0,1]^m\rightarrow D^m_n\setminus \{x=(x_1,...,x_m)| x_i=x_j\}.$$
Projecting onto the configuration space  using $\pi_n$, we obtain a submanifold:
 $$ \DD_{e}: \bar{\D}^m(=[0,1]^m), \partial \bar{\D}^m )\rightarrow (C_{n,m},\partial C_{n,m} ).$$

{\bf 2) Base Points}: As in the case of multiforks, the paths to the base point ${\bf d} \in C_{n,m} $ help us to lift the submanifold $\DD_{e}$ towards the covering $\tilde{C}_{n,m}$. 
Consider the path in the configuration space:
$$\delta^e := \pi_m \circ (\delta^{e}_1, ..., \gamma^{e}_m ) : [0,1] \rightarrow C_{n,m}.$$
Define $\tilde{\delta}^e$ to be the unique lift of the path $\delta^e$ such that: 
$$\tilde{\delta}^e:  [0,1] \rightarrow \tilde{C}_{n,m}$$ 
$$\tilde{\delta}^e(0)=\tilde{\bf d}.$$

{\bf 3) Barcodes} Consider $\tilde{\DD}_{e}$ to be the unique lift of $\DD_e$ to the covering which passes through $\tilde{\delta}^e (1)$:
$$\tilde{\DD}_{e}: \D^m \rightarrow \tilde{C}_{n,m}$$
$$\tilde{\delta}^e (1) \in \tilde{\DD}_{e}.$$ 
 
Then $\tilde{\DD}_{e}$ defines a class in the homology relative to the boundary $$[\tilde{\DD}_{e}] \in H_m(\tilde{C}_{n,m}, \Z; \partial)$$
called the barcode corresponding to the element $e \in E_{n,m}$.
\end{definition}

\begin{definition} (The "dual" representation)\\ \label{barc}
Let the subspace generated by all the barcodes:
$$\HH^{\partial}_{n,m}:= <[\tilde{\DD}_e] \ | \  e\in E_{n,m}>_{\Z[x^{\pm},d^{\pm}]} \subseteq H_m(\tilde{C}_{n,m}, \Z; \partial). \ \ \ \ \ \ \  \ref{R:mm}$$
We call $\HH^{\partial}_{n,m}$ as the "dual" representation of $\HH_{n,m}$. Also, let us consider the set given by all barcodes:
 $$\mathscr B_{\HH^{\partial}_{n,m}}:= \{ [\tilde{\DD}_e] | e \in E_{n,m}   \}.$$
\end{definition}
\begin{remark}
We do not know yet that $\mathscr B_{\HH^{\partial}_{n,m}}$ is a basis for $\HH^{\partial}_{n,m}$, but we will prove this in the next section, using a pairing between $\HH_{n,m}$ and $\HH^{\partial}_{n,m}$.

\end{remark}

\subsection{Graded Intersection Pairing}

In this part, we will describe how the Borel-Moore homology and the homology relative to the boundary of $\tilde{C}_{n,m}$ are related by an intersection form. More precisely we are interested to define a Blanchfield type pairing between $\HH_{n,m}$ and $\HH^{\partial}_{n,m}$. 

We will present a Poincar\'e-Lefschetz type duality, which uses the middle dimensional homologies of the covering with respect to differents parts of the its boundary. We use the space $\tilde{C}_{n,m}$ and think about its boundary as having two parts. The first part, the "boundary at infinity", contains the multi-points in $\tilde{C}_{n,m}$ such that either one of their components projects on $D_n$ "close to a puncture" or where two components get very close one to another after the projection. The second part, is the actual boundary and contains the multi-points for which there exists a component which projects onto the boundary of $D_n$.

The Borel-Moore homology of $\tilde{C}_{n,m}$, can be thought as the homology with respect to the first boundary from above, relative to infinity. The second homology that we will use will be the homology with respect to the boundary of $\tilde{C}_{n,m}$, as described above.

We will follow \cite{Big2}, \cite{Big3}, especially the way of computing the pairing in the case where the homology classes are given by geometric submanifolds.
Let us take two homology classes $[\tilde{M}] \in H^{lf}_m(\tilde{C}_{n,m}, \Z) $ and $[\tilde{N}] \in H_m(\tilde{C}_{n,m},\Z; \partial )$ which can be represented as the classes given by lifts of two m-dimensional submanifolds in the base space $M, N \subseteq C_{n,m}$. The idea is to fix the lift of the second submanifold $\tilde{N}$ in the covering and act with all deck transformations onto the first submanifold $\tilde{M}$. Each time, we will count the geometric intersection between the two submanifolds multiplied with the coefficient given by the element from the deck group. Recall that the local system is given by
\begin{equation*}
\begin{cases}
\phi: \pi_1(C_{n,m})\rightarrow \Z \oplus \Z\\
 Deck(\tilde{C}_{n,m})=\Z \oplus \Z.
 \end{cases}
 \end{equation*}
\begin{definition}\cite{Big}(Graded intersection)\\
Let $F \in H^{lf}_m(\tilde{C}_{n,m}, \Z) $ and $G \in H_m(\tilde{C}_{n,m},\Z; \partial )$. Suppose that there exist $M,N \subseteq C_{n,m}$ transverse submanifolds of dimension $m$ which intersect in a finite number of points such that there exist lifts in the covering $\tilde{M}, \tilde {N}$ with 
$$F=[\tilde{M}] \ \ \ \text{and} \ \ \  G=[\tilde{N}].$$ Then the graded intersection between the submanifolds $\tilde M$ and $\tilde{N}$ is defined by the formula:
$$<<\tilde M,\tilde N>>:= \sum_{(u,v) \in \Z \oplus \Z} (x^ud^v \curvearrowright \tilde{M} \cap \tilde{N}) \cdot x^ud^v \in \Z[x^{\pm}, d^{\pm}]$$ where $( \cdot \cap \cdot)$ means the geometric intersection number between submanifolds.
\end{definition}
\begin{remark}
For any $\varphi \in Deck(\tilde{C}_{n,m})$:
$$\varphi \tilde{M} \cap \tilde{N} \subseteq \pi^{-1} (M \cap N)$$ 
This, together with the lifting property ensures that the previous sum has a finite number of non-zero terms and the graded geometric intersection between $\tilde{M}$ and $\tilde{N}$ is well defined.
\end{remark}
In the sequel, we will see that even if a priori the graded intersection between $\tilde{M}$ and $\tilde{N}$ is defined in the covering space $\tilde{C}_{n,m}$, it can be actually computed in the base, using $M$ and $N$ and the local system for coefficients. More specifically, the intersection pairing is described as a sum parametrised by all intersection points of $M$ and $N$ in $C_{n,m}$, where for each point counts together with a coefficient which is prescribed by the local system. 
\begin{proposition}\label{P:phi}
Let $x \in M \cap N$. Then there exists an unique $\varphi_x \in Deck(\tilde{C}_{n,m})$ such that 
\begin{equation}\label{E:1} 
(\varphi_x\tilde{M} \cap \tilde{N})\cap \pi^{-1}(x) \neq \varnothing.  
\end{equation}
\end{proposition}
\begin{proof}
The fact that $N$ is a submanifold guarantees that $\forall y \in N$:
$$\text{card }|\tilde{N}\cap \pi^{-1}(\{ y \})|=1.$$ 
Using the same property for $M$ as well, let us denote:
\begin{equation*}
\begin{cases}
\tilde{y}_{\tilde{N}}:=\tilde{N}\cap \pi^{-1}(\{ y \})\\
\tilde{y}_{\tilde{M}}:=\tilde{M}\cap \pi^{-1}(\{ y \}).
\end{cases}
\end{equation*}
Then it follows that:
$$\left(\varphi\tilde{M} \cap\tilde{N}\right) \cap \pi^{-1}\left(\{ x \}\right) \neq 0 \text{ iff } \tilde{x}_{\tilde N} \in \varphi \tilde{M}.$$
This remark shows that if $\varphi$ satisfies the condition given in relation \ref{E:1}, then 
$$\varphi\left(\tilde{x}_{\tilde M}\right)=\tilde{x}_{\tilde N}.$$
Using the properties of Deck transformations, this condition gives a characterisation for an unique $\varphi_x$.
\end{proof}
\begin{remark}
The last proposition shows that the intersection points between all the translations of $\tilde{M}$ by the deck transformations and $\tilde{N}$ are actually parametrised by the intersection points between $M$ and $N$:
$$ \bigcup_{\varphi \in Deck(\tilde{C}_{n,m})}(\varphi \tilde{M} \cap \tilde{N}) \longleftrightarrow M \cap N.$$
\end{remark}
{\bf Computation}\label{ComBl} 
For this part, we follow \cite{Big}. Let us fix a basepoint $d \in C_{n,m}$ and $\tilde{d}\in \pi^{-1}(d)$. 
From the last discussion, we notice that in order to compute the pairing $<<\tilde{M}, \tilde{N}>>$, it is enough to consider a sum parametrised by the set $M \cap N$ and see which is the corresponding coefficient for each intersection point. Let $x \in M \cap N$ and $\varphi_x \in Deck(\tilde{C}_{n,m})$ as in \ref{P:phi}. Denote by 
$$\tilde{x}=(\varphi_x\tilde{M} \cap \tilde{N})\cap \pi^{-1}(x).$$
Now we will describe $\varphi_x$ using just the local system $\phi$ and the point $x$.
We notice that we have the same sign of the intersection in the covering and in the base:
$$(\varphi_x \tilde{M} \cap \tilde{N})_{\tilde{x}}= (M \cap N)_x$$ 
Denote this sign by $c_x$. Suppose that we have two paths 
\begin{equation*}
\gamma_M, \delta_N: [0,1]\rightarrow C_{n,m}
\end{equation*}
 such that if we take their unique lifts which start in $\tilde{d}$, denoted by 
$$\tilde{\gamma}_M, \tilde{\delta}_N:[0,1] \rightarrow \tilde{C}_{n,m}$$ we have the following properties: 
\begin{equation*}
\begin{cases}
\gamma_M(0)=d; \ \ \gamma_M(1) \in M; \ \ \tilde{\gamma}_M(1) \in \tilde{M}\\ 
\delta_N(0)=d; \ \ \ \delta_N(1) \in N; \ \ \ \tilde{\delta}_N(1) \in \tilde{N}.
\end{cases}
\end{equation*}
After that, let us consider two paths $\hat{\gamma}_M,\hat{\delta}_N: [0,1]\rightarrow C_{n,m}$ such that 
\begin{equation*}
\begin{cases}
Im(\hat{\gamma}_M)\subseteq M; \ \ \hat{\gamma}_M(0)=\gamma_M(1); \ \ \hat{\gamma}_M(1)=x\\
Im(\hat{\delta}_N)\subseteq N; \ \ \ \ \hat{\delta}_N(0)=\delta_N(1); \ \ \ \hat{\delta}_N(1)=x.
\end{cases}
\end{equation*}
We consider the following loop: $$l_x:=\delta_N \hat{\delta}_N \hat{\gamma}^{-1}_M\gamma_M^{-1}.$$
\begin{proposition}Folllowing \cite{Big}, one has that:
$$\varphi_x= \phi(l_x).$$
\end{proposition}
\begin{corollary} The pairing between $\tilde M$ and $\tilde{N}$ can be computed using just the submanifolds in the base space $C_{n,m}$ and the local system as follows:
\begin{equation}\label{E:2}
<<\tilde{M},\tilde{N}>>= \sum_{x \in M \cap N} c_x \phi(l_x )\in \Z[x^{\pm}, d^{\pm}].
\end{equation}
\end{corollary}
This pairing $<< , >>$ can be defined in a similar way for homology classes $F \in H^{lf}_m(\tilde{C}_{n,m}, \Z) $ and $G \in H_m(\tilde{C}_{n,m},\Z; \partial )$ that can be represented as linear combinations of homology classes of lifts of submanifolds of the type that we described above, requiring the condition concerning a finite set of intersection points.  
\begin{lemma}(\cite{Big3}(6.2))
The paring $<<F,G>>$ does not depend on the choice of representatives for the homology classes, so it is well defined at the level of homology.
\end{lemma}

\subsection{Pairing between $\HH_{n,m}$ and $\HH^{\partial}_{n,m}$}
\begin{definition}
Let us consider the Blanchfield pairing:
$$< , >: \HH_{n,m} \otimes \HH^{\partial}_{n,m}\rightarrow \Z[x^{\pm},d^{\pm}]$$
$$<[\tilde{\F}_e],[\tilde{\DD}_f]>=<<\tilde{\F}_e,\tilde{\DD}_f>>.$$
This leads to a sesquilinear form (with respect to the transformations \\ $x\leftrightarrow x^{-1}, d\leftrightarrow d^{-1}$). 
\end{definition}
\begin{lemma}
For any $e,f \in E_{n,m}$, the pairing has the following form:
$$<[\tilde{\F}_e],[\tilde{\DD}_f]>= p_e \cdot \delta_{e,f}$$ 
where $p_e\in \N[d^{\pm}]$ and $p_e \neq 0$ with a non-zero constant term.
\end{lemma}
\begin{proof}
Since we are working in the configuration space, we remark that:
$$\F_e\cap \DD_f= \emptyset \text{ if } e \neq f.$$ Let us fix a partition  $e \in E_{n,m}$. Following equation \ref{E:2}, we have that:
\begin{equation}
<<\tilde{\F}_e,\tilde{\DD}_e>>= \sum_{x \in \F_e \cap \DD_e} c_x \phi(l_x )\in \Z[x^{\pm}, d^{\pm}]. \label{form1}
\end{equation}
Secondly, we notice that the previous intersection can be computed using the intersections between the submanifolds that lead to $\F_{e}$ and $\DD_{e}$ "supported" between punctures $i$ and $i+1$, in the following manner: 
\begin{equation}
<<\tilde{\F}_e,\tilde{\DD}_e>>= \prod^{n-1}_{i=1}<<\tilde{\F}_{e_i},\tilde{\DD}_{e_i}>>\label{form2}
\end{equation}
where $\F_{e_i}:= \F_{(0,0,...,e_i,...,0)}$ and $\DD_{e_i}:=\DD_{(0,0,...,e_i,...,0)}$.

Now we compute the pairing $<<\tilde{\F}_{e_i},\tilde{\DD}_{e_i}>>$.
We notice that each intersection point $x \in \F_{e_i}\cap\DD_{e_i}$ is characterised by an $e_i$-uple which pairs a horizontal line from the multifork with a vertical line from the barcode. In other words, $x$ is determined by a permutation on the grid, which we denote by $$\sigma_x \in S_{e_i}.$$ It follows:
\begin{equation}\label{form3}
<<\tilde{\F}_{e_i},\tilde{\DD}_{e_i}>>=\sum_{\sigma \in S_{e_i}} c_{\sigma} \phi(l_{\sigma} ).
\end{equation}
In this formula, $c_{\sigma}$ is the coefficient which counts whether ${\F_{e_i}}$ and  ${\DD_{e_i}}$ have a positive or a negative intersection in the multipoint: 
$$x= \left( x_{(1,\sigma(1))},...,x_{(e_i, \sigma(e_i))} \right).$$
In order to compute this sign, we use that the configuration space on the disc is orientable.
Let us consider $\mathscr R=\{ v^1, v^2 \}$ to be the standard base for the tangent space of the disc.
Let $c=(c_1,...,c_m) \in C_{n,m}$ and a tangent vector in this point $w$. We will define the orientation of $w$ by writing it into the form 
$(w^1_{c_1},...,w^m_{c_m}, w^2_{c_1},...,w^2_{c_m})$ and see if written in the standard basis $\mathscr R$ has the same sign or not as the vector 
$(v^1_{c_1},....,v^1_{c_m},v^2_{c_1},...,v^2_{c_m})$.
This is well defined in the configuration space, because we are working on a manifold of even dimension, so if we change the order of points by a transposition, we will have to modify the matrix by an even number negative signs. 

Following this recipe, we see that $c_\sigma$ is the sign of the tangent vector $v_{\sigma}$ obtained by taking the tangent vectors at the multiforks followed by the tangent vectors at the barcode: 
$$v_{\sigma}=(v^1_{x_{(1,\sigma(1))}},....,v^1_{x_{(e_i,\sigma(e_i))}},v^2_{x_{(1,\sigma(1))}},...,v^2_{x_{(e_i,\sigma(e_i))}}).$$
Here, we used that all segments of the multifork are oriented in the same way, and also, that all parts of the barcode have the same orientation. We conclude that $$c_{\sigma}=1.$$

Now we will look at the polynomial part from the graded intersection. Following the previous description of computation, for any $k\in \{1,...,m\}$ let: 
$$\hat{\gamma}^e_k \subseteq I_k \text { such that } \hat{\gamma}^e_k(0)=\gamma^e_k(1); \ \hat{\gamma}^e_k(1)=x_{(k, \sigma(k))} $$ 
$$\hat{\delta}^e_k \subseteq J_k \text { such that } \hat{\delta}^e_k(0)=\delta^e_k(1); \ \hat{\delta}^e_k(1)=x_{(k, \sigma(k))} $$ 

Let us denote $a_i:=e_1+...+e_{i-1}$ and the following paths in the configuration space of $e_i$ points in the punctured disc:
$$\Gamma_{e_i}:= \left( \gamma^e_{a_i+1},...,\gamma^e_{a_i+e_i} \right)  \ \ \ \ \hat{\Gamma}_{e_i}:= \left( \hat{\gamma}^e_{a_i+1},...,\hat{\gamma}^e_{a_i+e_i} \right)       $$
$$\Delta_{e_i}:= \left( \delta^e_{a_i+1},...,\delta^e_{a_i+e_i} \right)  \ \ \ \ \hat{\Delta}_{e_i}:= \left( \hat{\delta}^e_{a_i+1},...,\hat{\delta}^e_{a_i+e_i} \right)       $$
Then, using relation \ref{E:2}, the loop corresponding to the point $x$ ( given by the permutation $\sigma$) has the following form:
$$l_{\sigma}=\Delta_{e_i} {\hat{\Delta}}_{e_i} { {\hat{\Gamma}_{e_i}} }^{-1}\Gamma_{e_i}^{-1} \subseteq Conf_{e_i}(\D_n).$$ 

The first remark is that following the loop $l_{\sigma}$ using the picture \ref{picture}, we see that neither of its components go around any of the punctures. So, the variable $x$ from the local system will not appear in the evaluation $\phi(l_\sigma)$.

Secondly, for $\sigma=Id$ the path $l_{Id}$ is the union of trivial loops, so $$\phi(l_{Id})=1.$$

Putting the previous remarks together in the formula \ref{form3}, we conclude that:
$$<<\tilde{\F}_{e_i},\tilde{\DD}_{e_i}>> \ \in \N[d^{\pm}]$$ and it has a nontrivial free term. Combining this with the computations from \ref{form1} and \ref{form2}, we conclude that 
$$<[\tilde{\F}_e],[\tilde{\DD}_e]> \in \N[d^{\pm}]$$ with a non trivial free part, which concludes the proof.  
\end{proof}
\begin{remark}\label{R:2} 
This computation shows that all the polynomials $\{p_e \mid e\in E_{n,m}\}$ are non-zero divisors in $\Z [x^{\pm}, d^{\pm}]$.
\end{remark}
This leads to the following result:
\begin{lemma}\label{L:bar}
The family of barcodes $\{ [\tilde{\DD}_e] \mid e\in E_{n,m} \} $ is linearly independent and it gives a basis for $\HH^{\partial}_{n,m}$. 
\end{lemma}
\begin{proof}
Let $\alpha_1,...,\alpha_{d_{n,m}} \in \Z[x^{\pm}, d^{\pm}]$ and suppose that:
\begin{equation*}
\sum_{i=1}^{d_{n,m}} \alpha_i [\tilde{\DD}_{e_i}] = 0 \in \HH^{\partial}_{n,m}.
\end{equation*}

Let us fix $j \in \{1,...,d_{n,m}\}$. The pairing with the multifork $[\tilde{\F}_{e_j}]$ leads to the following:
\begin{equation*}
<[\tilde{\F}_{e_j}], \sum_{i=1}^{d_{n,m}} \alpha_i [\tilde{\DD}_{e_i}]> = <[\tilde{\F}_{e_j}], \alpha_{e_j} [\tilde{\DD}_{e_j}]>= \alpha_j \cdot p_{e_j}=0.
\end{equation*}
Following remark \ref{2}, all coefficients are not zero divisors in $\Z[x^{\pm}, d^{\pm}]$, and we conclude that $$\alpha_j=0, \forall j \in \{1,...,d_{n,m} \}.$$ 
\end{proof}
\begin{notation}
The set $\mathscr B_{\HH^{\partial}_{n,m}}$ will be called the barcodes basis for $\HH^{\partial}_{n,m}$.
\end{notation}
\begin{remark}\label{L:mult}
By an analog argument, we re-obtain also a proof for the fact that the multiforks $ \{ [\tilde{\F}_e] \mid e \in E_{n,m} \}$ are linearly independent in $H^{lf}_m(\tilde{C}_{n,m},\Z)$(\cite{Ito2}-3.1).
\end{remark}
\begin{remark}
From the previous computation, we get the matrix of the graded intersection pairing $< , >$ in the bases of multiforks $ \mathscr B_{\HH_{n,m}}$ and barcodes $\mathscr B_{\HH^{\partial}_{n,m}}$:
\[
  M_{<,>}=
  \left[{\begin{array}{cccc}
   p_{e_1} &0  & ... & 0\\
   0 &p_{e_2} & 0 \ \ \ \  ...  & 0\\
   \\
   0 & ... &  0 &p_{e_{d_{n,m}}}
   \end{array} } \right]
\] 

(where $p_1,...,p_{e_{d_{n,m}}}\in \Z[x^{\pm}, d^{\pm}]$ are all non-zero divisors).\\ 
\end{remark}
The form of this pairing, leads to the following property.
\begin{corollary}
The Blanchfield pairing is a non-degenerate sesquilinear form: 
$$< , >: \HH_{n,m} \otimes \HH^{\partial}_{n,m}\rightarrow \Z[x^{\pm},d^{\pm}].$$ 
\end{corollary}

\subsection{Specialisations}
Our aim is to describe the coloured Jones polynomials in a homological way. For this purpose, our starting point is the deep connection proved by Kohno, that relates quantum representations of the braid groups and certain specializations of the Lawrence representations.    
In this part we will focus on those specializations of the Lawrence representation which are used in Kohno's Theorem. Our aim is to obtain non-degenerate itersection forms between these specialisations.    
\begin{definition} (Specialisation of coefficients) Let $\lambda=N-1 \in \N$ be a parameter.\\
Consider the specialization of the coefficients $\psi_{\q, \lambda}$ defined by:
$$\psi_{\lambda}: \Z[x^{\pm},d^{\pm}]\rightarrow \Z[q^{\pm}]$$ 
$$\psi_{\lambda}(x)= \qs^{2 \lambda}, \ \ \psi_{\lambda}(d)=-\qs^{-2}.$$
 \end{definition}
\begin{definition} 
 The specialized Lawrence representation and its dual are given by:
$$\HH_{n,m}|_{\psi_{\lambda}}=\HH_{n,m}\otimes_{\psi_{\lambda}}\Z[q^{\pm}]=<[\tilde{\F}_{e}] \mid e \in E_{n,m}]>_{\Z[q^{\pm}]}$$
and these multiforks define a basis of $\HH_{n,m}|_{\psi_{ \lambda}}$ over $\Z[q^{\pm}]$, using \ref{L:mult}
$$\HH^{\partial}_{n,m}|_{\psi_{\lambda}}=\HH^{\partial}_{n,m}\otimes_{\psi_{\lambda}}\Z[q^{\pm}]=<[\tilde{\D}_{e}] \mid e \in E_{n,m}]>_{\Z[q^{\pm}]}$$
and these barcodes will define a basis of $\HH_{n,m}|_{\psi_{\lambda}}$ over $\Z[q^{\pm}]$, using \ref{L:bar}.
 \end{definition}
\begin{definition}
Let us consider the specialised Blanchfield pairing, obtained from the generic pairing $< , >$ by  specialising its coefficients using $\psi_{\lambda}$:
$$< , >|_{\psi_{\lambda}}: \HH_{n,m}|_{\psi_{\lambda}} \otimes \HH^{\partial}_{n,m}|_{\psi_{\lambda}}\rightarrow \Z[q^{\pm}]$$
$$<[\tilde{\F}_e],[\tilde{\DD}_f]>|_{\psi_{\lambda}}= \psi_{\lambda}(p_e) \cdot \delta_{e,f}.$$
\end{definition}

\begin{remark}
We notice that $$ \{ p_{e} \mid e \in E_{n,m} \} \cap Ker (\psi_{\lambda})=  \varnothing .$$ 

1) At this point we see that the choice of barcodes on the dual side of $\HH_{n,m}$ has an important role. The geometric intersection pairing between multiforks and these barcodes, has a corresponding matrix $M_{<,>}$ which is diagonal with non-zero polynomials $p \in \N[d^{\pm}]$ on the diagonal. This fact, ensures that these polynomials become non-zero elements in $\Z[q^{\pm}]$ through the specialization $\psi_{ \lambda}$. 

2) It would be interesting to compare this situation with a Bigelow-Lawrence type situation, where we would use dual-noodles (noodles with multiplicities) instead of barcodes. In that case, the generic pairing will have as coefficients on the diagonal, polynomials $p \in \Z[x^{\pm}, d^{\pm}]$, which have both variables and moreover they would have $\Z$ coefficients not only $\N$ coefficients. 

Concerning our aim,  for our topological model for the coloured Jones polynomial $J_N(L,\q)$, we will use the specialisation $\psi_{N-1}$ with natural parameter $\lambda=N-1 \in \N$.
In this case, some of these diagonal polynomials for the noodle case might become zero through the specialisation $\psi_{N-1}$  because this change of coefficients impose essentially the relations
\begin{equation*}
\begin{cases}
x=-d^{-\lambda}\\
\lambda=N-1 \in \N.
\end{cases}
\end{equation*}

3)An interesting question that arises from this discussion is to understand the pairing in the noodle case and to compute its kernel.
\end{remark}
\begin{corollary}
The form $< , >|_{\psi_{\lambda}}$ is sesquilinear and non-degenerate over $\Z[q^{\pm}]$.
\end{corollary}

\subsection{Dualizing the algebraic evaluation}\label{dualisation}
This part is motivated by the fact that we are interested to describe the third level of a braid closure (the union of "caps"), viewed through the Reshetikhin-Turaev functor, in a geometrical way using the geometric intersection pairing. We will see the details in the folowing section \ref{s}, but for this part the aim is to be able to understand an element of the dual of $\HH_{n,m}|_{\psi_{\lambda}}$, as a geometric intersection $<\cdot, \mathscr G>$ for some $G \in \HH_{n,m}|_{\psi_{\lambda}}$.
\begin{remark}
The pairing specialised pairing $$<,>|_{\psi_{\lambda}}: \HH_{n,m}|_{\psi_{\lambda}} \otimes \HH^{\partial}_{n,m}|_{\psi_{\lambda}}\rightarrow \Z[q^{\pm}]$$ is non-degenerate and has the matrix:
\begin{center}
\[
  M_{<,>}=
  \left[ {\begin{array}{cccc}
   \psi_{\lambda}(p_{e_1}) &0  & ... & 0\\
   0 &\psi_{\lambda}(p_{e_2}) & 0 \ \ \ \  ...  & 0\\
   \\
   0 & ... &  0 &\psi_{\lambda}(p_{e_{d_{n,m}}})
   \end{array} } \right]
\] 
\end{center}
(where $\psi_{\lambda}(p_1),...,\psi_{\lambda}(p_{e_{d_{n,m}}}) \in \Z[q^{\pm2}]$ are polynomials with non-zero free term).
\end{remark}
In particular, this shows that the diagonal coefficients of the pairing are not necessary invertible elements in $\Z[q^{\pm}]$.
\begin{problem}
 Following this remark, we notice that a priori not any element of $ \mathcal{F} \in (\HH_{n,m}|_{\psi_{\lambda}})^{*}$ can be described as a geometric intersection pairing $<\cdot, \mathscr{G}_{\mathcal{F}}>$ for some $G_{\mathcal{F}}\in \HH^{\partial}_{n,m}|_{\psi_{ \lambda}}$. This issue comes from the fact that we are working over a ring and not over a field. In order to overcome this problem, we will change the ring of coefficients from $\Z[q^{\pm}]$ to the field of fractions $\Q(q)$.
\end{problem}
We remember the specialisation $\psi_{\lambda}: \Z[x^{\pm},d^{\pm}]\rightarrow \Z[q^{\pm}]$ described by: 
$$\ \ \ \ \ \ \ \ \ \ \psi_{\lambda}(x)= \qs^{2 \lambda} \ \ \ \psi_{\lambda}(d)=-\qs^{-2}.$$

Let us consider the embedding $i: \Z[q^{\pm}]\hookrightarrow \Q(q)$ and use $\Q(q)$ as field of coefficients.
\begin{definition} (New Specialisation)\\\label{p}
1)Let the specialization $\alpha_{\lambda}: \Z[x^{\pm},d^{\pm}] \rightarrow \Q(q)$ defined by: $$\alpha_{\lambda}=i \circ \psi_{\lambda}.$$
2)Let the specialised Lawrence representations defined in a similar way as before:
$$\HH_{n,m}|_{\alpha_{\lambda}}:=\HH_{n,m}\otimes_{\alpha_{\lambda}}\Q(q)=<[\tilde{\F}_{e}] | e \in E_{n,m}]>_{\Q(q)}$$
and the multiforks define a basis of $\HH_{n,m}|_{\alpha_{\lambda}}$ over $\Q(q)$, using \ref{L:mult}
$$\HH^{\partial}_{n,m}|_{\alpha_{\lambda}}=\HH^{\partial}_{n,m}\otimes_{\alpha_{\lambda}}\Q(q)=<[\tilde{\F}_{e}] | e \in E_{n,m}]>_{\Q(q)}$$
and the barcodes define a basis of $\HH_{n,m}|_{\alpha_{\lambda}}$ over $\Q(q)$, using \ref{L:bar}.
\end{definition}
We notice that the previous specialisations are related in the following manner:
$$\HH_{n,m}|_{\alpha_{\lambda}}:=\HH_{n,m}|_{\psi_{\lambda}}\otimes_{i}\Q(q) ; \ \ \ \ \ \  \HH^{\partial}_{n,m}|_{\alpha_{\lambda}}:=\HH^{\partial}_{n,m}|_{\psi_{\lambda}}\otimes_{i}\Q(q).$$
\begin{notation}\label{N:2}  
Let us denote the corresponding change of the coefficients at the homological level by:
$$p_{\lambda}: \HH_{n,m}|_{\psi_{\lambda}} \rightarrow^{( \ \cdot \ \otimes_{i}1)} \HH_{n,m}|_{\alpha_{\lambda}}.$$ 
\end{notation}
\begin{definition}
Consider the specialised Blanchfield pairing constructed in a similar manner, by specialising the pairing $< , >$ using $\alpha_{\lambda}$:
$$< , >|_{\alpha_{\lambda}}: \HH_{n,m}|_{\alpha_{\lambda}} \otimes \HH^{\partial}_{n,m}|_{\alpha_{\lambda}}\rightarrow \Q(q)$$
$$<[\tilde{\F}_e],[\tilde{\DD}_f]>|_{\alpha_{\lambda}}= \alpha_{\lambda}(p_e) \cdot \delta_{e,f}.$$
\end{definition}
\begin{remark}
We notice that for any partition $e\in E_{n,m}$, the evaluation $\alpha_{\lambda}(p_{e})\in \Q(q)$ is a non-zero element, so it is invertible. 
This shows that $<,>|_{\alpha_{\lambda}}$ is a non-degenerate sesquilinear form.
\end{remark}
 Moreover, working on a field, we conclude that any element in the dual of the first homology group, can be described as a pairing with a fixed element from the second homology. More precisely, we obtain the following description.
\begin{corollary}
For any $\mathscr{G} \in (\HH_{n,m}|_{\alpha_{\lambda}})^{*}$, there exist a corresponding homology class $\tilde{\mathscr G}\in \HH^{\partial}_{n,m}|_{\alpha_{\lambda}} $ such that: 
\begin{equation}\label{C:duals}
\mathscr{G}=<\cdot, \tilde{\mathscr G}>|_{\alpha_{\lambda}}.
\end{equation}
\end{corollary}
\begin{remark} (Construction of geometric duals)\\
Let us start with an element $$\mathscr{G}_0 \in (\HH_{n,m}|_{\psi_{\lambda}})^{*}=\Hom(\HH_{n,m}|_{\psi_{\lambda}}, Z[q^{\pm}]).$$
Construct the following corresponding element:
\begin{equation}\label{eq:16}
\mathscr{G}:=\mathscr{G}_0\otimes Id_{\Q(q)} \in (\HH_{n,m}|_{\alpha_{\lambda}})^{*}.
\end{equation}
Then, considering the pairing with the dual element $\tilde{\mathscr G}$  of $\mathscr{G}$ given by relation \ref{C:duals}, we obtain $\mathscr{G}_0$ in a topological way:
\begin{equation}\label{eq:17} 
 \mathscr{G}_0 \otimes Id_{\Q(q)}=<\cdot, \tilde{\mathscr G}>|_{\alpha_{\lambda}}.
\end{equation} 
 \label{R:pair}
\end{remark}

\section{Identifications between quantum representations and homological representations} \label{ident}
So far, we have presented two important constructions that lead to representations of the braid group: the quantum representation and the Lawrence representation. A priori, they are defined using totally different tools, since the quantum representation comes from the algebraic world whereas the Lawrence representation has a homological description. In this section we will discuss both of them, using a result due to Kohno that relates these two representations. The identifications over two parameters were presented in \cite{Ito} based on a continuity procedure. The results from this section follow from this identification, however, we explain in more details the continuity argument. 

Let $h, \lambda \in \C$ and $\q=e^h$.
Let us consider the folowing specialisations of the coefficients defined using these complex numbers.\label{D:spec}

1) For the quantum representation $\hat{W}$ (defined over $\Z[\qs^{\pm}, s^{\pm}]$):
$$\eta_{\q,\lambda}: \Z[\qs^{\pm1}, s^{\pm1}]\rightarrow \C$$
$$\eta_{\q,\lambda}(\qs)=\q; \ \ \eta_{\q,\lambda}(s)=\q^{\lambda}.$$

2) For the Lawrence representation $\HH_{n,m}$ (defined over $\Z[x^{\pm},d^{\pm}]$):
$$\psi_{\q,\lambda}: \Z[x^{\pm},d^{\pm}]\rightarrow \C$$ 
$$\psi_{\q,\lambda}(x)= \q^{2 \lambda}; \ \ \psi_{\q,\lambda}(d)=-\q^{-2}.$$
Kohno relates these two representations, by connecting each of them with a monodromy representation of the braid group which arises using the theory of KZ-connections. We will shortly describe these relations, following \cite{Ito}. 
\subsection{ KZ-Monodromy representation}
Let the Lie algebra $sl_2(\C)$ and consider an orthonormal basis $ \{ I_{\mu} \}_{\mu}$ for its Cartan-Killing form. Denote by 
$$\Omega=\sum_{\mu}I_{\mu} \otimes I_{\mu} \in sl(2) \otimes sl(2).$$
\begin{definition} For $\lambda \in \C^*$ consider $M_{\lambda}$ to be the Verma module of $sl(2)$ defined as:
$$M_{\lambda}=<v_0,v_1,...>_{\C}$$ with the following actions:  
\begin{equation}
\begin{cases}
Hv_i=(\lambda-2i) v_i \\
Ev_i=v_{i-1} \\
Fv_i=(i+1)(\lambda-i) v_{i+1}.\\
\end{cases}
\end{equation}
\end{definition}
\begin{notation}
Let $n \in \N$ and for any $i,j \in \{1,...,n\}$ consider the endomorphism 
$$\Omega_{i,j} \in End(M_{\lambda}^{\otimes n}).$$ 
to be the action of $\Omega$ onto the $i^{th}$ and $j^{th}$ components. 
\end{notation}
The monodromy representation of the braid group, will be constructed using the complement of a hyperplane arrangement.
\begin{definition} Let us consider the spaces:
$$X_n= \C^n \setminus \left( \bigcup_{1\leq i,j \leq n} Ker (z_i=z_j) \right)$$
$$Y_n:= X_n / S_n. \ \ \ \ \ \ \ \ \ \ \ \ \ \ \ \ \ \ \ \ \ \ \ \ \ \ \ $$
\end{definition}
\begin{definition}(KZ-connection)
Let $h \in \C^*$ be a parameter.\\ 
Consider $\omega_h$ the following $1$-form defined over $Y_n$ with values in $End(M^{\otimes n}_{\lambda})$, called the KZ-connection (Knizhnik-Zamolodchikov):
$$\omega_h= \frac{h}{ \sqrt{-1} \ \pi} \sum_{1\leq i,j\leq n} \Omega_{i,j} \frac{dz_i-dz_j}{z_i-z_j}.$$
Then, this describes a flat connection with values into the trivial bundle over $Y_n$: $$Y_n \times M^{\otimes n}_{\lambda}.$$ 
\end{definition}
\begin{definition}
The monodromy of this connection leads to a representation:
$$\nu_h : B_n \rightarrow Aut\left( M^{\otimes n}_{\lambda}\right).$$
\end{definition}
Similarly to the case of the quantum group, one can define certain subspaces in the tensor product of Verma modules, by requiring a specified action of the generators $H$ and $E$.
\begin{definition}(Space of null vectors)\\
Let $m \in \N$. The space of null vectors in $M^{\otimes n}_{\lambda}$ corresponding to the weight $m$ is defined in the following manner: 
$$N[n \lambda-2m]:= \{v \in M^{\otimes n}_{\lambda} \mid Ev=0; Hv=(n\lambda-2m)v \}.$$
\end{definition}
\begin{definition}(Monodromy representation from $\omega_h$ and $M_{\lambda}$)\\
For any $m \in \N$, the monodromy of the KZ-connection $\omega_h$ induces a braid group representation on the spaces of null vectors:
$$\nu_h: B_n\rightarrow Aut(N[n \lambda-2m]).$$ 
\end{definition}
The next proposition gives a certain basis in the space of null vectors, which will play an important role in the identification between the three types of braid group representations.
\begin{proposition} Following \cite{Ito}, for $e \in E_{n,m}$, consider the vector:
$$w_e:=\sum^m_{i=0} (-1)^i\frac{1}{\lambda(\lambda-1)\cdot ...\cdot (\lambda-i)} \  F^i v_0 \otimes E^i \left(F^{e_1}v_0\otimes...\otimes F^{e_{n-1}}v_0 \right).$$
Then, for any $\lambda \in \C^* \setminus \N$, the following set describes a basis for the space of null vectors $N[n \lambda-2m]$:
$$\mathscr B_{N[n \lambda-2m]}:= \{w_e| e \in E_{n,m} \}.$$
\end{proposition}
\begin{remark}\label{pb}
For natural parameters $\lambda \in \N$, the set $\mathscr B_{N[n \lambda-2m]}$ is not even well defined.
\end{remark}

\begin{theorem}\cite{Ito},\cite{Koh}(Kohno's Theorem)\label{T:Kohno}\\
There exist an open dense set $U \subseteq \C^* \times \C^*$ such that for any $(h,\lambda) \in U$ there is the following identification between representations of the braid group:
$$ \left( \hat{W}_{n,m}^{\q,\lambda}, \mathscr B_{\hat{W}^{\q,\lambda}_{n,m}}  \right) \ \simeq_{\Theta_{\q,\lambda}} \ \left( \HH_{n,m}|_{\psi_{\q,\lambda}}, \mathscr B_{\HH_{n,m}}|_{{\psi}_{\q,\lambda}} \right) $$
More precisely, the quantum representation $\phh$ and Lawrence representation $l_{n,m}|_{{\psi}_{\q,\lambda}}$ are the same in the bases described above in Proposition \ref{P:hw} and Definition \ref{multif}). 
\end{theorem}
\subsection{Identifications with $q$ and $\lambda$ complex numbers}
We are interested in understanding the quantum representations with natural parameter $\lambda=N-1 \in \N$. This case does not belong anymore to the "generic parameters" discussion. In the sequel, we will study the relation between the previous braid group representations specialised with any parameters.

We will start with some general remarks about the group actions on modules and how they behave with respect to specialisations.
\begin{remark}
Let $R$ be a ring and $M$ an $R$-module with a fixed basis $\mathscr B$ of cardinal $d$. Consider a group action $G \curvearrowright M$ and a representation of $G$ using the basis $\mathscr B$:
$$\rho: G \rightarrow GL(d,R).$$
Suppose that $S$ is another ring and we have a specialisation of the coefficients, given by a ring morphism:
$$\psi: R \rightarrow S$$
Let us denote $M^{\psi}:=M \otimes_{R} S$ and $\mathscr B_{M^{\psi}}:=\mathscr B \otimes_{R}1 \in M^{\psi} $.
The specialisation $\psi$ leads to an induced group action $G\curvearrowright M^{\psi}$. Then, the following properties hold:

1) $\mathscr B_{M^\psi}$ is a basis for $M^{\psi}$.

2)Let $\rho^{\psi}: G\rightarrow GL(d,S)$ the representation of $G$ on $M^{\psi}$ coming from the induced action, in the basis $\mathscr B_{M^{\psi}}$. In this way, the two actions, before and after specialisation give the same action in the following sense:
$$\rho^{\psi}(g)=\rho(g)|_{\psi} \ \ \  \forall g \in G.$$
Here if $f: M \rightarrow M$, denote by $f|_{\psi}: M^{\psi}\rightarrow M^{\psi}$ the specialisation
$f|_{\psi}=f \otimes_{R}Id_{S}$. 
\end{remark} 
\begin{Comment}
We are interested in the case of non-generic complex parameters $(h,\lambda) \in \C^*\times \C$. We would like to to emphasise that quantum representation and Lawrence representation on one side and the KZ-monodromy representation on the other have different natures with respect to the complex parameters $(h,\lambda)$. Actually, both quantum representation and Lawrence representation
$$\left( \phh \ \ \ \ \ \ l_{n,m}|_{\psi_{\q,\lambda}}\right)$$
are coming from certain generic braid group representations 
$$\left( \phs  \ \ \ \ \ \ l_{n,m} \right) \ \ \ $$
and then they are specialised using the procedure from the previous remark, corresponding to the specialisations  
$$ \left( \eta_{\q,\lambda} \ \ \ \ \ \ \psi_{\q,\lambda} \right). \ \ \ $$
On contrary, in order to obtain the KZ-monodromy representation $\nu_h$, one has to fix the complex numbers $(\lambda, h)$ and do all the construction through this parameters. This is not globalised in a way that does not depend on the two specific complex numbers, in the sense that
we can't construct a representation over certain abstract variables such that the KZ-representation at the complex parameters is obtained from the abstract one by a specialisation, as in the previous remark.
\end{Comment}
\begin{problem}
The KZ representation does not come from a specialisation procedure and moreover we notice that we do not have a well defined action corresponding to a well defined basis for any complex parameters. From the remark \ref{pb}, for $\lambda \in \N$ a natural parameter $\mathscr B_{N[n \lambda-2m]}$ is not even a well defined set in $N[n \lambda-2m]$.
However, the isomorphism between the quantum and homological representations still works for any parameters, using a continuity argument. 
\end{problem}
\begin{theorem}\label{T:identif}
Let $(h,\lambda) \in \C^*\times \C $ any fixed parameters. Then the following braid group representations are isomorphic, using the following corresponding bases:
$$ \left( \hat{W}^{\q,\lambda}_{n,m}, \mathscr B_{\hat{W}^{\q,\lambda}_{n,m}} \right) \ \simeq_{\Theta_{\q,\lambda}} \ \left( \HH_{n,m}|_{\psi_{\q,\lambda}}, \mathscr B_{\HH_{n,m}|_{\psi_{\q,\lambda}}} \right)$$
\end{theorem}
\begin{proof}
1) In the proof of \ref{T:Kohno}, that is stated for any parameters in a dense open subset in $\C^* \times \C$, there are glued two identifications between representations of the braid group $B_n$. Basically, the relation between the quantum representation and the Lawrence representation is established by passing from both of them to the monodromy of the KZ-connection. More precisely, there are constructed two isomorphisms of braid group  representations:
\begin{equation}
\begin{cases}
f^{WN}_{\q,\lambda}:\HH_{n,m}|_{\psi_{\q,\lambda}}\rightarrow N[n\lambda-2m]\\
f^{NH}_{\q,\lambda}: N[n\lambda-2m]\rightarrow \hat{W}_{n,m}^{\q,\lambda} , \ \ \ \ \ \ \ \ \forall (h,\lambda) \in U.
\end{cases}
\end{equation}
These isomorphisms are proved using correspondences between the following bases:
$$\phh \ \ \ \ \ \ \ \ \ \ \nu _h \ \ \ \ \ \ l_{n,m}|_{\psi_{\q,\lambda}}$$
$$\hat{W}_{n,m}^{\q,\lambda} \simeq N[n\lambda-2m]\simeq\HH_{n,m}|_{\psi_{\q,\lambda}}$$
$$\mathscr B_{\hat{W}_{n,m}^{\q,\lambda}} \ \ \mathscr B_{N[n\lambda-2m]} \ \ \mathscr B_{\HH_{n,m}|_{\psi_{\q,\lambda}}}$$
$$\ \ \ \ \ \ \ \ \ \ \ \ \ \ \ \phi(v^s_{\iota(e)})|_{\eta_{\q,\lambda}} \ \ \leftarrow \ \ w_e \ \  \leftarrow \ \ [\tilde{\F}_e] \ \ \ \ \ \ \ \ \ \ \ \ \ \ \ \ \  $$
$$ \ \ \  \ \ f^{NH}_{\q,\lambda} \ \ \ \ \ \ \ f^{WN}_{\q,\lambda}$$
2)From this, Kohno proved that for any pair of parameters $(h, \lambda) \in U$:
$$\phh(\beta)= l_{n,m}|_{\psi_{\q,\lambda}}(\beta), \ \ \ \ \forall \beta\in B_n$$
3) Let us denote by $$\Theta_{\q, \lambda}: \HH_{n,m}|_{\psi_{\q,\lambda}}\rightarrow \hat{W}_{n,m}^{\q,\lambda}$$
$$\Theta_{\q, \lambda}([\tilde{\F}_e] )=\phi(v^s_e)|_{\eta_{\q,\lambda}}, \forall e \in E_{n,m}$$
This function is defined for all $(h, \lambda) \in \C \times \C$.  
We notice that, having in mind that they are defined directly on the bases, the functions $f^{WN}_{\q,\lambda}, f^{NH}_{\q,\lambda}$ are continuous with respect to the parameters $(h,\lambda)\in U$. This means that the function $\Theta_{\q, \lambda}$ is continuous with respect to the two complex parameters. Now, we will see what is happening in the case of non-generic parameters.\\
4) We are interested in the specialisation of the quantum representation. 





We know that $\mathscr B_{\hat{W}_{n,m}}$ is a basis for $\hat{W}_{n,m}$. The specialisation $\eta_{\q, \lambda}$ means to make a tensor product and consider everything in this situation. This will ensure that $\mathscr B_{\hat{W}_{n,m}}|_{\eta_{\q,\lambda}}$ will still describe a basis for the specialised module. We conclude that $$\mathscr B_{\hat{W}_{n,m}^{\q,\lambda}}:=\mathscr B_{\hat{W}_{n,m}}|_{\eta_{\q,\lambda}}$$ is a well defined basis of $\hat{W}_{n,m}^{\q,\lambda}$, for any $(h,\lambda)\in \C^* \times \C$.

5) Since the specialisation $\eta_{\q,\lambda}$ is well defined for any complex parameters $(h, \lambda) \in \C ^* \times \C$, all the coefficients from $\phs|_{\eta_{\q,\lambda}}$ are well defined complex numbers. In particular, the action $\phh$ in the basis $\mathscr B_{\hat{W}_{n,m}}|_{\eta_{\q,\lambda}}$ has all the coefficients well defined. 

6) Using the previous steps 4) and 5), we conclude that for any braid $\beta \in B_n$, the specialisation of the matrix obtained from the initial action $\phs$ onto $\hat{W}_{n,m}$ in the basis $\mathscr B_{\hat{W}_{n,m}}$, is actually the matrix of the specialised action $\phh$ in the specialised basis $\mathscr B_{\hat{W}_{n,m}^{\q,\lambda}}$:
$$\ \ \ \ \ \ \ \ \ \ \ \ \ \ \ \ \ \ \ \ \ \ \ \ \ \ \ \ \ \ \varphi^{\hat{W}}_{n,m}(\beta)|_{\eta_{\q,\lambda}}=\phh(\beta) , \ \ \ \ \ \ \ \ \forall (h,\lambda) \in (\C^*\times \C).$$

7) The set $\mathscr B_{\HH_{n,m}|_{\psi_{\q,\lambda}}}$ is well defined and describes a basis for $\HH_{n,m}|_{\psi_{\q,\lambda}}$ for any parameters $ (h, \lambda)\in \C^* \times \C$ (\ref{L:mult}).

8)This shows that for every $\beta \in B_n$, the specialisations of the matrices from the action on $\HH_{n,m}$ in the multifork basis, are actually the same as the matrices of the specialised Lawrence action, in the specialised multifork basis $\mathscr B_{\HH_{n,m}}|_{\psi_{\q,\lambda}}$:
$$\ \ \ \ \ \ \ \ \ \ \ \ \ \ \ \ \ \ \ \ \ \ \ \ \ \ \ \ \ \ l_{n,m}(\beta)|_{\psi_{\q, \lambda}}=l_{n,m}|_{\psi_{\q, \lambda}}(\beta), \ \ \ \ \ \ \ \  \forall (\q,\lambda) \in (\C^*\times \C)$$

Combining the remarks from 2), 3), 6), 8) we obtain that for any parameters $(q, \lambda) \in \C^* \times \C$, the following identification holds:
$$\phh(\beta)= l_{n,m}|_{\psi_{\q,\lambda}}(\beta), \ \ \ \ \forall \beta\in B_n$$
This concludes that the quantum representation and the Lawrence representation are isomorphic for any parameters.
\end{proof}
\subsection{Identifications with $q$ indeterminate}\label{oneindeter}
From the previous discussion, we know that the quantum representation $\hat{W}_{n,m}$ and the Lawrence representation $l_{n,m}$ are isomorphic after appropriate identifications of the coefficients, as long as we fix $(\q,\lambda)$ complex numbers. In the sequel, we will state a similar result, but for the case where we keep $q$ an indeterminate. \\
\begin{definition}
Let us fix $\lambda =N-1 \in N$ and $q$ an indeterminate.
Consider the specialisations of the coefficients:
$$\eta_{\lambda}: \Z[\qs^{\pm1}, s^{\pm1}]\rightarrow \Z[q^{\pm}]$$
$$\eta_{\lambda}(s)=\qs^{\lambda}$$
$$\psi_{\lambda}: \Z[x^{\pm},d^{\pm}]\rightarrow \Z[q^{\pm}]$$ 
$$ \ \ \ \ \ \ \ \ \ \ \ \ \ \ \ \ \ \ \ \ \ \ \psi_{\lambda}(x)= \qs^{2 \lambda}; \ \ \psi_{\lambda}(d)=-\qs^{-2}.$$
For $\q \in \C$, let $f_{\q}: \Z[\qs^{\pm}]\rightarrow \C$ be the evaluation function:
$$f_{\q}(\qs)=\q.$$
Then we notice that the specialisations with one respectively with two complex number are related by the relations:
\begin{equation}
\begin{cases}
\eta_{\q,\lambda}=f_{\q} \circ \eta_{\lambda}\\
\psi_{\q, \lambda}=f_{\q} \circ \psi_{\lambda}.
\end{cases}
\end{equation}
We remind the notations: 
\begin{equation}
\begin{cases}
\hat{W}_{n,m}^{\lambda}=\hat{W}_{n,m}\otimes_{\eta_{\lambda}}\Z[q^{\pm}]\\
$$\HH_{n,m}|_{\psi_{\lambda}}=\HH_{n,m}\otimes_{\psi_{\lambda}}\Z[q^{\pm}].
\end{cases}
\end{equation}
\end{definition}

\begin{theorem}\label{Th:K}  
The braid group representations with respect to the specialisation with one complex number are isomorphic over $\Z[q^{\pm}]$:
\begin{equation}
 \left( \hat{W}_{n,m}^{\lambda}, \mathscr B_{\hat{W}_{n,m}^{\lambda}} \right) \ \simeq_{\Theta_{\lambda}} \  \left( \HH_{n,m}|_{\psi_{\lambda}}, \mathscr B_{\HH_{n,m}}|_{\psi_{\lambda}} \right).
 \end{equation}
\end{theorem}
\begin{proof}

We will use Theorem \ref{T:identif}, and study a little more its proprieties.
Let $$\Theta_{\lambda}: \HH_{n,m}|_{\psi_{\lambda}}\rightarrow \hat{W}_{n,m}|_{\eta_{\lambda}} \ \ \ \ \ $$
$$\Theta_{\lambda}(\tilde{\F}_e )=\phi\left(v^s_{\iota(e)}\right)|_{\eta_{\lambda}}, \forall e \in E_{n,m}.$$
1) We notice that  $\mathscr B_{\hat{W}_{n,m}}|_{\eta_{\lambda}}$ is well defined and, as in the proof of Theorem \ref{T:identif}, it gives a basis in $\hat{W}_{n,m}^{\lambda}$.\\
2)Similarly $\mathscr B_{\HH_{n,m}}|_{\psi_{\lambda}}$ is a basis of $\HH_{n,m}|_{\psi_{\lambda}}$.\\
3) We have the following relations between specialisations:
\begin{equation*}
\begin{cases}
\hat{W}_{n,m}^{\q,\lambda}=\hat{W}_{n,m}^{\lambda}\otimes_{f_{\q}} \C\\
\HH_{n,m}|_{\psi_{\q,\lambda}}=\HH_{n,m}|_{\psi_{\lambda}}\otimes_{f_{\q}} \C.
\end{cases}
\end{equation*}
4) If we take $\beta \in B_n$, we notice that for any $\q \in \C$:
$$\phh(\beta)=f_{\q}\left(\phs(\beta)|_{\eta_{\lambda}}\right)$$
$$l_{n,m}(\beta)|_{\psi_{\q,\lambda}}=f_{\q} \left( l_{n,m}(\beta)|_{\psi_{\lambda}} \right)$$
(here, the sense is that $f_{\q}: M(d_{n,m}, \Z[\qs^{\pm}])\rightarrow M(d_{n,m}, \C)$, by specialising every entry of the matrix using the function $f_{\q}$ ).\\
5)This shows that $$ \phs(\beta) |_{\eta_{\lambda}}=l_{n,m}(\beta)|_{\psi_{\lambda}}, \forall \beta \in B_n$$
6) Using 1) and 2) it follows that these matrices corresponds to well defined actions in the two bases, and we conclude that: 
$$\phs(\beta)|_{\eta_{\lambda}}=\pv(\beta)$$
$$\ \ \ \ \ \ \ \ \ \ \ \ \ \ \ \mathscr B_{\hat{W}_{n,m}}|_{\eta_{\lambda}}$$
$$l_{n,m}(\beta)|_{\psi_{\lambda}}=l_{n,m}|_{\psi_{\lambda}}(\beta)$$
$$ \ \ \ \ \ \ \ \ \ \ \ \ \ \ \ \mathscr B_{\HH_{n,m}}|_{\psi_{\lambda}}$$
Combining the previous remarks from 5) and 6) we obtain that the braid group actions $\pv$ and $ l_{n,m}|_{\psi_{\lambda}}$ are isomorphic:
$$\pv(\beta)=l_{n,m}|_{\psi_{\lambda}}(\beta)$$
$$\mathscr B_{\hat{W}_{n,m}^{\lambda}} \ \ \ \ \mathscr B_{\HH_{n,m}}|_{\psi_{\lambda}}.$$ 

\end{proof}
\section{Homological model for the Coloured Jones Polynomial} \label{s}
In this section, we present a topological model for the Coloured Jones polynomials. We will start with an oriented knot and consider a braid that leads to the knot by braid closure. In the first part, we will study the Reshetikhin-Turaev functor on a link diagram that leads to the invariant, by separating it on three main levels. Secondly, for each of these levels we will construct step by step a homological counterpart in the Lawrence representation and its dual. Finally, we will show that the evaluation of the Reshetikhin-Turaev functor on the whole knot corresponds to the geometric intersection pairing between the homological counterparts.  

Let $N \in \mathbb N$ be the colour of the invariant that we want to study.
Let the parameter $\lambda= N-1$ and the specialisations defined as in Section \ref{D:spec}:
$$\left( \eta_{N-1} \ \ \ \ \ \ \ \  \psi_{ N-1}\right).$$
In the sequel, we will use the braid group actions corresponding to the quantum representations:
\begin{equation*}
\begin{cases}
\hat{W}_{n,m}^{N-1} \ \ \ \ \ \longleftrightarrow\ \ \ \ \ \pv \\
W_{n,m}^N \ \ \ \ \ \ \longleftrightarrow \ \ \ \ \ \varphi^{{W}^N}_{n,m}.
\end{cases}
\end{equation*}
We recall the change of coefficients $\alpha_{\lambda}$ from definition \ref{p}:
\begin{equation*}
\begin{cases}
\alpha_{\lambda}: \Z[x^{\pm},d^{\pm}] \rightarrow \Q(q)\\
\alpha_{\lambda}(x)= \qs^{2 \lambda}; \ \ \alpha_{\lambda}(d)=-\qs^{-2}.
\end{cases}
\end{equation*}
Using these notations, we will prove the following model.
\begin{theorem}({\small \bf Topological model for coloured Jones polynomials with specialised homology classes})\label{T:geom1}

Let $n \in \mathbb{N}$. Then, there exist two homology classes
 $$\tilde{\mathscr F}_n^N \in H_{2n, n(N-1)}|_{\alpha_{N-1}} \ \ \text{and} \ \ \  \tilde{\mathscr G}_n^N \in H^{\partial}_{2n, n(N-1)}|_{\alpha_{N-1}}$$ 
 such that for any knot $L$ with $L=\hat{\beta_n}$ for $\beta_n \in B_n$, the $N^{th}$ coloured Jones polynomial has the formula:
\begin{equation}\label{formula} 
J_N(L,q)= \frac{1}{[N]_q}q^{-(N-1)w(\beta_n)} < \left( \beta_{n} \otimes \mathbb I_n \right) \tilde{\mathscr{F}}_n^N , \tilde{\mathscr{G}}_n^N>|_{\alpha_{N-1}}. 
\end{equation}
\end{theorem}

\begin{proof}
Let $L$ be a knot and $\beta_{n} \in B_{n}$ such that $L=\hat{\beta}_n$ ( braid closure).
Consider the corresponding planar diagram for the knot L, using the closure of the braid $\beta_{n}$ which has the following three main levels:

1) the evaluation                              
$ \ \ \ \ \ \ \ \ \ \ \  \tikz[x=1mm,y=1mm,baseline=0.5ex]{\draw[<-] (3,0) .. controls (3,3) and (0,3) .. (0,0); \draw[<-] (6,0) .. controls (6,6) and (-3,6) .. (-3,0); \draw[draw=none, use as bounding box](0,0) rectangle (3.5,3);}
$

2) braid level                                    $ \ \ \ \ \ \ \ \ \ \ \ \ \ \ \ \beta_n \otimes \bar{\unit}_{n}$

3) the coevaluation                                 
$ \ \ \ \ \ \ \  \  \tikz[x=1mm,y=1mm,baseline=0.5ex]{\draw[<-] (0,3) .. controls (0,0) and (3,0) .. (3,3); \draw[<-] (-3,3) .. controls (-3,-3) and (6,-3) .. (6,3); \draw[draw=none, use as bounding box](-0.5,0) rectangle (3,3);}
$

Following the definition of the Reshetikhin-Turaev functor and proposition \ref{P:1}, the coloured
Jones polynomial, can by obtained in the following way:
\begin{equation}\label{eq:0}
J_N(L,q)=\frac{1}{[N]_q}q^{-(N-1)w(\beta_n)}\left( \ev^{\otimes n}_{V_N} \  \circ \ \F_{V_N} (\beta_{n} \otimes \bar{\mathbb I}_n ) \  \circ \ {\tcoev}^{\otimes n}_{V_N} \right) (1).
\end{equation}

\subsection{(Step I)-Coevaluation corresponding to the cups}\label{StepI}

\

Looking at the bottom part of the diagram, we notice that the algebraic properties of the quantum group actions on its representations imply that  the first morphism $\tcoev^{\otimes n}_{V_N}$ naturally arrives in particular subspace in $V_N^{\otimes n} \otimes \left(V_N^*\right)^{\otimes n}$. This subspace would correspond to a certain highest weight space if it would be inside the tensor power of the same representation. 
\begin{remark} One has the following:
\begin{equation}
 Im \left({\tcoev}^{\otimes n}_{V_N}\right) \subseteq Ker(E) \cap Ker (K-Id) \left(\subseteq V_N^{\otimes n} \otimes \left(V_N^*\right)^{\otimes n} \right).
 \end{equation}
 \end{remark}
\begin{proof}
From the fact that $${\tcoev}^{\otimes n}_{V_N}: \Z[q^{\pm}] \rightarrow V_N^{\otimes n} \otimes \left(V_N^*\right)^{\otimes n}$$ is an isomorphism $\U$-modules (definition \ref{D:1}), this will commute with the $E$ and $K$-actions. Since $\Z[q^{\pm 1}]$ is regarded as being the trivial representation, this shows that:
\begin{equation}
\begin{cases}
Im \left( {\tcoev}^{\otimes n} \right) \subseteq Ker (E\curvearrowright V_N^{\otimes n} \otimes \left(V_N^*\right)^{\otimes n} )\\
Im \left( {\tcoev}^{\otimes n} \right) \subseteq Ker \left( (K-Id) \curvearrowright V_N^{\otimes n} \otimes \left(V_N^*\right)^{\otimes n}  \right).
\end{cases}
\end{equation}
\end{proof}
Moreover, from this remark, one gets that for any vector $v \in Im \left( {\tcoev}^{\otimes n} _{V_N}\right)$:
$$K v= v=q^{0}v.$$
Having in mind the notion of weight spaces, we could write $q^0=q^{{\bf 2n}(N-1)- 2 {\bf (N-1)n}}$ and conclude that: 
$$K v= v=q^{{\bf 2n}(N-1)- 2 {\bf (N-1)n}}v,  \ \ \ \forall v\in Im \left( {\tcoev}^{\otimes n} _{V_N}\right).$$
\begin{notation}
Let us consider the vector:
$$w_n^N:={\tcoev}^{\otimes n} _{V_N}(1) \in V_N^{\otimes n} \otimes \left(V_N^*\right)^{\otimes n}.$$
More precisely, it has the form:
\begin{equation}\label{eq:2}
w_n^N=\sum_{i_1,...,i_{n}=0}^{N-1} v_{i_n}\otimes ... \otimes v_{i_{1}} \otimes (v_{i_{1}})^* \otimes ... \otimes (v_{i_n})^*.
\end{equation}
\end{notation}
\subsection{(Step II)-Arriving in a highest weight space}\label{StepII}

\

As we have seen in section \ref{ident}, quantum representations of the braid group encode homological information. This means that for the braid part of the diagram, we could create a bridge towards a homological action if all the  strands would have the same orientation. Therefore, we are interested to arrive at a formula for $J_N(L,q)$ which contains in the middle the action of $B_{2n}$ onto $V_N^{\otimes 2n}$. For the moment, corresponding to the braid group action, we have the Reshetikhin-Turaev functor as follows: $$\F_{V_N}( \beta_{n} \otimes \bar{\unit }_n)\in Aut \left(V_N^{\otimes n} \otimes (V_N^{*})^{\otimes n}\right).$$

The second idea is to insert additional isomorphisms at the first and the third level, which transform $(V_{N}^*)^{\otimes n}$ into $V_N^{\otimes n}$ and act non-trivially just on the last $n$-components of the tensor product, corresponding to the last $n$ strands of the diagram. Then, in the middle, we will have the Reshetikhin-Turaev functor evaluated on the braid, which is exactly the quantum representation:
 $$\varphi^{V_N}_{2n}\left( \beta_n \otimes \mathbb I_n \right)\in Aut \left(V_N^{\otimes 2n}\right).$$ 
In the sequel, we will make this precise. Following the first step, we notice that the bottom part of the diagram corresponding to the cups lead towards an analog of the highest weight space of weight $n(N-1)$. Now we show that we can arrive in the corresponding highest weight space inside $V_N^{\otimes 2n}$.
\begin{lemma} 
For any $n \in \N$, there exist an isomorphism of vector spaces $$\alpha_{n,N}:\left(V_N^{*}\right)^{\otimes n}\rightarrow V_N^{\otimes n}$$
 such that:
\begin{equation}\label{eq:1} 
\left( Id_{V_N}^{\otimes n} \otimes \alpha_{n,N}\right) (w_n^N) \in W^{N}_{2n,n(N-1)}.
\end{equation}
\end{lemma}
\begin{proof}
We search the function $\alpha_{n,N}$ of the form:
$$\alpha_{n,N}=f_1\otimes...\otimes f_n,$$
where $f_i:V_N^{*}\rightarrow V_N$ are isomorphisms of $\Z[q^{\pm 1}]-$ modules, for all $i \in \{1,...,n\}$. We prove the statement by induction on the number of strands. For $n \in \N$, let us consider the following statement $\bf{P(n)}$:

There exists a sequence of isomorphisms of $\Li$-modules $\{f_k \mid k\in \overline{1,n} \}$ such that:
\begin{equation}
\begin{cases}
f_k:V^{*}_N\rightarrow V_N\\
(Id_{V_N})^{\otimes n}\otimes \left( f_1\otimes...\otimes f_n \right) (w_n^N) \in W^{N}_{2n,n(N-1)} (\ref{eq:1}).
\end{cases}
\end{equation}
We remind that in the definition \ref{D:2} of highest weight spaces inside finite dimensional modules that we work with, there is a requirement that they come as specialisations from highest weight spaces over two variables:
$$W^N_{2n,n(N-1)}:=\hat{W}_{2n,n(N-1)}|_{\eta_{N-1}}\cap  \ V^{\otimes 2n}_{N}.$$
\begin{remark}
On the other hand, having in mind the type of isomorphisms that occur in the case of the usual version of the quantum group $U_q(sl(2))$ between the corresponding $N$-dimensional representation and its dual, it would be natural to search for a sequence of coefficients $\{c^i_k \in \Li \} $ such that the functions from $P(n)$ have the form: $$f_k(v^{\star}_i)=c^i_k \ v_{N-i-1}.
$$
\end{remark}
Combining this with the definition of highest weight spaces from above, shows that we would need an extra requirement, namely the existence of a sequence of lifts of the coefficients over two variables $\{ \tilde{c}^i_k \in \Li_s\}$ such that:
\begin{equation}\label{eq:-1}
\begin{cases}
\eta_{N-1}(\tilde{c}^i_k)=c^i_k\\ 
\exists \ \tilde{v}_n^N \in \hat{W}_{2n,n(N-1)} \text{ such that } \eta_{N-1}(\tilde{v}_n^N)=\left( Id^{\otimes n} \otimes \alpha_{n,N}\right) (w_n^N).
\end{cases}
\end{equation}

Having in mind this requirement and the definition of $w_n^N$, presented in equation \ref{eq:2}, we restate the induction hypothesis that we will prove:

$\bf P(n)$: There exist a sequence of coefficients in two variables:
$$\{   \tilde{c}^i_k \in \Li_s \mid k\in \overline{1,n}, i\in \overline{0,N-1}\}$$ such that if one consider the following vector in $\hat{V}^{\otimes 2n}$: 
\begin{equation}\label{eq:2'}
\tilde{v}_n^N:=\sum_{i_1,...,i_{n}=0}^{N-1} \tilde{c}^{i_1}_1 \cdot...\cdot \tilde{c}^{i_{n}}_n \ v_{i_n} \otimes ... \otimes v_{i_{1}} \otimes v_{N-i_{1}-1} \otimes ... \otimes v_{N-i_n-1}
\end{equation}
then it belongs to the highest weight space associated to the weight $n(N-1)$:
 $$\tilde{v}_n^N \in \hat{W}_{2n,n(N-1)}.$$
In other words, on requires the following conditions:
\begin{equation}\label{eq:3}
\begin{cases}
K \tilde{v}_n^N=s^{2n} q^{-2n(N-1)}\tilde{v}_n^N\\
E \tilde{v}_n^N=0.
\end{cases}
\end{equation}
We start with the discussion concerning the $K$-action from above.
\begin{remark}
Since the vector $\tilde{v}_n^N$ has the form given in equation \ref{eq:2'}, more specifically having the property that all its monomials have a constant sum of the indices, then for any choice of the coefficients, the following relation holds:
$$K \tilde{v}_n^N=\tilde{v}_n^N.$$
\end{remark}
In order to see this, we remind that the comultiplication of the quantum group acts as follows:
$$\Delta^{2n-1}(K)=K \otimes ... \otimes K.$$
Then for any indices $i_1,...,i_n \in \{1,...,N-1\}$, we have:
\begin{equation}
\begin{aligned}
K(v_{i_n} \otimes ... \otimes v_{i_{1}} \otimes v_{N-i_{1}-1} \otimes ... \otimes v_{N-i_n-1})= \ \ \ \ \ \ \ \ \ \ \ \ \ \ \ \ \ \ \ \ \ \ \ \ \\
=s^{2n} q^{-2\left(i_1 +...+i_n+(N-1-i_1)+...+(N-1-i_n)\right)} v_{i_n} \otimes ... \otimes v_{i_{1}} \otimes v_{N-i_{1}-1} \otimes ... \otimes v_{N-i_n-1}=\\
=s^{2n} q^{-2n(N-1)} v_{i_n} \otimes ... \otimes v_{i_{1}} \otimes v_{N-i_{1}-1} \otimes ... \otimes v_{N-i_n-1}. \ \ \ \ \ \ \ \ \ \ \ \ \ \ \ \ \ \ \ \ \ \ \
\end{aligned}
\end{equation}
This shows that $\forall i_1,...,i_n \in \{1,...,n\}$: 
$$ v_{i_n} \otimes ... \otimes v_{i_{1}} \otimes v_{N-i_{1}-1} \otimes ... \otimes v_{N-i_n-1} \in \hat{V}_{2n,n(N-1)}. $$
Therefore the first requirement from equation \ref{eq:3} is fulfilled: 
$$\tilde{v}_n^N \in \hat{V}_{2n,n(N-1)}.$$ 
This shows that the only condition that we require from $P(n)$ concerns the $E$-action. This action is done using the iterated comultiplication, which is given by:
$$\Delta^{2n-1}(E)=\sum_{j=1}^{2n}1^{j-1}\otimes E \otimes K^{2n-j}.$$
So, the requirement concerning the coefficients in two variables from equation \ref{eq:3} becomes:
\begin{equation}
\begin{aligned}
\sum_{j=1}^{2n}Id^{j-1}\otimes E \otimes K^{2n-j} \curvearrowright \ \ \ \ \ \ \ \ \ \ \ \ \ \ \ \ \ \ \ \ \ \ \ \ \ \ \ \ \ \ \ \ \  \\
\left( \sum_{i_1,...,i_{n}=0}^{N-1} \tilde{c}^{i_1}_1 \cdot...\cdot \tilde{c}^{i_{n}}_n \ v_{i_n} \otimes ... \otimes v_{i_{1}} \otimes v_{N-i_{1}-1} \otimes ... \otimes v_{N-i_n-1} \right)=0
\end{aligned}
\end{equation}
In the sequel we will show by induction that this condition can be fulfilled.

{\bf I) Base case n=1} 

We search for a sequence of coefficients $\{ \tilde{c}_1^i \in \Li_s \mid i \in \overline{0,N-1}\}$ such that:
\begin{equation}
\left( E \otimes K + 1 \otimes E\right) \curvearrowright \left( \sum_{i=0}^{N-1} \tilde{c}^{i}_1 \cdot  v_{i} \otimes v_{N-i-1} \right)=0.
\end{equation}
This is equivalent to:
\begin{equation}
 \sum_{i=0}^{N-2} sq^{-2(N-i-1)} \tilde{c}^{i}_1 \cdot  v_{i-1} \otimes v_{N-i-1} + \sum_{i=0}^{N-2} \tilde{c}^{i}_1 \cdot  v_{i} \otimes v_{N-i-2} =0.
\end{equation}
By changing the variable $i$ to $i+1$ in the first sum, the equation becomes:
\begin{equation}\label{eq:4}
 \sum_{i=0}^{N-2} \left( sq^{-2(N-i-2)} \tilde{c}^{i+1}_1+\tilde{c}^{i}_1\right) \cdot  v_{i} \otimes v_{N-i-2} =0.
\end{equation}
We consider the sequence of coefficients such that it satisfies the condition below, which implies equation \ref{eq:4}:
\begin{equation}
\begin{cases} 
 \tilde{c}^{i+1}_1=-s^{-1}q^{2(N-i-2)} \tilde{c}^{i}_1\\
  \tilde{c}^{0}_1=1.
  \end{cases}
\end{equation}
This concludes the verification step.

{\bf II) The inductive step: $\bf {P(n)\Rightarrow P(n+1)}$}\\ 

Let us suppose that $P(n)$ is true and aim to prove $P(n+1)$, by searching a sequence of coefficients $\{ \tilde{c}_{n+1}^{i_{n+1}} \in \Li_s \mid i_{n+1} \in \overline{0,N-1}\}$ such that:
\begin{equation}\label{eq:5}
\sum_{j=1}^{2n+2}Id^{j-1}\otimes E \otimes K^{2n+2-j} \curvearrowright \ \ \ \ \  
\end{equation}
\begin{equation*}
\left( \sum_{i_1,...,i_{n+1}=0}^{N-1} \tilde{c}^{i_1}_1 \cdot...\cdot \tilde{c}^{i_{n+1}}_{n+1} \ v_{i_{n+1}} \otimes ... \otimes v_{i_{1}} \otimes v_{N-i_{1}-1} \otimes ... \otimes v_{N-i_{n+1}-1} \right)=0
\end{equation*}

\begin{notation}
Let us consider the following notation:
$$u_{i_1,...,i_n}:=\ v_{i_{n}} \otimes ... \otimes v_{i_{1}} \otimes v_{N-i_{1}-1} \otimes ... \otimes v_{N-i_{n}-1} $$
$$\tilde{c}_{i_1,...,i_n}:=\tilde{c}^{i_1}_1 \cdot...\cdot \tilde{c}^{i_{n}}_{n}.$$
\end{notation}
Then, $P(n)$ is equivalent to:
\begin{equation}\label{eq:6}
\begin{aligned}
\Delta^{2n-1}(E)\left( \sum_{i_1,..,i_n=0}^{N-1} \tilde{c}_{i_1,...,i_n} u_{i_1,...,i_n} \right) =0.
\end{aligned}
\end{equation}
In the sequel, we will study the condition for $P(n+1)$ from equation \ref{eq:5}, by splitting the coevaluation of $E$ into two parts: one part which is associated to the first and the last strand, and another part which corresponds to all the other strands in the middle. This means that there is the first par, which has one term with $E$ on the first component and another term with $E$ on the last component and a second part, which contains all the other terms with $E$ in the middle. 
\begin{equation}
\begin{aligned}
\left( \left( E \otimes K^{\otimes 2n+1}+1^{\otimes 2n+1} \otimes E \right) +  1 \otimes \left( \sum_{j=1}^{2n}1^{j-1}\otimes E \otimes K^{2n-j} \right)\otimes K \right) \curvearrowright \\
\left( \sum_{i_1,...,i_{n+1}=0}^{N-1} \tilde{c}_{i_1,...,i_{n+1}} \left( v_{i_{n+1}}\otimes ... \otimes v_{i_1} \otimes v_{N-1-i_{1}} \otimes ... \otimes v_{N-1-i_{n+1}} \right) \right) =0.
\end{aligned}
\end{equation}
Separating the two parts of the sum, we get:
\begin{equation}\label{eq:7}
\begin{aligned}
\left(  E \otimes K^{\otimes 2n+1}+1^{\otimes 2n+1} \otimes E \right) \curvearrowright \ \ \ \ \ \ \ \ \ \ \ \ \ \ \ \ \ \ \ \ \ \ \ \ \ \ \ \ \ \ \ \\
\left( \sum_{i_{n+1}=0}^{N-1} \tilde{c}_{n+1}^{i_{n+1}} \sum_{i_1,...,i_n=0}^{N-1} \tilde{c}_{i_1,...,i_{n}} \left( v_{i_{n+1}}\otimes  u_{i_1,...,i_n} \otimes v_{N-1-i_{n+1}} \right) \right)+ \ \ \ \ \ \\
+  1 \otimes { \color{red} \left(  \Delta^{2n-1}(E) \right) } \otimes K  \curvearrowright \ \ \ \ \ \ \ \ \ \ \ \ \ \ \ \ \ \ \ \ \ \ \ \ \ \ \ \ \ \ \ \ \ \ \ \ \ \ \ \ \  \\
\left( \sum_{i_{n+1}=0}^{N-1} \tilde{c}_{n+1}^{i_{n+1}}  v_{i_{n+1}}\otimes { \color{red} \left(  \sum_{i_1,...,i_n=0}^{N-1}  \tilde{c}_{i_1,...,i_{n}}  u_{i_1,...,i_n} \right)} \otimes  v_{N-1-i_{n+1}} \right) =0.
\end{aligned}
\end{equation}
Using the induction hypothesis reformulated as in equation \ref{eq:6}, we conclude that the second sum from this formula vanishes. In other words, the conditions for the sequence $\{ \tilde{c}_{n+1}^{i_{n+1}}\}$ are given by:
\begin{equation}
\begin{aligned}
\left(  E \otimes K^{\otimes 2n+1}+1^{\otimes 2n+1} \otimes E \right) \curvearrowright \ \ \ \ \ \ \ \ \ \ \ \ \ \ \ \ \ \ \ \ \ \ \ \ \ \ \ \ \ \ \ \\
\left( \sum_{i_{n+1}=0}^{N-1} \tilde{c}_{n+1}^{i_{n+1}} \sum_{i_1,...,i_n=0}^{N-1} \tilde{c}_{i_1,...,i_{n}} \left( v_{i_{n+1}}\otimes  u_{i_1,...,i_n} \otimes v_{N-1-i_{n+1}} \right) \right)=0.
\end{aligned}
\end{equation}
This is equivalent to the following requirements:
\begin{equation}\label{eq:8}
\begin{aligned}
 \sum_{i_{n+1}=0}^{N-1} \tilde{c}_{n+1}^{i_{n+1}} s^{2n+1} q^{-2(n+1)(N-1)+2 i_{n+1}} \cdot\ \ \ \ \ \ \ \ \ \ \ \ \ \ \ \ \ \ \ \ \ \ \ \ \ \ \ \ \ \ \ \ \ \ \ \ \ \ \  \\
\cdot v_{i_{n+1}-1} \otimes \left( \sum_{i_1,...,i_n=0}^{N-1} \tilde{c}_{i_1,...,i_{n}} u_{i_1,...,i_n}  \right) \otimes v_{N-1-i_{n+1}} + \ \ \ \  \\
+ \sum_{i_{n+1}=0}^{N-1} \tilde{c}_{n+1}^{i_{n+1}} v_{i_{n+1}} \otimes \left( \sum_{i_1,...,i_n=0}^{N-1} \tilde{c}_{i_1,...,i_{n}} u_{i_1,...,i_n}  \right) \otimes v_{N-2-i_{n+1}} =0.
\end{aligned}
\end{equation}
By changing the parameter $i_{n+1}-1$ to $i_{n+1}$ and denoting it by $i$ in the first sum and then gluing the previous two terms together, we obtain the condition:
\begin{equation}\label{eq:9}
\begin{aligned}
 \sum_{i=0}^{N-2} \left( \tilde{c}_{n+1}^{i+1} s^{2n+1} q^{-2(n+1)(N-1)+2 (i+1)} + \tilde{c}_{n+1}^{i}\right) \ \ \ \ \ \ \ \ \ \ \ \ \ \ \ \ \ \ \ \ \ \ \ \ \ \ \ \ \ \ \ \ \ \ \ \ \\
 v_{i} \otimes \left( \sum_{i_1,...,i_n=0}^{N-1} \tilde{c}_{i_1,...,i_{n}} u_{i_1,...,i_n}  \right) \otimes v_{N-2-i} =0.
\end{aligned}
\end{equation}
Then, let us consider the coefficients for $f_{n+1}$ given by the following conditions, which imply equation \ref{eq:9}: 
\begin{equation}
\begin{cases} 
 \tilde{c}^{i+1}_{n+1}=-s^{-(2n+1)}q^{2(n+1)(N-1)-2(i+1)} \tilde{c}^{i}_{n+1}\\
  \tilde{c}^{0}_{n+1}=1.
  \end{cases}
\end{equation}
This concludes the induction step $P(n+1)$.
Then, the set of coefficients $\{c_{i}^{k} \in \Li\}$ is given by the evaluation of the previous sequence using the specialisation $\eta_{N-1}$:
\begin{equation}
\begin{cases} 
 c^{i+1}_{n+1}=-q^{(N-1)-2(i+1)} \tilde{c}^i_{n+1}\\
  c^{0}_{n+1}=1.
  \end{cases}
\end{equation}
We notice that these sequences of coefficients do not depend actually on the strand to which they correspond, namely $n+1$. This shows that all the functions are the same 
$$f=f_1=...=f_n:(V_N)^*\rightarrow V_N$$
$$f(v^{\star}_i)=(-1)^{i}q^{i(N-i)}v_{N-1-i}.$$
However, their lifts towards the generic highest weight space are different and depend on the strand. We obtain the normalising function $\alpha_{n,N}$ as in equation \ref{eq:1} given by:
\begin{equation} \label{L:isostar}
\alpha_{n,N}=f^{\otimes n}.
\end{equation}
\end{proof}
We conclude that using this normalising function $\alpha_{n,N}$, we arrive in the following highest weight space:
$$\left( Id_{V_N}^{\otimes n} \otimes \alpha_{n,N}\right) (w_n^N) \in W^{N}_{2n,n(N-1)}.$$
\subsection{(Step III)-The invariant seen through highest weight spaces}\label{StepIII}

\

So far, we have seen that if we use the extra normalising function from before, we arrive in a highest weight space. Pursuing this idea, we introduce $\alpha_{n,N}$ and its inverse to the first and third level of the corresponding diagram. The interesting part is that this procedure does not modify the invariant that we get. More precisely, we have the following:
\begin{lemma}\label{L:1} 
The coloured Jones polynomial has the following description:
\begin{equation}\label{eq:10} 
\begin{aligned}
J_N(L,q)=\frac{1}{[N]_q}q^{-(N-1)w(\beta_n)} \ev^{\otimes n}_{V_N} \  \circ \ { \color{blue} \left(Id_{V_N}^{\otimes n} \otimes \alpha_{n,N}^{-1}\right)} \circ \ \ \ \ \ \ \ \ \ \ \ \ \ \ \ \ \ \\ 
 \circ \ \F_{V_N} (\beta_{n} \cup \mathbb I_n ) \  \circ { \color{blue} \left(Id_{V_N}^{\otimes n} \otimes \alpha_{n,N} \right)} \circ \ {\tcoev}^{\otimes n}_{V_N}(1). \ \ 
\end{aligned}
\end{equation}
\end{lemma}
\begin{proof} This is the formula from equation \ref{eq:0}, with the two extra terms that contain the normalisation function. The key point is the fact that this function acts on the strands that correspond to $\mathbb I_n$ whereas we have identity which acts to the part that correspond to the braid $\beta_n$. So, the two functions $\alpha_{n,N}$ and its inverse cancel out one with the other. 
\begin{equation}
\begin{aligned}
{ \color{blue} \left(Id_{V_N}^{\otimes n} \otimes \alpha_{n,N}^{-1} \right)}\circ
\F_{V_N} (\beta_{n} \cup \mathbb I_n ) { \color{blue} \left(Id_{V_N}^{\otimes n} \otimes \alpha_{n,N} \right)}= \ \ \ \ \ \ \ \ \ \ \ \\
={ \color{blue} \left(Id_{V_N}^{\otimes n} \otimes \alpha_{n,N}^{-1} \right)} \circ \left( \F_{V_N} (\beta_{n})\otimes Id^{\otimes n}_{V_N} \right) \circ { \color{blue} \left(Id_{V_N}^{\otimes n} \otimes \alpha_{n,N} \right)}= \ \ \ \ \ \ \ \ \ \\
=\left( \F_{V_N} (\beta_{n}) \cup Id^{\otimes n}_{V^{*}_N} \right)=\F_{V_N} (\beta_{n} \cup \bar{\mathbb I}_n). \ \ \ \ \ \ \ \ \ \ \ \ \ \ \ \ \ \ \ \ \  \ \ \ \ \ \ \ \ \ \ \ \ \ \ 
\end{aligned}
\end{equation}
Then, replacing this relation in equation \ref{eq:0}, we conclude the lemma.
\end{proof}
Then, we conclude that we can obtain the invariant by composing the morphisms corresponding to the following diagram:

1) the evaluation                              
$ \ \ \ \ \ \ \ \ \ \ \ \ \ \ \ \ \ \ \ \ \  \tikz[x=1mm,y=1mm,baseline=0.5ex]{\draw[<-] (3,0) .. controls (3,3) and (0,3) .. (0,0); \draw[<-] (6,0) .. controls (6,6) and (-3,6) .. (-3,0); \draw[draw=none, use as bounding box](0,0) rectangle (3.5,3);}
$

\

2)normalising function   $ \ \ \ \ \ \ \ \ \ Id^{\otimes n}_{V_N}\otimes \alpha^{-1}_{n,N}$

\

3) braid level                                    $ \ \ \ \ \ \ \ \ \ \ \ \ \ \ \ \ \ \ \ \ \  \ \ \ \ \beta_n \otimes \unit_{n}$

\

4)normalising function $ \ \ \ \ \ \ \ \ \ Id^{\otimes n}_{V_N}\otimes \alpha_{n,N}$

\

5) the coevaluation                                 
$ \ \ \ \ \ \ \ \ \  \ \ \ \ \ \ \ \   \  \tikz[x=1mm,y=1mm,baseline=0.5ex]{\draw[<-] (0,3) .. controls (0,0) and (3,0) .. (3,3); \draw[<-] (-3,3) .. controls (-3,-3) and (6,-3) .. (6,3); \draw[draw=none, use as bounding box](-0.5,0) rectangle (3,3);}
$

\

\

\begin{adjustwidth}{-4cm}{-4cm} 
\newcommand{\tEv}{\stackrel{\longleftarrow}{\operatorname{Ev}}}
\begin{center}\label{diagram}
\begin{tikzpicture}   
[x=1.2 mm,y=1.8mm]

\node (f) [color=blue]  at (90,-5)   {$\mathbf{\mathscr{F}_{0}}$};
\node (f') [color=blue]  at (90,22)   {$(\beta_n\cup \unit_n)\mathbf{\mathscr{F}_{0}}$};
\node (f'') [color=blue]  at (80,35)   {$\mathbf{(\beta_n\cup \unit_n)\tilde{\mathscr{F}}^N_{n}}$};

\node (g) [color=teal]  at (105,35)   {$\mathbf{\tilde{\mathscr{G}}_{n}^{N}}$};

\node (Alex)  [color=red]             at (25,55)    {$Coloured \ Jones \ invariant$};
\node (c5') [color=red]  at (30,-28)   {\rotatebox{90}{$\in$} };
\node (c5'') [color=red]  at (30,48)   {\rotatebox{270}{$\in$}};

\node (c4') [color=red]  at (32,50)   {$J_N(L,\lambda)$};
\node (q) [color=teal]  at (78,55)   {$Topological \ intersection \ pairing$};
\node (q') [color=teal]  at (80,50)   {$<(\beta_n\cup \unit_n)\tilde{\mathscr{F}}^N_{n},\tilde{\mathscr{G}}_{n}^{N}>_{\alpha_{N-1}}$};

\node (c0')  [color=red]             at (30,-30)    {$1$};
\node (s1)               at (0,-20)    {\text{Step I} \ (\ref{StepI})};
\node (s2)               at (-10,5)    {\text{Step II} \ (\ref{StepII})};
\node (s3)               at (19,5)    {\color{green}\text{Step III} \ (\ref{StepIII})};
\node (s4)               at (44,25)    {\color{green}\text{Step IV} \ (\ref{StepIV})};
\node (s5)               at (44,5)    {\color{red}\text{Step V} \ (\ref{StepV})};
\node (s6)               at (75,-18)    {\color{blue}\text{Step VI} \ (\ref{StepVI})};
\node (s7)               at (105,27)    {\color{teal}\text{Step VII} \ (\ref{StepVII})};
\node (s8)               at (52,48)    {\color{teal}\text{Step VIII} \ (\ref{StepVIII})};

\node (b1)               at (0,0)    {$\mathbf{V}_N^{\otimes n}$};
\node (b2) [color=green] at (30,0)   {$\mathbf{W}^{N}_{2n,n(N-1)}$};
\node (b3) [color=red]   at (60,0)   {$\widehat{\mathbf{W}}^{N-1}_{2n,n(N-1)}$};
\node (b4) [color=blue]  at (90,0)   {$\mathbf{H}_{2n,n(N-1)}|_{\psi_{N-1}}$};
\node (t1)               at (0,20)   {$\mathbf{V}_N^{\otimes n}$};
\node (t2) [color=green] at (30,20)  {$\mathbf{W}^{N}_{2n,n(N-1)}$};
\node (t3) [color=red]   at (60,20)  {$\widehat{\mathbf{W}}^{N-1}_{2n,n(N-1)}$};
\node (t4) [color=blue]  at (90,20)  {$\mathbf{H}_{2n,n(N-1)}|_{\psi_{N-1}}$};
\node (t4') [color=blue]  at (80,31)  {$\mathbf{H}_{2n,n(N-1)}|_{\alpha_{N-1}}$};

\node      [color=teal]  at (103,31) {$\otimes \mathbf{H}_{2n,n(N-1)}^{\partial}|_{\alpha_{N-1}}$};
\node (cbottom)          at (30,-25) {$\Z[q^{\pm}]$};
\node (ctop)             at (30,45)  {$\Z[q^{\pm}]$};
\node (cq) [color=teal]  at (70,45) {$\Q(q)$};

\draw[->]             (cbottom) to node[left,yshift=-2mm,font=\small]{${1) \tCoev}^{\otimes n}_{V_N}$} (b1);
\draw[->,color=green] (cbottom) to node[right,font=\small]{$ 2) \   \textcolor{green}
{\tCoev^{\otimes n}_{V_N}} \ \textcolor{red} {3)}$}            (b2);
\draw[->]             (b1)      to node[left,font=\small]{$1) \varphi^{V_N}_{2n} (\beta_n \cup \mathbb I_n)$}                           (t1);
\draw[->,color=green] (b2)      to node[right,font=\small]{$2) \varphi^{W^N}_{2n,n(N-1)}(\beta_n \cup \mathbb I_n)$}                           (t2);
\draw[->,color=red]   (b3)      to node[right,font=\small]{$3) \hat{\varphi}^{\hat{W}^{N-1}}_{2n,n(N-1)}(\beta_n \cup \mathbb I_n) $}                           (t3);
\draw[->,color=blue]  (b4)      to node[right,font=\small]{$4) l_{2n,n(N-1}(\beta_n \cup \mathbb I_n)$}                           (t4);
\draw[->]             (t1)      to node[left,yshift=1mm,font=\small]{${1)\Ev}^{\otimes n}_{V_N}$}   (ctop);
\draw[->,color=green] (t2)      to node[right,font=\small]{${2)\Ev}^{\otimes n}_{V_N}$}              (ctop);
\draw[->,color=red]   (t3)      to node[right,yshift=3mm,font=\small]{$3){\Ev}^{\otimes n}_{\hat{V}_{\eta_{N-1}}}$}  (ctop);
\draw[->,color=teal]  [in=-10,out=130] (98,35)  to node[right,yshift=2mm, font=\small]{$\langle \, , \rangle_{\alpha_{N-1}}$}      (cq);
\draw[->,color=teal,dashed]  (ctop)  to node[right,font=\small]{}      (cq);
\draw[->]   [color=blue,dashed]          (cbottom) to node[left,yshift=-2mm,font=\small]{} (f);

\node at (15,0)                {$\supseteq$};
\node at (15,20)               {$\supseteq$};
\node at (45,0)  [color=green] {$\subseteq_{ \ \iota}$};
\node at (45,20) [color=green] {$\subseteq$};
\node at (75,0)  [color=blue]  {$\simeq_{\Theta_{N-1}}$};
\node at (75,20) [color=blue]  {$\simeq_{\Theta_{N-1}}$};

\node at (15,14) [font=\small]             { $\equiv$};
\node at (45,14) [font=\small,color=green] {$\equiv$};
\node at (75,14) [font=\small,color=red]   {$\equiv$};
\end{tikzpicture}
\end{center}
\end{adjustwidth}

In the following part, we aim to show that we can see the whole coloured Jones polynomial through the highest weight space from Step II. We introduce the following notation.
\begin{notation}(Normalising the evaluation and coevaluation)

Consider the following morphisms:
\begin{equation}
\begin{cases}
\Ev_{V_N}^{\otimes n}: V^{\otimes 2n}_N \rightarrow \Z[q^{\pm 1}]\\
\tCoev_{V_N}^{\otimes n}: \Z[q^{\pm 1}] \rightarrow  V^{\otimes 2n}_N
\end{cases}
\end{equation}
given by:
\begin{equation} \label{R:ev}
\begin{cases}
\Ev_{V_N}^{\otimes n}:= \ev_{V_N}^{\otimes n} \circ \ \left(Id_{V_N}^{\otimes n} \otimes \alpha_{n,N}^{-1}\right)\\
\tCoev_{V_N}^{\otimes n}:=\left( Id^{\otimes n} \otimes \alpha_{n,N} \right) \  \circ \tcoev_{V_N}^{\otimes n}. 
\end{cases} 
\end{equation}
 \end{notation}

\begin{corollary}
The formula for $J_N(L,q)$ presented in Lemma \ref{L:1} together with the definition of the quantum representation given in \ref{P:hww} lead to the following description:  
\begin{equation}\label{eq:11} 
J_N(L,q)= \frac{1}{[N]_q}q^{-(N-1)w(\beta)} \left( \Ev^{\otimes n}_{V_N} \  \circ \ \varphi^{V_N}_{2n} ( \beta_n \cup \mathbb I_n) \  \circ \ {\tCoev}^{\otimes n}_{V_N} \right) (1) \ \in \Z[q^{\pm 1}].
\end{equation}
\end{corollary}

\begin{remark}\label{R:3} 
Putting together the notation given in equation \ref{eq:2}, the properties of the normalising function $\alpha_{n,N}$ presented in Lemma \ref{eq:1} and the previous notation \ref{R:ev}, we obtain that:
$$\tCoev_{V_N}^{\otimes n}(1) \in W^N_{2n,n(N-1)}.$$
\end{remark}
On the other hand, we know that the action: $$B_{2n}\curvearrowright V^{\otimes 2n}_N$$ preserves the highest weight spaces, in particular preserves $W^N_{2n,n(N-1)}$. Using this invariance together with the previous remark and the formula from \ref{eq:11}, we notice that actually we can obtain $J_N(L,q)$ using just highest weight spaces, by composing the morphisms from the second column (\color{green} 2)\color{black}) from the diagram. We conclude the following.
\begin{proposition}(Coloured Jones invariant through the highest weight space)
\begin{equation}\label{eq:12}
J_N(L,q)=\frac{1}{[N]_q}q^{-(N-1)w(\beta_n)} \left( \Ev^{\otimes n}_{V_N} \  \circ \ \varphi^{W^N}_{2n,n(N-1)} ( \beta_{n} \cup \mathbb I_n) \  \circ \ {\tCoev}^{\otimes n}_{V_N} \right)(1). 
\end{equation}
\end{proposition}
\subsection{(Step IV)-Change of the highest weight space}\label{StepIV}

\

So far, we have seen that $J_N(L,q)$ is encoded by the action through highest weight spaces corresponding to the finite dimensional module $V_N$.
\begin{problem}
These highest weight spaces $W^N_{2n,m}$ inside $V^{\otimes 2n}_N$ do not have a geometric counterpart known yet. This is one of the reasons why there were not known topological interpretations for these invariants. On the other hand, Kohno's Theorem provide a geometric flavor for the bigger highest weight spaces $\hat{W}^{N-1}_{2n, n(N-1)}$, which live inside the power of the Verma module $\hat{V}_{N-1}$. 
\end{problem}
\begin{notation}
  Having this in mind, we look at the inclusion
$$\iota: W^{N}_{2n,n(N-1)}\hookrightarrow  \hat{W}^{N-1}_{2n, n(N-1)}$$ 
\end{notation}
In this part, we study the behaviour of this inclusion with respect to the braid group action. 
More precisely, we show that the quantum representation behaves well with respect to $\iota$.
 \begin{lemma}
The $B_n$-action on the highest weight spaces from the Verma module, lives invariant the highest weight spaces of the finite dimensional module: 
\begin{equation}
\varphi^{\hat{W}^{N-1}}_{n, m} \mid_{W^{N}_{n,m}}=\varphi^{{W}^N}_{n,m} \ \ \ \ \ \ \ \forall n,m \in \N.
\end{equation}
\end{lemma}
\begin{proof}
In relation \ref{R:R} it is given the action of the $R$-matrix on $\hat{V} \otimes \hat{V}$. We are interested in the specialisation of this action by the function $\eta_{N-1}$, which corresponds to the identifications:
 $$\lambda= N-1; \ \ \ \ \ s=q^\lambda.$$ 
It follows that the action $R \curvearrowright \hat{V}_{N-1} \otimes \hat{V}_{N-1}$ has the following form:
\begin{equation}
\begin{aligned}
\mathscr R(v_i\otimes v_j)= q^{-(N-1)(i+j)} \sum_{n=0}^{i} F_{i,j,n}(q) \cdot \ \ \ \ \ \ \ \ \ \ \ \ \ \ \ \ \ \ \ \ \ \ \ \ \ \ \ \ \ \ \ \ \ \ \ \ \ \ \ \\
\cdot  \prod_{k=0}^{n-1}(q^{(N-1)-k-j}-q^{-((N-1)-k-j)}) \ v_{j+n}\otimes v_{i-n}.
\end{aligned} 
\end{equation}

We prove that this action preserves $V_N \otimes V_N$ inside $\hat{V}_{N-1} \otimes \hat{V}_{N-1}$ (\ref{inclusion}). We show this by checking it on the basis $$\{v_i \otimes v_j \mid 0 \leq i,j\leq N-1 \}.$$ Let $0 \leq i,j\leq N-1 $.
We notice that in the above formula, all the indices of the second components decrease, and so the vectors $v_{i-n}$ will remain in $V_N$. 

For the first components, let us suppose that we pass over $V_N$ inside $\hat{V}_{N-1}$, in other words we have the term corresponding to the index $$j+n \geq N.$$
We prove that in this situation, the coefficient will vanish. We notice that for $k=N-1-j$, the corresponding term vanishes:
$$q^{(N-1)-k-j}-q^{-((N-1)-k-j)}=0.$$
Moreover, for $j+n \geq N$, it follows that $N-1-j \leq n-1$, so the term corresponding to $k=N-1-j$ will appear in the previous product. We conclude that the coefficient of the vector $v_{j+n} \otimes v_{i-n}$ vanishes.

Secondly, we have the action: $B_n \curvearrowright W^N_{n,m}\subseteq \hat{W}^{N-1}_{n,m}$. We show that each generator of the braid group $\sigma_i$ preserves $ W^N_{n,m}$. Let $w\in  W^N_{n,m}$. Using that $W^N_{n,m} \subseteq V^N_{n,m}$, it is possible to write: 
$$w= \sum_{e \in E^N_{n,m}} \alpha_{e} v_{e_1}\otimes v_{e_i}\otimes v_{e_{i+1}}\otimes...\otimes v_{e_n}.$$
We remind the action: $\sigma_i w= (Id^{\otimes(i-1)} \otimes \mathscr R \otimes Id^{\otimes(n-i-1)})w$.\\
From the first part, $\sigma_i w$ will modify just the components $i$ and $i+1$ of $w$, and the indexes corresponding to all vectors will remain strictly smaller than $N$. This shows that: 
\begin{equation}\label{eq:13}
 \sigma_i w \in V_N^{\otimes n}.
\end{equation}
Since the action of $B_n$ is an action of $\U$-modules, this commutes with the action of the generators from the quantum group $E$ and $K$, so it preserves the weights and the kernel of $E$. As a conclusion, using this remark and relation \ref{eq:13}, we obtain that: \[\sigma_i w \in \hat{W}^{N-1}_{n,m} \cap V_N^{\otimes n}=W^N_{n,m}.\qedhere\] 
\end{proof}
Up to this moment, we focused on the phenomena that occurs at the bottom part of the diagram, concerning the cups and the braid.  Now we will study the upper part containing the caps, which correspond to the algebraic evaluation.

For the case of the specialised Verma module $\hat{V}_{N-1}$,  we do not have a well defined coevaluation, since it is infinite dimensional. However, we are interested in defining an evaluation type map corresponding to this module. In order to do this, we use the finite dimensional submodule inside it 
$V_N\subseteq \hat{V}_{N-1}$, which has a well defined corresponding evaluation from equation \ref{E:DualityForCat} and define an evaluation type map on $\hat{V}_{N-1}$, supported on the submodule $V_N$.
Let us make this precise. 
 \begin{definition}(Normalised evaluation on the Verma module) 
 
 Consider $\Ev^{\otimes n}_{\hat{V}_{N-1}}: \hat{V}_{N-1}^{\otimes n} \rightarrow \Z[q^{\pm}] $ given by the expression:
\begin{equation}
\Ev^{\otimes n}_{\hat{V}_{N-1}}(v_{i_1}\otimes...\otimes v_{i_n})=
\begin{cases}
& \Ev^{\otimes n}_{V_N}(v_{i_1}\otimes...\otimes v_{i_n}),  \textit{if} \ \ i_1,...,i_n \leqslant N-1  \\
& 0, \ \ \ \ \ \ \ \textit{otherwise}.\\
\end{cases}
\end{equation}
and extended it by linearity. 
\end{definition}
\begin{remark}
This is an extension of the previous evaluation:
$$\Ev^{\otimes n}_{\hat{V}_{N-1}} |_{ V_N^{\otimes n}}=\Ev^{\otimes n}_{V_N}.$$
\end{remark}

\subsection{(Step V)-The invariant through highest weight spaces in the Verma module}\label{StepV}

\

In the following part, our strategy is to use the highest weight spaces from the Verma module and show that we can see the coloured Jones polynomial through them. 

So far, we have seen that we could construct the invariant following the second column of the diagram. Now, we will start with the normalised coevaluation which arrives in the highest weight spaces $W^N_{2n,n(N-1)}$ from $V_N$, then following the inclusion into $\hat{W}^{N-1}_{2n,n(N-1)}$. Now, we follow the braid group action on these bigger highest weight spaces and finally close with the evaluation $\Ev^{\otimes n}_{\hat{V}_{N-1}}$.
\begin{proposition}\label{eq:14}
 We conclude that the coloured Jones polynomial can be obtained through the highest weight spaces of weight $n(N-1)$ from the Verma module, following column (\color{red}3) \color{black}) from the diagram \ref{diagram}:
\begin{equation}
J_N(L,q)=\frac{1}{[N]_q}q^{-(N-1)w(\beta_n)} \left( \Ev_{\hat{V}_{N-1}}^{\otimes n} \  \circ \ \hat{\varphi}^{\hat{W}^{N-1}}_{2n,n(N-1)} (\beta_{n}\cup \unit_n)  \ \ \circ \iota \ \ \circ \ \tCoev^{\otimes n}_{V_N} \right)(1).
\end{equation}
\end{proposition}
\subsection{(Step VI)-Construction of the first homology class } \label{StepVI}

\

Now, we start the construction of the homology classes. The advantage of the bigger highest weight spaces $\hat{W}^{N-1}_{2n,n(N-1)}$ consists in the fact that they have an homological correspondent, given by the Lawrence representation $\HH_{2n,n(N-1)}$, due to Kohno's relation. In the sequel, we encode the cups of the diagram, corresponding to the normalised coevaluation using the Lawrence representation.

We consider the element corresponding to the image of $1$ through the normalised coevaluation seen inside the highest weight space $\hat{W}^{N-1}_{2n,n(N-1)}$. After that, we reverse it using Khono's function towards a topological class in the Lawrence representation, as follows.
\begin{definition}(The first homology class $ \mathscr F_0$, over $\Z[q^{\pm 1}]$)

Let us define the vector $v \in \hat{W}^{N-1}_{2n, n(N-1)}$ given by:
\begin{equation}\label{eq:15'}
v=\iota \circ {\tCoev}^{\otimes n}_{V_N}(1) \in \hat{W}^{N-1}_{2n, n(N-1)}.
\end{equation}
Then, using the isomorphism between quantum and homological representations from Theorem \ref{Th:K}, we consider the homology class in the Lawrence representation given by the following relation:
\begin{equation}\label{eq:15}  
\mathscr F_0:= \Theta_{N-1}^{-1}(v) \in \HH_{2n,n(N-1)}|_{\psi_{N-1}}.  
\end{equation}
\end{definition}
\begin{proposition}(Braid group action) \label{P:2}
The correspondence between the vector $v$ and the homology class $\mathscr F_0$ is preserved under the braid group action and we have the following relation:
\begin{equation}
\varphi^{\hat{W}^{N-1}}_{2n,n(N-1)}(\beta_{n}\cup \unit) (v)= \Theta_{N-1} \left( l_{2n,n(N-1)}|_{\psi_{N-1}}\left(\beta_{n}\cup \unit\right) (\mathscr F_0) \right).\label{R:iden}
\end{equation}
\end{proposition}
\begin{proof}
This comes from the identification between braid group actions from Theorem \ref{Th:K} and definition of the homology class $\mathscr F_0$ given in equation \ref{eq:15}.
\end{proof}

Up to this point, we found the first homology class $\mathscr F_0$ which encodes homologically the algebraic coevaluation ${\tCoev}^{\otimes n}_{V_N}$. Moreover, the braid part of the diagram, which is encoded by the braid group action on the quantum side, corresponds to the homological braid group action applied onto this class.
\subsection{(Step VII)-Construction of the second homology class }\label{StepVII}

\

 In the sequel, we are interested to find the second homology class $\tilde{\mathscr G}_n^N$, which will be a geometric counterpart for the part of the diagram containing the caps. The main ingredient that we use is the non-degenerate intersection form between Lawrence representation and its dual. More precisely, our aim is to encode homologically the evaluation 
 $$\Ev_{V_N}^{\otimes n}: W^N_{2n,n(N-1)}\rightarrow \Z[q^{\pm 1}].$$ Even if we are interested in this function, in practice we will use the extended evaluation 
 $$\Ev_{\hat{V}_{N-1}}^{\otimes n}: \hat{W}^{N-1}_{2n, n(N-1)}\rightarrow \Z[q^{\pm 1}],$$ which encodes the evaluation on the highest spaces of the finite dimensional module, but it is defined on the bigger highest weight space.
Then, we study this as an element of the dual space:
$$\Ev_{\hat{V}_{{N-1}}}^{\otimes n}\in \left(\hat{W}^{N-1}_{2n, n(N-1)}\right)^*.$$Via the identification between quantum and homological braid group representations, it corresponds to an element from the dual space:
$$\left( \HH_{2n,n(N-1)}|_{\psi_{N-1}}\right)^{\star}.$$ 
Our aim is to make the correspondence between this map and a geometric element from the dual Lawrence representation $\HH^{\partial}_{2n,n(N-1)}|_{\psi_{N-1}}$, using the intersection pairing. We will use the discussion from Section\ref{dualisation} concerning different flavours of specialisations of the Blanchfield pairing. 

At this point we notice a subtlety concerning the coefficients that we work with. So far, concerning the homological side, we needed the specialisation:  $$\psi_{N-1}: \Z[x^{\pm},d^{\pm}]\rightarrow \Z[q^{\pm}].$$
Corresponding to these coefficients, one has the non-degenerate intersection form:
$$< , >|_{\psi_{N-1}}: \HH_{2n,n(N-1)}|_{\psi_{N-1}}\otimes \HH^{\partial}_{2n,n(N-1)}|_{\psi_{N-1}}\rightarrow \Z[q^{\pm 1}].$$
However, the non-degenerancy does not ensures that the evaluation, seen as an element from the dual of this specialisation of the Lawrence representation can be seen as the intersection with a dual element. The issue comes from the fact that we are working over a ring and not a fleld. In the sequel, we change all these coefficients passing to the field of fractions, to arrive in the situation where the dualising procedure is more convenient. 
We remind the change of coefficients from definition \ref{p}:
\begin{equation}
\begin{cases}
\alpha_{N-1}: \Z[x^{\pm},d^{\pm}] \rightarrow \Q(q)\\\alpha_{\lambda}=\iota \circ \psi_{N-1}.
\end{cases}
\end{equation}
Then, using this specialisation, there is the following non-degenerate sesquilinear form:
\begin{equation}
< , >|_{\alpha_{N-1}}: \HH_{2n,n(N-1)}|_{\alpha_{N-1}}\otimes \HH^{\partial}_{2n,n(N-1)}|_{\alpha_{N-1}}\rightarrow \Q(q).
\end{equation}
\begin{definition}(The second homology class $\tilde{ \mathscr G}_n^N$)\label{G}

Let us define the following element:
$$\mathscr{G}_0:=\Ev_{\hat{V}_{{N-1}}}^{\otimes n} \circ \ \Theta_{N-1} \in Hom(\HH_{2n,n(N-1)}|_{\psi_{N-1}}, \Z[q^{\pm}]).$$
Then, let us consider the associated dual class given by Remark \ref{R:pair}:
$$\tilde{\mathscr G}_n^N \in \HH^{\partial}_{n,m}|_{\alpha_{N-1}}.$$
\end{definition}
\begin{remark}\label{GG}
This means that $\forall \mathscr E \in \HH_{2n,n(N-1)}|_{\alpha_{N-1}}$ we have:
$$ \mathscr{G}_0 \otimes Id_{\Q(q)} (\mathscr E)= <\mathscr E, \tilde{\mathscr G}^N_n>|_{\alpha_{N-1}}.$$
\end{remark}
Since for the construction of the second homology class we needed this change of coefficients, we consider the element which corresponds to the first homology class over this field as follows. 
\begin{definition}(The first homology class $\tilde{\mathscr F}_n^N$)

Let us consider the homology class corresponding to $\mathscr F_0$ over the field $\Q(q)$, using the map $p_{N-1}$ as in \ref{N:2} :
\begin{equation} \label{eq:23}
 \tilde{\mathscr F}_n^N:=p_{N-1}(\mathscr F_0)=\left( \mathscr F_0 \otimes_{\iota} 1 \right)\in \HH_{2n,n(N-1)}|_{\alpha_{N-1}}.
\end{equation}
\end{definition}
\subsection{(Step VIII)-Proof of the intersection formula }\label{StepVIII}

\

Now we will prove that the coloured Jones polynomial can be obtained from the intersection formula from equation \ref{formula}. Putting all the previous steps together, we obtain the following:
\begin{equation}\label{eq:19}
\begin{aligned}
J_N(L,q)=^{{\text Prop } \  \ref{eq:14}} \frac{1}{[N]_q}q^{-(N-1)w(\beta_n)} \ \ \ \ \ \ \ \ \ \ \ \ \ \  \ \ \ \ \ \ \ \ \ \ \ \ \ \ \ \ \ \ \ \ \ \ \ \ \ \ \ \ \ \ \ \ \ \ \\
\left( \Ev_{\hat{V}_{N-1}}^{\otimes n} \  \circ \ \hat{\varphi}^{\hat{W}^{N-1}}_{2n,n(N-1)} (\beta_{n}\cup \unit_n) \circ \iota \circ \ \tCoev^{\otimes n}_{V_N} (1)\right)= \\
\end{aligned}
\end{equation}
\begin{equation*}
\begin{aligned}
=^{{\text Equation } \  \ref{eq:15}} \frac{1}{[N]_q}q^{-(N-1)w(\beta_n)} \cdot \ \ \ \ \ \ \ \ \ \ \ \ \ \ \ \ \ \ \ \ \ \ \  \ \ \ \ \ \  \ \ \ \ \ \ \ \ \ \ \ \ \ \ \ \ \ \ \ \ \ \ \ \ \ \ \ \ \ \ \ \\
\left( \Ev_{\hat{V}_{N-1}}^{\otimes n} \circ {\color{red} {\Theta_{N-1}}}\circ{\color{green} {\Theta_{N-1}}^{-1}} \circ  \hat{\varphi}^{\hat{W}^{N-1}}_{2n,n(N-1)} (\beta_{n}\cup \unit_n) \circ {\color{green} {\Theta_{N-1}}}\circ {\color{red} {\Theta_{N-1}^{-1}}}(v)\right)= \\
=^{{\text Definition } \  \ref{eq:15}} \frac{1}{[N]_q}q^{-(N-1)w(\beta_n)} \cdot \ \ \ \ \ \ \ \ \ \ \ \ \ \ \ \ \ \ \ \ \ \ \ \ \ \ \ \ \ \ \ \ \ \ \ \ \ \ \ \ \ \ \ \ \ \ \ \ \ \ \ \ \ \ \ \ \ \ \ \\
\left( \Ev_{\hat{V}_{N-1}}^{\otimes n} \circ {\color{red} {\Theta_{N-1}}}\circ{\color{green} {\Theta_{N-1}}^{-1}} \circ  \hat{\varphi}^{\hat{W}^{N-1}}_{2n,n(N-1)} (\beta_{n}\cup \unit_n) \circ {\color{green} {\Theta_{N-1}}}\left( \mathscr F_0\right)\right)= \\
=^{{\text Proposition } \  \ref{P:2}} \frac{1}{[N]_q}q^{-(N-1)w(\beta_n)} \cdot \ \ \ \ \ \ \ \ \ \ \ \ \ \ \ \ \ \ \ \ \ \ \ \ \ \ \ \ \ \ \ \ \ \ \ \ \ \ \ \ \ \ \ \ \ \ \ \ \ \ \ \ \ \ \ \\
\left( \Ev_{\hat{V}_{N-1}}^{\otimes n} \circ {\color{red} {\Theta_{N-1}}} \circ  l_{2n,n(N-1)}|_{\psi_{N-1}} (\beta_{n}\cup \unit_n)\left( \mathscr F_0\right)\right)= \\
=^{ \text{Definition}  \ \ref{G}} \frac{1}{[N]_q}q^{-(N-1)w(\beta_n)} \cdot  \mathscr{G}_0 \left( l_{2n,n(N-1)}|_{\psi_{N-1}} (\beta_{n}\cup \unit_n)\ (\mathscr F_0) \right). \  \ \ \ \ \ \
\end{aligned}
\end{equation*}

\begin{remark}\label{GGG}
Let the morphism of changing the coefficients $p_{N-1}$, as in \ref{p}. Then $p_{N-1}$ commutes with the braid groups actions on the two specialisations of the Lawrence representations as below: 
\begin{center}
\begin{tikzpicture}
[x=1.2mm,y=1.4mm]

\node (b1)  [color=blue]             at (0,0)    {$\HH_{2n,n(N-1)}|_{\psi_{N-1}}$};
\node (b2) [color=cyan] at (30,0)   {$\HH_{2n,n(N-1)}|_{\alpha_{N-1}}$};
\node (t1)  [color=blue]             at (0,20)   {$\HH_{2n,n(N-1)}|_{\psi_{N-1}}$};
\node (t2) [color=cyan] at (30,20)  {$\HH_{2n,n(N-1)}|_{\alpha_{N-1}}$};

\draw[->,color=blue]             (b1)      to node[left,font=\small]{$l_{2n,n(N-1)}|_{\psi_{N-1}} $}                           (t1);
\draw[->,color=cyan] (b2)      to node[right,font=\small]{$l_{2n,n(N-1)}|_{\alpha_{N-1}} $}                           
(t2);
\draw[->,color=cyan]   (b1)      to node[right,font=\small]{$$}                           (b2);
\draw[->,color=cyan]  (t1)      to node[right,font=\small]{$$}                           (t2);

\
\node at (-8,-5)             {\color{blue}$\mathscr F_0$};
\node at (-8,16)             {\color{blue}$\beta_{} \mathscr F_0$};
\node at (37,-5)             {\color{cyan}$\tilde{\mathscr F}_n^N$};
\node at (37,16)             {\color{cyan}$\beta_{}\tilde{\mathscr F}_n^N$};

\node at (15,3)             {\color{blue}$p_{N-1}$};
\node at (15,18)             {\color{blue}$p_{N-1}$};
\node at (15,10)             {$\equiv$};
\end{tikzpicture}
\end{center}
\begin{equation}\label{eq:20}
p_{N-1}\left(l_{2n,n(N-1)}|_{\psi_{N-1}}\left(\beta\right) \left( \mathscr F_0\right) \right)=l_{2n,n(N-1)}|_{\alpha_{N-1}}\left(\beta\right) \left( \tilde{\mathscr F}_n^N \right), \ \ \forall \beta \in B_{2n}.
\end{equation}
\end{remark}
\begin{remark}
Following the properties of the geometric duals from \ref{C:duals}, we have:
\begin{equation}\label{eq:18}
\begin{aligned}
\mathscr G_0(\cdot)=(\mathscr G_0 \otimes Id_{\Q(q)})\circ p_{N-1}(\cdot) = \ \ \ \ \ \ \ \ \ \ \ \ \ \ \ \ \ \ \ \ \ \ \ \ \ \ \ \ \ \ \ \ \ \ \ \ \ \ \\
=\mathscr G \circ p_{N-1}(\cdot)= <p_{N-1}(\cdot), \tilde{\mathscr{G}}_n^N>|_{\alpha_{N-1}}, \ \ \ \forall \cdot \in \HH_{2n,n(N-1)}|_{\psi_{N-1}}.
\end{aligned}
\end{equation} 
\end{remark}
Following the previous two remarks concerning the change of coefficients over the field of fractions and the braid group actions, together with equation \ref{eq:19}, we obtain:
\begin{equation*}
\begin{aligned}
J_N(L,q)=^{\text{Equation }\ref{eq:18}}\frac{1}{[N]_q}q^{-(N-1)w(\beta_n)} \cdot \ \ \ \ \ \ \ \ \ \ \ \ \ \ \ \ \ \ \ \ \ \ \ \ \ \ \ \ \ \ \ \ \ \ \ \ \ \ \ \ \ \ \ \ \ \ \ \ \ \ \ \ \ \ \ \ \ \ \ \ \ \ \ \\
   <p_{N-1}\left( l_{2n,n(N-1)}|_{\psi_{N-1}} (\beta_{n}\cup \unit_n)\ (\mathscr F_0) \right), \tilde{\mathscr G}_n^N>|_{\alpha_{N-1}}= \ \ \ \ \ \ \ \  \ \ \ \ \ \ \ \ \  \ \ \ \ \ \ \ \ \\
=^{\text{ Equation } \ref{eq:20}} \frac{1}{[N]_q}q^{-(N-1)w(\beta_n)} \cdot < l_{2n,n(N-1)}|_{\alpha_{N-1}} (\beta_{n}\cup \unit_n)\ (\tilde{\mathscr F}_n^N), \tilde{\mathscr G}_n^N >|_{\alpha_{N-1}}. \ \ \ \ \ \ \ \ \ \  
\end{aligned}
\end{equation*}
Simplifying the notations in the last equality, we obtain the desired interpretation, which concludes the proof.
\end{proof}

\section{Topological model with non-specialised Homology classes } \label{ss}

For the homological model for $J_N(L,q)$, we have constructed homology classes
 $$\tilde{\mathscr F}_n^N \in H_{2n, n(N-1)}|_{\alpha_{N-1}} \ \ \  \text{and} \ \ \ \ \tilde{\mathscr G}_n^N \in H^{\partial}_{2n, n(N-1)}|_{\alpha_{N-1}}$$ which lead to the invariant through the topological intersection pairing. We notice that the colour $N$ appears in two places. The first part where it showes up concerns the number of points from the configuration space, since we are using configuration spaces in a fixed punctured disk (with $2n$ points removed), but with $ n(N-1)$ points. Secondly, the specialisation $\alpha_{N-1}$ depends on the color $N$.

 In this section, we will show that actually $ \tilde{\mathscr F}_n^N $ and $\tilde{\mathscr G}_n^N$ come from two homology classes that live in the unspecialised Lawrence representation. More specifically, they live in the homology of the configuration space $\tilde{C}_{2n. n(N-1)}$, over a larger  ring of coefficients. The feature of this model is that now the color shows up just in the number of points from the configuration space, but not anymore in the specialisation. More precisely, we will prove the statement from Theorem \ref{T:geom}, showing that we can construct two homology classes
$$ \mathscr F_n^N \in H_{2n, n(N-1)}|_{\gamma} \ \ \text{and} \ \ \  \mathscr G_n^N \in H^{\partial}_{2n, n(N-1)}|_{\gamma}$$ 
 so that the $N^{th}$coloured Jones polynomial has the formula:
\begin{equation}
J_N(L,q)= \frac{1}{[N]_q}q^{-(N-1)w(\beta_n)}< (\beta_{n}\cup \unit_n) \mathscr{F}_n^N , \mathscr{G}_n^N>|_{\delta_{N-1}}. 
\end{equation}
Here $\gamma$ is a specialisation over the field $\Q(s,q)$ which does not depend on $N$, whereas $\delta_{N-1}$ is a change of coefficients towards $\Q(q)$, defined using the colour. 

\subsection{Identifications with $q,s$ indeterminates}

In Section \ref{ident}, we have studied identifications between quantum representations and homological representations which are specialised with two complex generic parameters or using a natural number and an indeterminate. In this section, we will show that, if we increase a bit the ring of coordinates, the identification holds also over a ring with two indeterminates.

We recall that the quantum representation $\hat{W}$ is defined over $\Z[\qs^{\pm}, s^{\pm}]$. On the other hand, the Lawrence representation $\HH_{n,m}$ is defined over $\Z[x^{\pm},d^{\pm}]$. We have the following spacialisations:
\begin{equation}
\begin{cases}
\eta_{\lambda}: \Z[\qs^{\pm1}, s^{\pm1}]\rightarrow \Z[\qs^{\pm}] \ \ \ \ \ \ \ \ \\
\eta_{\lambda}(\qs)=\qs; \ \ \eta_{\lambda}(s)=\qs^{\lambda}.
\end{cases}
\end{equation}
\begin{equation}
\begin{cases}
\psi_{\lambda}: \Z[x^{\pm},d^{\pm}]\rightarrow \Z[\qs^{\pm}]\\
\psi_{\lambda}(x)= \qs^{2 \lambda}; \ \ \psi_{\lambda}(d)=-\qs^{-2}.
\end{cases}
\end{equation}

\begin{definition}\label{coeff}
Consider the specialisation which increases the ring of coefficients in the following manner:
\begin{equation}
\begin{cases}
\xi: \Z[x^{\pm},d^{\pm}]\rightarrow \Z[s^{\pm},\qs^{\pm}]\\
\xi(x)= s^2; \ \ \xi(d)=-\qs^{-2}.
\end{cases}
\end{equation}
For a certain reason which we will see later, we need to work over a field. Let us consider the inclusion map:
$$j: \Z[s^{\pm},\qs^{\pm}]\rightarrow \Q(s,\qs).$$
Then, let us define the extension of the initial ring by:
\begin{equation}
\begin{cases}
\gamma: \Z[x^{\pm},d^{\pm}]\rightarrow \Q(s,\qs)\\
\gamma=j \circ \xi.
\end{cases}
\end{equation}
In order to make the connection to the specialisation that we used for the model from Theorem \ref{T:geom1}, let us consider the change of coefficients:
\begin{equation}
\begin{cases}
\delta_{\lambda}:  \Q(s,\qs)\rightarrow  \Q(\qs)\\
\delta_{\lambda}(s)=\qs^{\lambda}.
\end{cases}
\end{equation}

\end{definition}

\begin{remark} \label{L:special}
This shows that we have the following commutative diagrams between the previous specialisations of the coefficients:
\begin{center}
\begin{tikzpicture}
[x=1.2mm,y=1.4mm]

\node (b1)               at (0,30)    {$\mathbb Z[x^{\pm1},d^{\pm1}]$};
\node (b2) [color=green] at (30,20)   {$\mathbb Z[s^{\pm1},q^{\pm1}]$};
\node (b3) [color=red]   at (30,8)   {$\mathbb Z[q^{\pm1}]$};
\node (b4) [color=blue]  at (30,-5)   {$\mathbb Q(q)$};
\node (b5) [color=orange]  at (60,30)   {$\mathbb Q(s,q)$};

\draw[->,color=green]   (b1)      to node[right,yshift=2mm,font=\large]{$\xi$}                           (b2);
\draw[->,color=red]             (b2)      to node[left,font=\large]{$\eta_{\lambda} $}   (b3);
\draw[->,color=teal]             (b3)      to node[left,font=\large]{$\iota$}   (b4);
\draw[->,color=red]             (b1)      to node[left,font=\large]{$\psi_{\lambda}$}   (b3);

\draw[->,color=blue]             (b1)      to node[left,font=\large]{$\alpha_{\lambda}$}   (b4);
\draw[->,color=orange]   (b1)      to node[right,xshift=-2mm,yshift=3mm,font=\large]{$\gamma$}                        (b5);
\draw[->,color=orange]   (b2)      to node[right,yshift=2mm,xshift=-5mm,font=\large]{$j$}                        (b5);
\draw[->,color=blue]   (b5)      to node[right,font=\large]{$\delta_{\lambda}$}                        (b4);


\end{tikzpicture}
\end{center}

\end{remark}

Using a similar argument as the one that we discussed in \ref{oneindeter}, one can concludes that the identification between quantum and homological representations works over a ring in two indeterminates. This was also briefly discussed in \cite{Ito}.
\begin{theorem} \label{T:twovar}
The braid group representations over $\Z[s^{\pm}, q^{\pm}]$ are isomorphic:
\begin{equation}
\left(\hat{W}_{n,m}, \mathscr B_{\hat{W}_{n,m}} \right) \simeq_{\Theta} \left(\HH_{n,m}|_{\xi},\mathscr B_{\HH_{n,m}}|_{\xi}\right)
\end{equation}
\end{theorem}

\subsection{Lift of the homology classes $\tilde{\mathscr {F}}_n^N$ and $ \tilde{\mathscr {G}}_n^N$}

\

Having in mind this identification between the braid group actions, a natural question would be to lift the homology classes constructed in Theorem \ref{T:geom1}, which live a priori in the homology groups specialised by $\alpha_{N-1}$, towards two elements belonging to the Lawrence representation specialised over two variables, using the specialisation $\xi$. However, since in our arguments we need to work over a field in order to be able to interpret the non-degeneracy of the Blanchfield pairing by dual elements, we will use the specialisation $\gamma$. We will prove the following lifting property.

\begin{lemma}\label{L:lift}
There exist two homology classes 
 $$\mathscr F_n^N \in H_{2n, n(N-1)}|_{\gamma}  \ \ \ \text{and} \ \ \ \mathscr G_n^N \in H^{\partial}_{2n, n(N-1)}|_{\gamma}$$
 such that under the specialisation $\delta_{N-1}$ one has:
 \begin{equation}
 \begin{cases}
 \mathscr F_n^N |_{\delta_{N-1}}=\tilde{\mathscr F}_n^N\\
 \ \mathscr G_n^N |_{\delta_{N-1}}=\tilde{\mathscr G}_n^N.
\end{cases}
\end{equation} 
\end{lemma}
\begin{proof}

1) We start with the definition of the first homology class and we aim to lift it over two variables. Following the discussion from  Step \ref{StepII}, the normalising function can be lifted over two variables. From the induction procedure that we used in \ref{eq:-1}, we know that there exist $\tilde{v}_n^N\in \hat{W}_{2n,n(N-1)}$ such that:
\begin{equation}
\eta_{N-1}(\tilde{v}_n^N)=\left( Id^{\otimes n} \otimes \alpha_{n,N}\right)\circ {\tcoev}^{\otimes n}_{V_N}(1).
\end{equation}
Following the definition of the normalising coevaluation, this means that we have:
\begin{equation}\label{eq:21} 
\eta_{N-1}(\tilde{v}_n^N)={\tCoev}^{\otimes n}_{V_N}(1).
\end{equation}
We remark that also in  \ref{P:bqu}, the function $\Theta$ is defined over the ring with two parameters and we have:
\begin{equation}\label{eq:22}
\Theta|_{\eta_{N-1}}=\Theta_{N-1}.
\end{equation}
\begin{definition}
Using the isomorphism from \ref{T:twovar}, let us consider the class:
$$\mathscr F^N_{0,n}:= \Theta^{-1}(v_n^N) \in \HH_{2n,n(N-1)}|_{\xi}.$$     
\end{definition}
\begin{remark}
Following the properties from equation \ref{eq:21} and equation \ref{eq:22}, and the construction of the first homology class from relation \ref{eq:15}, we obtain:
\begin{equation}\label{eq:24}
\mathscr F^N_{0,n}|_{\eta_{N-1}}=\mathscr F_0.
\end{equation}
\end{remark}
\begin{definition}(Lift of the first homology class) Having in mind the construction of the first homology class over a field, and relation \ref{eq:24}, let us consider:
\begin{equation}\label{eq:25}
\mathscr F^N_{n}:= \mathscr F^N_{0,n}|_{\j} \in \HH_{2n,n(N-1)}|_{\gamma}.
\end{equation}     
\end{definition}
\begin{remark}
Putting together the definition of the first class from equation \ref{eq:23} and the globalised class from equation \ref{eq:25}, together with relation \ref{eq:24} and the commutativity of the specialisations:
\begin{equation}\label{eq:27}
\delta_{N-1} \circ j=\iota \circ \eta_{N-1},
\end{equation}
 we conclude the first relation from the statement: 
\begin{equation}
\mathscr F_n^N |_{\delta_{N-1}}=\tilde{\mathscr F}_n^N.
\end{equation}
 \end{remark}

2) On the other hand, the action of the quantum group, especially of the generator $K$ can be seen over $\Z[q^{\pm1},s^{\pm 1}]$ and the evaluation from relation $\ref{E:DualityForCat}$, that corresponds to the caps from the diagram of the knot, can be defined over the ring $\Z[s^{\pm},q^{\pm}]$ in two parameters as well. Since we have seen that the coefficients of $\alpha_{n,N}$ can be lifted naturally over two variables, we conclude that we can lift the normalised evaluation as follows.
\begin{remark}
There exists a normalised evaluation over two variables:
\begin{equation}
\Ev_{\hat{V}}^{\otimes n}: \hat{W}_{2n, n(N-1)}\rightarrow \Z[q^{\pm 1},s^{\pm}]
\end{equation}
which specialises to the normalised evaluation:
\begin{equation}
\Ev_{\hat{V}}^{\otimes n}|_{\eta_{N-1}}=\Ev_{\hat{V}_{N-1}}^{\otimes n}.
\end{equation}
\end{remark}
\begin{definition}
Using this evaluation and the Kohno's function over two variables, let us consider the elements:
\begin{equation}
\begin{cases}
\mathscr{G}_{0,n}:=\Ev_{\hat{V}_{}}^{\otimes n} \circ \ \Theta \in Hom(\HH_{2n,n(N-1)}|_{\xi}, \Z[q^{\pm},s^{\pm}])\\
\mathscr{G}_n:=\mathscr{G}_{0,n} \otimes Id_{\Q(q,s)}\in Hom\left(\HH_{2n,n(N-1)}|_{\gamma}, \Q(q,s) \right).
\end{cases}
\end{equation}
\end{definition}
Secondly, the Blanchfield pairing will remain non-degenerate when we specialise the coefficients using the function $\gamma$:
\begin{equation}
< , >|_{\gamma}: \HH_{2n,n(N-1)}|_{\gamma}\otimes \HH^{\partial}_{2n,n(N-1)}|_{\gamma}\rightarrow \Q(q,s).
\end{equation}
\begin{definition}(Globalisation of the second homology class)
Dualising the globalised evaluation $\mathscr{G}_n$ using the non-degenerate pairing $< , >|_{\gamma}$ we get a homology class
$$\mathscr G_n^N \in H^{\partial}_{2n, n(N-1)}|_{\gamma}$$ such that:
\begin{equation}\label{eq:26} 
\mathscr{G}_n (\cdot)= <\cdot, \mathscr G^N_n>|_{\gamma}.
\end{equation}
\end{definition}
Following the construction from definition \ref{G}, the definition of the globalised class from relation \ref{eq:26} and the commutativity property related to the rings of coefficients from \ref{eq:27}, we obtain the second specialisation property from the statement:
$$ \ \mathscr G_n^N |_{\delta_{N-1}}=\tilde{\mathscr G}_n^N.$$
 \end{proof}
Moreover, using that the braid group action commutes with the specialisation of the coefficients, the intersection pairings are related one with the other as follows:
\begin{equation}
<(\beta_n \cup \unit_n) \tilde{\mathscr F}^N_n, \tilde{\mathscr G}^N_n>|_{\alpha_{N-1}}= <(\beta_n \cup \unit_n) \mathscr F^N_n, \mathscr G^N_n>|_{\delta_{N-1}}.
\end{equation}

Following the homological model from Theorem \ref{T:geom1} and the result concerning the lift of the homology classes from Lemma \ref{L:lift}, we conclude the topological model for $J_N(L,q)$, as it is presented in relation \ref{T:geom} and conclude the proof of the main Theorem \ref{T:geom}.

\
\
\url{https://www.maths.ox.ac.uk/people/cristina.palmer-anghel}  

\end{document}